\documentclass[english,twoside]{article}

\title{
Stability and instability in scalar balance laws:\\
fronts and periodic waves}
\author{V. Duch\^{e}ne\thanks{Univ Rennes, CNRS, IRMAR - UMR 6625, F-35000 Rennes, France. \url{vincent.duchene@univ-rennes1.fr}} \and L.M. Rodrigues\thanks{Univ Rennes \& IUF, CNRS, IRMAR - UMR 6625, F-35000 Rennes, France. \url{luis-miguel.rodrigues@univ-rennes1.fr}.
}}
\date{\today}

\usepackage{babel}        
\usepackage[utf8]{inputenc} 
\usepackage[T1]{fontenc}   
\usepackage{amsmath, amsthm,amsfonts,amssymb}
\usepackage{float,subcaption}
\usepackage[pdftex]{graphicx}    
\usepackage{fancyhdr}        
\usepackage{url,hyperref}    
\usepackage{mathtools} 
\usepackage{titling}  

\usepackage{color}

\numberwithin{equation}{section}

\makeatletter
\let\Title\@title
\let\Author\@author
\makeatother
 
\fancypagestyle{mystyle}{%
	\fancyfoot{}
	\fancyhead[CO]{\Title} 
	\fancyhead[CE]{\scriptsize V. Duch\^{e}ne and L.M. Rodrigues} 
	\fancyhead[RE]{\scriptsize\today}  
	\fancyhead[LO]{}  
	\fancyhead[LE,RO]{\scriptsize\thepage}
	}
\fancypagestyle{plain}{\fancyhead{} } 

\newtheorem{Theorem}{Theorem}[section]
\newtheorem{Definition}[Theorem]{Definition}
\newtheorem{Proposition}[Theorem]{Proposition}
\newtheorem{Corollary}[Theorem]{Corollary}
\newtheorem{Lemma}[Theorem]{Lemma}
\newtheorem{Assumption}[Theorem]{Assumption}
\newtheorem{Remark}[Theorem]{Remark}

\newcommand{\eps}{\varepsilon}

\newcommand\uU{{\underline U}}
\newcommand\uu{{\underline u}}
\newcommand\cS{{\mathcal S}}
\newcommand\cL{{\mathcal L}}
\newcommand\cC{{\mathcal C}}

\newcommand\tu{{\widetilde u}}

\newcommand\tw{{\widetilde w}}

\newcommand\tchi{{\widetilde \chi}}
\newcommand\tA{\widetilde{A}}
\newcommand\tB{\widetilde{B}}

\newcommand{\RR}{\mathbb{R}}
\newcommand{\CC}{\mathbb{C}}
\newcommand{\NN}{\mathbb{N}}
\newcommand{\ZZ}{\mathbb{Z}}
\newcommand\cD{{\mathcal D}}
\newcommand\cN{{\mathcal N}}
\newcommand\cX{{\mathcal X}}
\newcommand\cY{{\mathcal Y}}

\newcommand\uI{{\underline I}}
\newcommand{\upsi}{x}

\DeclareMathOperator\dd{d}
\renewcommand\d{\partial}
\newcommand{\Id}{{\rm Id}}
\DeclareMathOperator\sign{sgn}

\DeclareMathOperator\supp{supp}

\DeclareMathOperator\Span{span}
\DeclareMathOperator{\Div}{div}
\newcommand{\fpar}{f_{\shortparallel}}
\newcommand{\fperp}{f_{\perp}}
\newcommand{\Fperp}{F_{\perp}}
\newcommand{\Phiperp}{\Phi_{\perp}}
\newcommand{\sigpar}{\sigma_{\shortparallel}}

\newcommand{\id}[1]{\left\vert_{_{#1}}\right.}
\newcommand{\eqdef}{\stackrel{\rm def}{=}}
\DeclarePairedDelimiter\abs{\lvert}{\rvert}
\DeclarePairedDelimiter\Norm{\big\lVert}{\big\rVert}
\newcommand{\dsp}{\displaystyle}

\begin{document}
\maketitle

\begin{abstract}
We complete a full classification of non-degenerate traveling waves of scalar balance laws from the point of view of spectral and nonlinear stability/instability under (piecewise) smooth perturbations. A striking feature of our analysis is the elucidation of the prominent role of characteristic points in the determination of both the spectra of the linearized operators and the phase dynamics involved in the nonlinear large-time evolution. For a generic class of equations an upshot of our analysis is a dramatic reduction from a tremendously wide variety of entropic traveling waves to a relatively small range of \emph{stable} entropic traveling waves.

\vspace{1em}

\vspace{1em}
\end{abstract}

\tableofcontents

\section*{Introduction}

In the present contribution, we continue our study, initiated in~\cite{DR1}, of the large-time asymptotic behavior of solutions to first order scalar hyperbolic balance laws, that is, of the form
\begin{equation}\label{eq-u}
\d_t u+\d_x \big(f(u)\big)=g(u), \qquad u:\RR^+\times\RR\to\RR\,,
\end{equation}
in neighborhoods of traveling waves.

Let us first recall, that, prior to~\cite{DR1}, under rather natural assumptions on $f$ and $g$ --- including the strict convexity of $f$ and the strict dissipativity at infinity of $g$ ---, it was already known that starting from an $L^\infty$ initial datum that is either spatially periodic or is constant near $-\infty$ and near $+\infty$, the large-time dynamics is well captured in $L^\infty$ topology by piecing together traveling waves (constants, fronts or periodic waves). Indeed on one hand it is proved in~\cite{FanHale93,Lyberopoulos94,Sinestrari95,Sinestrari97} that in a spatially periodic setting, every solution approaches asymptotically either a periodic (necessarily discontinuous) traveling wave, or a constant equilibrium. Moreover, periodic traveling waves are actually unstable and the rate of convergence is exponential in the latter case whereas for the untypical solutions that do converge to a periodic wave convergence rates may be arbitrarily slow. On the other hand it is proved in ~\cite{Sinestrari96,MasciaSinestrari97} that, starting from data with essentially compact support on the whole line, the large-time asymptotics may a priori involve several blocks of different kinds (constants, fronts or periodics). Yet again the scenario generating periodic blocks is also non generic and unstable. Note that at the level of regularity considered there the strict convexity assumption on $f$ plays a key role as it impacts the structure of possible discontinuities. The few contributions relaxing the convexity assumption add severe restrictions on $g$ or on the initial data, for instance linearity of $g$ in~\cite{Lyberopoulos92}, Riemann initial data in~\cite{Sinestrari97a,Mascia00} and monotonicity of the initial data in~\cite{Mascia98}. At a technical level, one key ingredient in the proofs of the aforementioned series of investigations are generalized characteristics of Dafermos~\cite{Dafermos77}. They provide a formulation of the equation that is well suited to comparison principles thus to asymptotics in $L^\infty$ topology. 

Our goal was to complete the foregoing studies by providing stability/instability results in strong topologies measuring the size of piecewise smooth functions, but assuming no localization on perturbations and relaxing also convexity assumptions. In the companion paper~\cite{DR1} we have already identified what are spectral stability conditions for traveling waves that are either constant states or Riemann shocks, and proved a (dynamical) nonlinear stability result for (non-degenerate) spectrally stable ones. Here we complete our program by providing 
\begin{itemize}
\item a complete classification of non-degenerate traveling waves according to their spectral stability;
\item proofs that for those waves spectral instability (resp. stability) yields dynamical nonlinear instability (resp. asymptotic stability).
\end{itemize}
The notions of non-degeneracy, stability and instability we use here are precisely introduced in Section~\ref{S.structure}. Yet we would like to emphasize already at this stage that the nonlinear instabilities we prove are dramatically strong, they hold even if one is allowed to fully re-synchronize before comparing shapes of solutions and to lose arbitrarily much on Sobolev scales between topologies used to measure initial data and resulting solutions. 

The upshot of our classification is that though Equation~\eqref{eq-u} may possess a tremendously huge number of (non-degenerate) traveling-wave solutions\footnote{In the present introduction, \emph{solution} always means entropy-admissible solution.} only very few of them are stable. To illustrate how strong the reduction is, let us momentarily focus on the case where $f''$ does not vanish and zeros of $g$ are simple, and consider a continuous stable front, or in other words, a stable traveling wave solution with continuous profile $\uU$ connecting two distinct endstates. Then by piecing together parts of the profile $\uU$ according to the Rankine-Hugoniot condition one may obtain non-degenerate piecewise smooth traveling waves with the same speed as the original front, realizing as a sequence of discontinuity amplitudes any prescribed sequence\footnote{Given as an element of $(0,A)^I$, where $I$ is either $[\![ 0,m]\!]$ for some $m\in\NN$, $\NN$, $-\NN$ or $\ZZ$, and $A$ is the maximal jump amplitude
\[
A\eqdef\sup_{\substack{(u,v)\in(\min(\{\uu_{-\infty},\uu_{+\infty}\}),\max(\{\uu_{-\infty},\uu_{+\infty}\}))^2\\f(u)-\sigma\,u=f(v)-\sigma\,v}}|u-v|
\]
with $\sigma$ the speed of the front and $\uu_{-\infty}$, $\uu_{+\infty}$ its endstates.}. Yet, though they are built out of pieces of a stable profile none of those discontinuous waves is actually stable.

\medskip

As stressed by the foregoing paragraph the spatial structure of traveling-wave profiles we consider is extremely diverse. Yet an outcome of our analysis is that stability of non-degenerate traveling waves of~\eqref{eq-u} is decided by conditions that are essentially local and involve only three kinds of points: infinities, points of discontinuity, and characteristic points. More explicitly, a non-degenerate entropy-admissible piecewise regular traveling wave of profile $\uU$ and speed $\sigma$ (as in the forthcoming Definition~\ref{A.generic-traveling-wave}) is spectrally unstable if and only if it exhibits at least one of the following features:
\begin{itemize}
\item an endstate $\uu_\infty$ --- that is, a limit of $\uU$ at $+\infty$ or $-\infty$ --- such that $g'(\uu_\infty)>0$;
\item a discontinuity point $d_0$ at which $\frac{[\,g(\uU)\,]_{d_0}}{[\,\uU\,]_{d_0}}>0$ (with $[\cdot]_{d_0}$ denoting jump at $d_0$);
\item a characteristic value $\uu_\star$ --- that is, a value $\uu_\star$ of $\uU$ with $f'(\uu_\star)=\sigma$ --- such that $g'(\uu_\star)<0$. 
\end{itemize}
The reader well-trained in stability of discontinuous solutions of hyperbolic systems may expect that from the first or the second conditions stems spectral instability. The impact of characteristic points seems however to be fully clarified here for the first time, and, to our knowledge, otherwise is only (briefly) mentioned in~\cite{JNRYZ}. It is significantly more striking that any of these conditions is also sufficient to bring nonlinear instability and even more that the absence of all these (again under non-degeneracy assumptions) yield nonlinear asymptotic stability. 

Before examining the consequences of the latter instability criteria let us pause to describe roughly the nature of each instability mechanism. Concerning instabilities at infinity, a key simple observation is that if $\uU$ is continuous near $+\infty$ (resp. $-\infty$) with limit $\uu_\infty$ such that $g'(\uu_\infty)>0$ then\footnote{See Proposition~\ref{P.structure}  and specifically~\eqref{signs-infinity} in its proof.} $f'(\uu_\infty)<\sigma$ (resp. $f'(\uu_\infty)>\sigma$) so that a perturbation starting sufficiently far near such infinity will move outward the infinity at hand and keep growing exponentially as long as it has not reached some macroscopic threshold (reversing the direction of propagation, canceling pointwise linear growth rate or reaching a discontinuity). 

Instabilities created by bad jump signs are really driven by an instability of shock positions; in such case an infinitesimally small kick in position will be exponentially enhanced (in one direction or the other depending on the initial sign of the kick) up to some macroscopic threshold. Note that in the latter case the instability manifests itself not so much in that the position of the discontinuity moves (since our notion of stability allows for resynchronization of positions) but by the fact that as the shock location moves on one side it erases a macroscopic part of the original profile and on the other side it unravels some macroscopic shape not originally present. 

At last, instabilities at bad characteristic point are of wave-breaking type. Near a point $x_\star$ such that $f'(\uU(x_\star))=\sigma$ and $g'(\uU(x_\star))<0$, $f'(\uU(\cdot))-\sigma$ is positive on the left and negative on the right\footnote{See Proposition~\ref{P.structure}  and specifically~\eqref{signs-characteristic} in its proof.}  so that a perturbation localized near $x_\star$ will concentrate at $x_\star$ and cause a finite-time blow-up of derivatives. We stress that, though we prove that the latter scenario takes place, this is not completely obvious from a purely formal point of view since one may expect that the sign condition $g'(\uU(x_\star))<0$ will also bring some damping in values near $x_\star$. Thus one needs to prove that the concentration phenomenon overtakes any possible damping.
\medskip

In the foregoing discussion, we have repeatedly used sign information on $f'(\uU(\cdot))-\sigma$ deduced from signs of $g'(\uU(\cdot))$ and the profile equation. Under mild genericity assumptions (see Assumption~\ref{A.generic} below), similar considerations show that, as opposed to the otherwise quite wild possibilities, profiles of  \emph{stable} non-degenerate piecewise regular traveling wave enter in a relatively small number of classes that we list now and, except for constant states, represent in Figure~\ref{F.classes}: 

\begin{figure}[!phtb]
\begin{subfigure}{.5\textwidth}
\includegraphics[width=\textwidth]{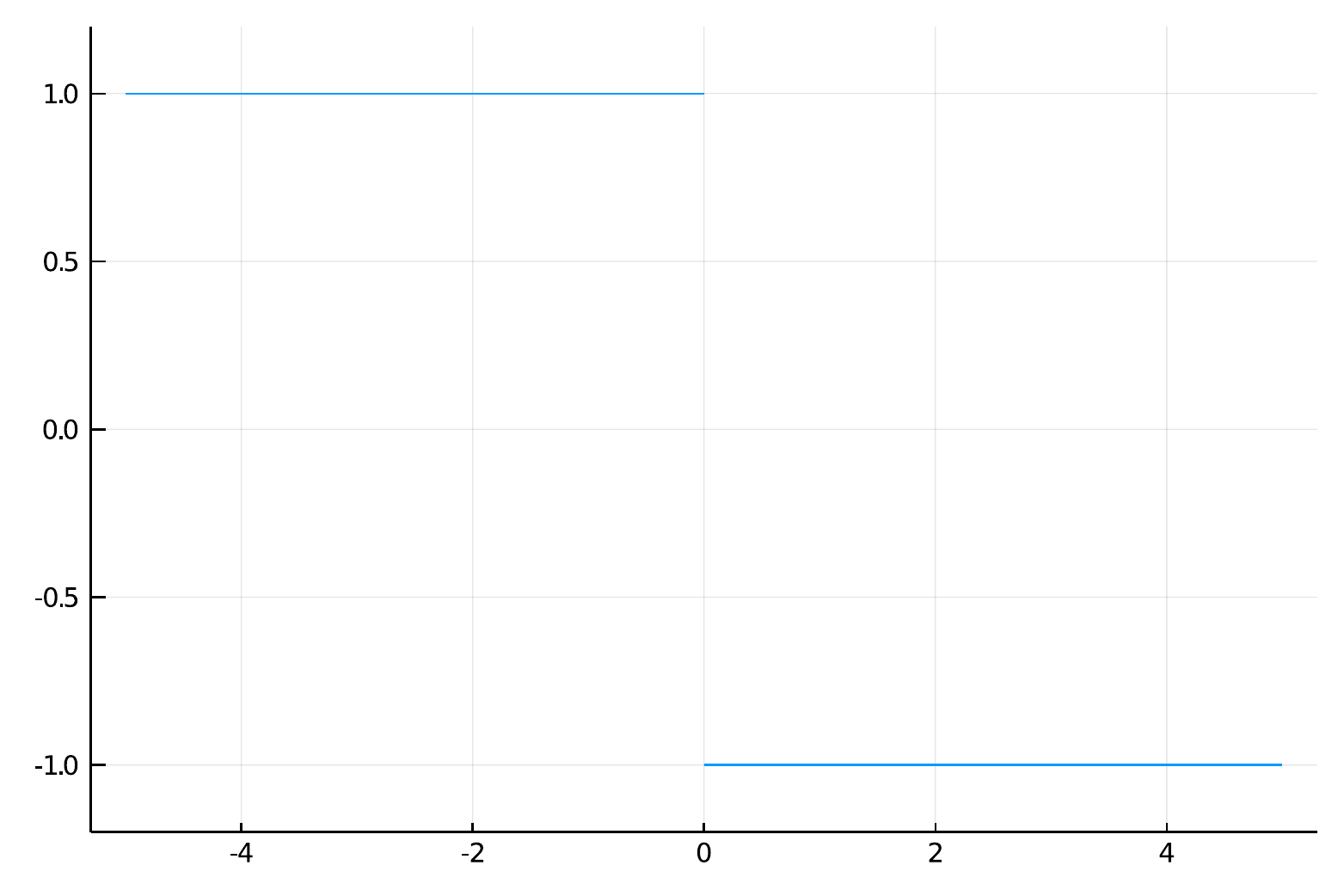}
\caption{A traveling wave profile of class \ref{class2}.}
\label{F.class2}
\end{subfigure}%
\begin{subfigure}{.5\textwidth}
\includegraphics[width=\textwidth]{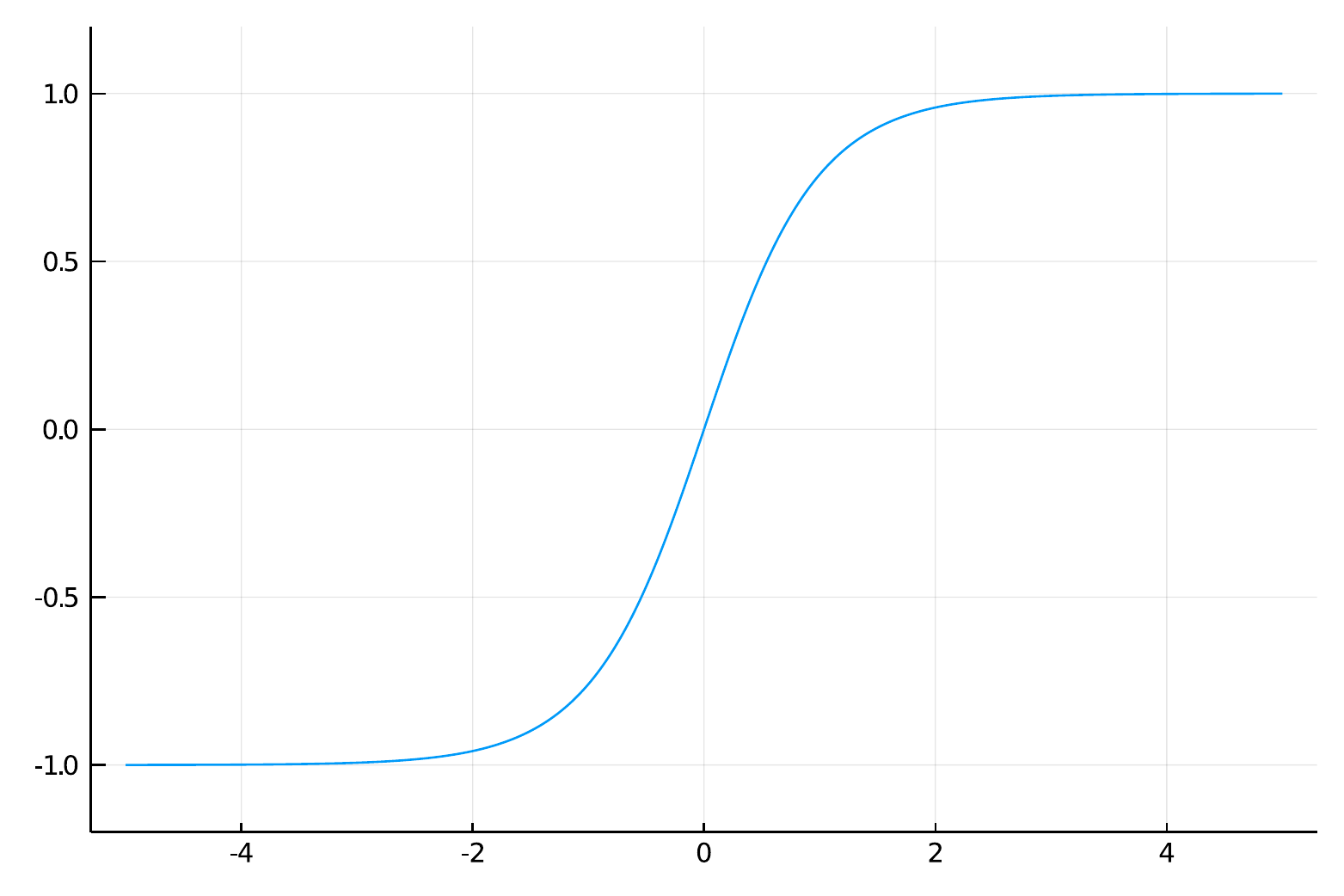}
\caption{A traveling wave profile of class \ref{class3}.}
\label{F.class3}
\end{subfigure}
\begin{subfigure}{.5\textwidth}
\includegraphics[width=\textwidth]{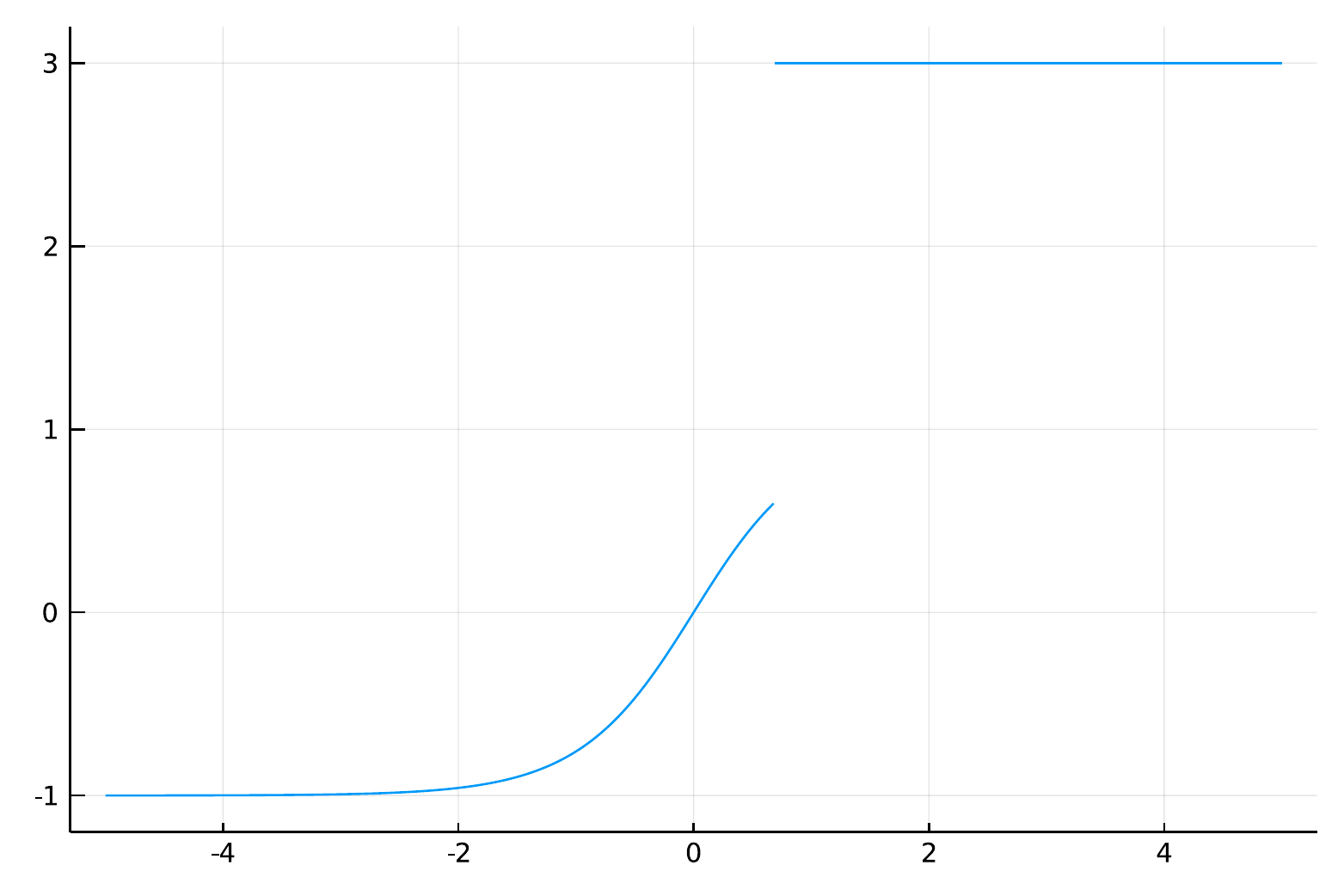}
\caption{A traveling wave profile of class \ref{class4}.}
\label{F.class4}
\end{subfigure}%
\begin{subfigure}{.5\textwidth}
\includegraphics[width=\textwidth]{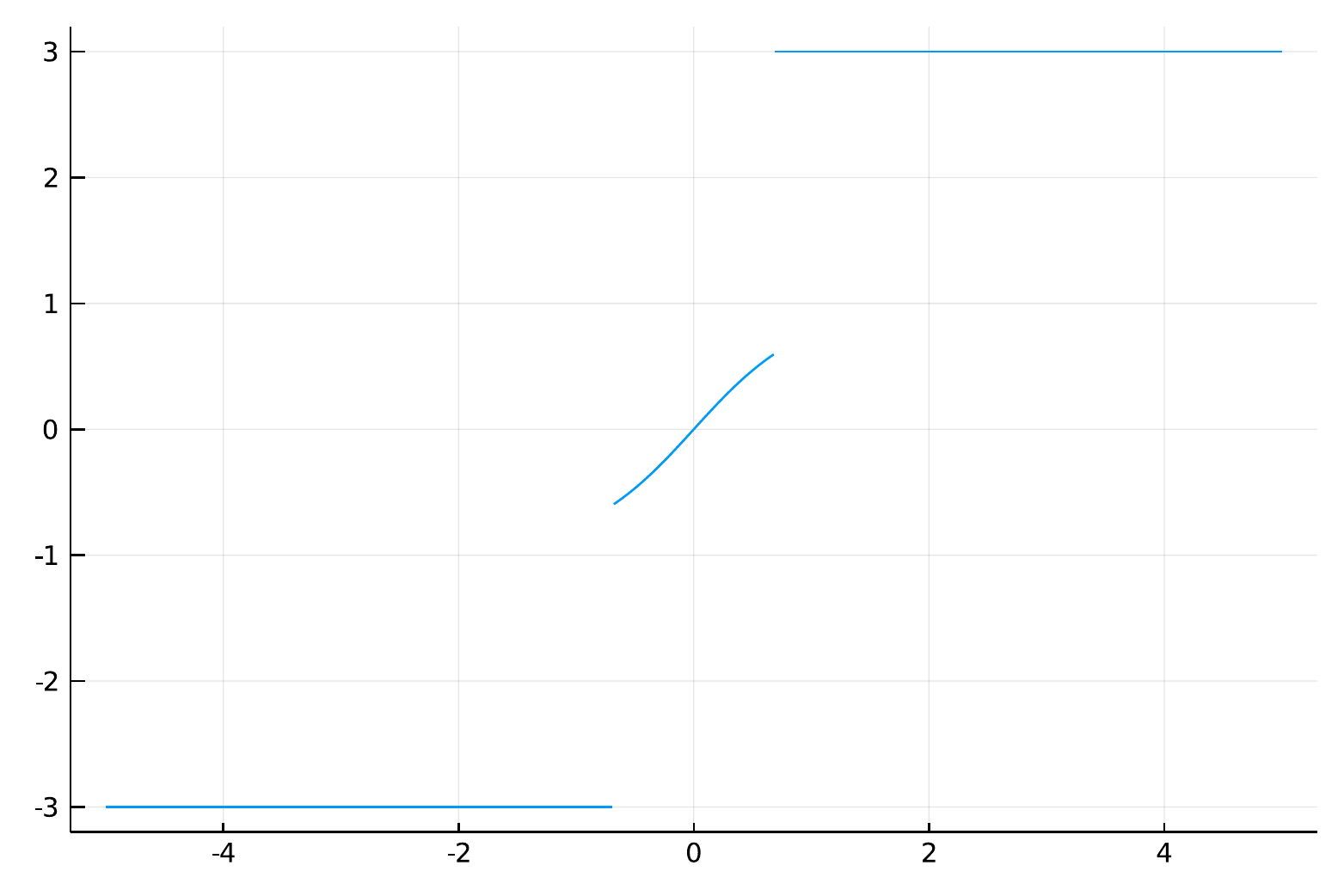}
\caption{A traveling wave profile of class \ref{class5}.}
\label{F.class5}
\end{subfigure}
\begin{subfigure}{\textwidth}
\includegraphics[width=\textwidth]{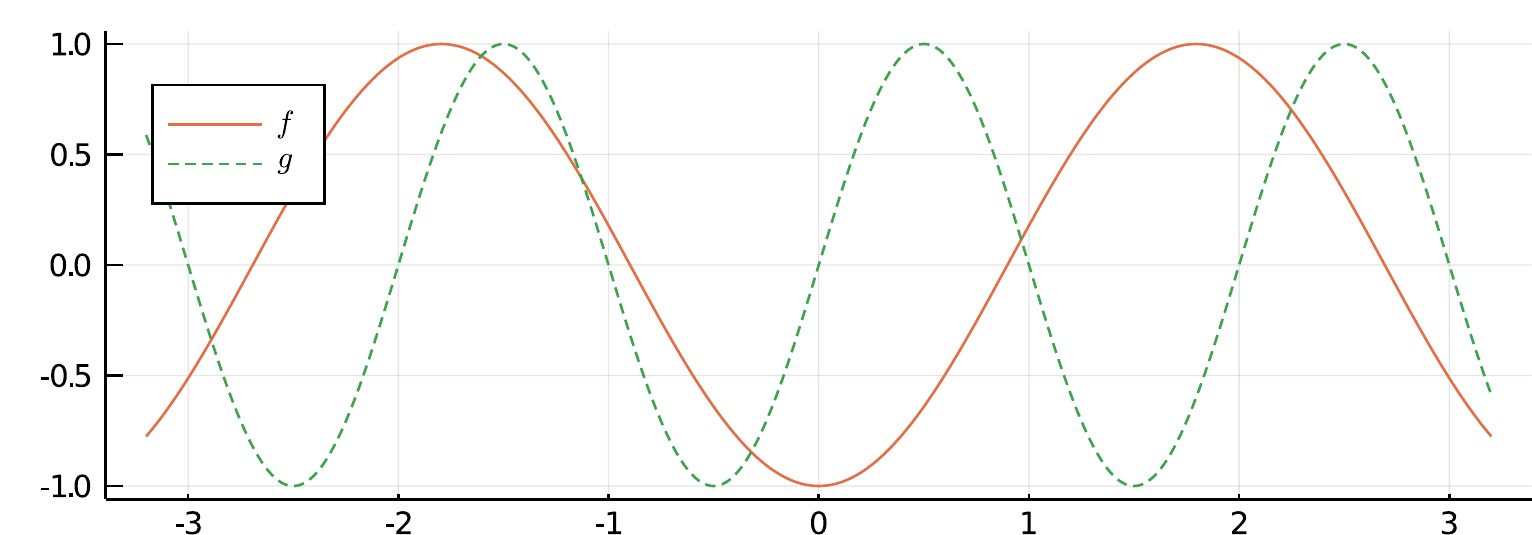}
\caption{Functions $f$ and $g$ used to trace above profiles. Specifically, $f(u) = -\cos(\tfrac74\,u)$ and  $g(u) = \sin(\pi\,u)$.}
\end{subfigure}
\caption{Classes of possibly stable non-degenerate piecewise regular traveling wave profiles, $\uU$ (constant states being omitted). The profiles represented in Figures~\ref{F.class3}, \ref{F.class4} and \ref{F.class5} pass through the characteristic value $\uu_\star=0$ at the characteristic point $x_\star=0$.  All the traveling waves represented have speed $\sigma=0$ and are spectrally and nonlinearly stable by Theorem~\ref{T.classification}.}
\label{F.classes}
\end{figure}
\begin{enumerate}
\item \label{class1} constant states (that is, $\uU$ takes only one value) ;
\item \label{class2} Riemann shocks (that is, $\uU$ takes two values, one value on a half-line, another one on the complementary half-line);
\item \label{class3} continuous fronts (that is, $\uU$ is continuous with distinct finite limits at $+\infty$ and $-\infty$) containing exactly one characteristic point;
\item \label{class4} profiles that are constant on a half-line and jump to a continuous part with a bounded limit at infinity, containing exactly one characteristic point;
\item \label{class5} profiles that are constant near $-\infty$ then jump to a continuous part containing exactly one characteristic point and then jump again to a constant value near $+\infty$. 
\end{enumerate}
Moreover the last and second-to-last possibilities are ruled out in the convex/concave case when $f''$ does not vanish. Let us clarify that the former regularity structure is not sufficient to deduce stability and that for each element of these classes one still needs to satisfy the aforementioned sign conditions at jumps, endstates and characteristic values.
\medskip

To complete this introduction we comment now on the part of our analysis proving nonlinear asymptotic stability from the aforementioned sign conditions. There are at least three obstacles that single out the problem at hand from even recent instances of more classical analyses~\cite{Henry-geometric,KapitulaPromislow-stability,JNRZ-conservation}:
\begin{itemize}
\item background waves are typically discontinuous, which prevents a direct stability analysis relying on naive Taylor expansions;
\item waves that are not piecewise constants contain a characteristic point so that the principal symbols of corresponding linearized operators vanish at some points;
\item Equation~\eqref{eq-u} is quasilinear and no (strong) regularization is present in the equation. 
\end{itemize}
Though much less studied than corresponding questions for either semilinear equations of any type or for quasilinear parabolic equations, the last point is the less challenging of our analysis and it is an issue here mostly because we aim at and obtain results that are optimal in terms of regularity in the class of (piecewise) strong solutions and do not require decay at infinity of perturbations, otherwise we could have followed a more standard strategy involving energy estimates to close in regularity as in, for instance,~\cite{KawashimaYong04,MasciaZumbrun05,BianchiniHanouzetNatalini07}. To achieve such sharp results, as in~\cite{DR1} a key point in our analysis is the identification of a class of perturbations of the linearized operators for which we may obtain decay estimates similar to those for the linearized dynamics. The foregoing class of linear dynamics needs to be suitably small to retain the key properties of the linearized dynamics and suitably large\footnote{In particular, the class must include time-dependent generators.} to be involved in a Duhamel-like formulation of~\eqref{eq-u} on which we may close a nonlinear estimate, lossless from both the points of view of regularity and decay. 

The first point is considerably more uncommon and only two instances have been dealt with so far, in the very recent\footnote{Note moreover that~\cite{DR1} was written as a companion paper to the present contribution and that parts of the key arguments used in~\cite{YangZumbrun20} actually originate in private communications of the second author of the present contribution to the second author of~\cite{YangZumbrun20} (see the Acknowledgments section there).}~\cite{DR1,YangZumbrun20}. Unlike~\cite{YangZumbrun20} (that not only restricts to smooth perturbations but also assumes that these are supported away from reference discontinuities) but as in~\cite{DR1} we make the challenge even greater by authorizing perturbations introducing new discontinuities. Thus, since small discontinuities may disappear in finite time, our analysis includes cases where the structure of discontinuities change with time. To achieve this goal, as in~\cite{DR1} we make the most of the scalar nature of the equations. Indeed in the scalar case the Rankine-Hugoniot condition may be solved by adjusting the shock location so that one strategy to analyze the piecewise regular case is to extend each smooth piece to the whole domain and then glue them according to the dynamics of the Rankine-Hugoniot conditions. This essentially breaks the study of the dynamics into two kinds of elementary problems: the analysis of whole domain problems on one hand, the study of the motions of shock locations into known environments on the other hand. Note that unlike the case treated in~\cite{DR1} where background waves are piecewise constants here the extension to the whole line is by essence artificial since in general smooth parts of stable wave profiles of~\eqref{eq-u} are not parts of a stable front of~\eqref{eq-u}. Yet we may still modify\footnote{The argument strongly echoes the classical reduction from the locally Lipschitz case to the globally Lipschitz case of the Cauchy-Lipschitz theorem.}~\eqref{eq-u} outside values of interest to turn those parts as portions of stable fronts of an equation with the same structure as~\eqref{eq-u}. As in~\cite{DR1} we stress that the notion of solution we use ensures uniqueness by the classical Kru\v zkov theory~\cite{Kruzhkov} and that the non uniqueness in the extension part of our argument is compensated by the fact that no artificial part is revealed by the motion of the shock locations.

At last, to our knowledge, the nonlinear analysis near waves with characteristic points is carried out here for the very first time. It seems that even its impact on the spectral problem is fully analyzed here for the first time, though some partial considerations were already present in~\cite{JNRYZ}. The paramount importance to be able to include characteristic points in the development of a general stability theory for hyperbolic systems originates in the fact that, as easily deduced from entropy constraints, any solution containing two shocks must contain a characteristic point. 

To  provide the reader with some insight on the impact of characteristic points, let us first recall what are the usual expectations for more standard waves with the simplest nontrivial spatial structure such as single-bump solitary waves or monotonic fronts (as opposed to constants on one side or periodic waves on the other). In general, for such stable waves, the best one may expect is that solutions arising from an initial perturbation of the wave profile will converge to some spatial translate of the original wave. This is known as asymptotic orbital stability with asymptotic phase. Thus part of the nonlinear analysis consists in projecting out the non-decaying critical phase dynamics. This typically involves the spectral projector of the linearized dynamics associated with the simple eigenvalue $0$, with (right) eigenfunction the spatial derivative of the wave profile. Separating the phase dynamics from shape deformations may thus be interpreted as imposing some orthogonality with the dual eigenfunction of the adjoint operator. Note that in general the resulting asymptotic phase depends in an intricate nonlinear way on the initial perturbation.

Characteristic points impact dramatically the critical phase dynamics in various striking ways. To begin with, note that the singularity of the linearized operator is reflected in the fact that the element of the kernel of the adjoint operator, dual to the derivative of the profile, is the Dirac mass at the characteristic point, so that the "orthogonality" condition is both very singular and quite simple. Concerning the latter, note in particular that taking the scalar product with a Dirac mass commutes with (local) nonlinear operations. Consistently, at the nonlinear level the presence of a characteristic point pins the critical asymptotic phase. More explicitly, solutions arising from a small initial perturbation converge to the translate of the original wave whose characteristic point agrees in location with the one of the perturbed initial data. From this point of view, in the generic case above-mentioned the classes of non-degenerate stable waves of~\eqref{eq-u} may be split further into three groups: constant states for which direct stability holds, Riemann shocks that exhibits a nearby classical phase dynamics and the three other classes associated with a pinned phase dynamics. We stress that, consistently with the elements of our strategy of proof sketched above, our analysis requires the identification of a suitably large class of linear dynamics retaining the key elements of the linearized dynamics expounded here.

\medskip

Since our contribution answers most of the general questions concerning the large-time dynamics near traveling waves of scalar first-order balance laws, we would like to conclude our introduction by pointing out a question that we do leave for further study. Concerning stable traveling waves of~\eqref{eq-u} whose profile exhibits no characteristic points, we have already proved in~\cite{DR1} that they are also nonlinearly stable as plane waves of multidimensional versions of~\eqref{eq-u}. In contrast, for more general stable waves we only prove that they are transversely stable under perturbations supported away from characteristic points. The restriction somewhat echoes the restriction in~\cite{YangZumbrun20} where profiles have no characteristic points and perturbations are supported away from discontinuities (but for a specific $2\times 2$ system rather than a general scalar equation). We expect that solutions arising from the multidimensional perturbation of general stable plane waves may converge to nearby genuinely multidimensional waves but we leave this for further investigation. In another direction of extension, we point out that even in the one-dimensional case, the derivation of a general framework for the stability analysis of discontinuous solutions of hyperbolic \emph{systems} of balance laws is still largely open.

\medskip

The rest of the present paper is organized as follows. In the next section we gather pieces of information on the structure of the wave profiles of~\eqref{eq-u} and introduce the precise definitions used throughout. Then we state and prove all our instability results and deduce a precise classification of spectrally stable waves. In the second to last section we analyze the nonlinear stability of spectrally stable continuous fronts, thus obtaining the key element missing in the overall strategy towards nonlinear stability derived in~\cite{DR1}. In the final section we complete our nonlinear stability analysis. 

\section{Preliminaries on traveling waves}\label{S.structure}

Prior to tackling stability/instability issues, we set terminology and collect elementary pieces of information on wave profiles. The reader is referred to~\cite{KapitulaPromislow-stability} for general background on traveling waves and to~\cite{Bressan} for elementary background on hyperbolic equations.

\subsection{Structure of profiles}

First we examine the structure of non-degenerate entropy-admissible traveling waves. We assume henceforth that $f\in\cC^2(\RR)$ and $g\in\cC^1(\RR)$.

\begin{Definition}\label{A.traveling-wave}
A \emph{piecewise regular traveling-wave} solution to~\eqref{eq-u} is an entropy solution to~\eqref{eq-u} in the form $(t,x)\mapsto \uU(x-\sigma\,t)$ with $(\uU,\sigma)\in L^\infty(\RR)\times\RR$ such that there exists a closed discrete set $D$ (possibly empty) such that $\uU$ is $\cC^1$ on $\RR\setminus D$.
\end{Definition}

Note that for $(\uU,\sigma,D)$ as above, $\RR\setminus D$ is a union of disjoint open intervals, and 
\begin{equation}\label{eq:ODE}
\forall x\in \RR\setminus D, \quad \big(f'(\uU(x))-\sigma\big)\uU'(x)=g(\uU(x)).
\end{equation}
Since the latter ODE is scalar, a wealth of information on $\uU$ may be derived from it provided it is non-degenerate. In this direction note that if for some $u\in \uU(\RR\setminus D)$, $f'(u)=\sigma$ then necessarily $g(u)=0$ and that the profile ODE is non-degenerate near this value provided $f$ and $g$ are sufficiently regular at $u$ and $g$ vanishes at least as the same order as $f'-\sigma$ at $u$. 

\begin{Proposition}\label{P.monotonic}
Let $(\uU,\sigma,D)$ define a piecewise regular traveling-wave solution to~\eqref{eq-u}. 
Let $X$ be a connected component of $\RR\setminus D$ such that 
\[\forall \uu_\star\in \uU(X),
\qquad f'(\uu_\star)=\sigma \Longrightarrow f''(\uu_\star)\neq 0\]
and such that $F_\sigma:\uU(X)\to\RR$ defined by
\begin{equation}\label{def-F}
\qquad F_\sigma(u)=\begin{cases}\frac{g(u)}{f'(u)-\sigma}&\textrm{if}\ f'(u)-\sigma\neq0\\
\frac{g'(u)}{f''(u)}&\textrm{otherwise}
\end{cases}
\end{equation}
is locally Lipschitz near any of its zeroes. Then $\uU$ is either constant or strictly monotonic on $X$. In particular, if this is true for any connected component then at any $d\in D$, $\uU$ possesses finite right and left limits, $\uU(d^-)$ and $\uU(d^+)$, and if one of the connected component is not lower (resp. upper) bounded then $\uU$ possesses a finite limit at $-\infty$ (resp. at $+\infty$).
\end{Proposition}

\begin{proof}
Under the foregoing assumption, we have
\[\forall x\in X, \quad \uU'(x)=F_\sigma(\uU(x))\,.\]
The assumption on $F_\sigma$ ensures that if $\uU'$ vanishes somewhere in $X$ then it vanishes everywhere on $X$, hence the claim on monotonicity. The existence of finite limits stems from monotonicity and boundedness of $\uU$. 
\end{proof}

Reciprocally, note that if $(\uU,\sigma)$ are such that there exists a closed discrete set $D$ (possibly empty) such that $\uU$ is $\cC^1$ on $\RR\setminus D$, and $\uU$ possesses left and right limits at any point of $D$ then $(t,x)\mapsto \uU(x-\sigma\,t)$ is an entropic solution provided that
\[\forall x\in \RR\setminus D, \quad \big(f'(\uU(x))-\sigma\big)\uU'(x)=g(\uU(x))\]
and at any $d\in D$ stand both the Rankine-Hugoniot condition,
\[
-\sigma[\uU]_d+[f(\uU)]_d=0\,,
\]
where we use jump notation $[A]_d\eqdef A(d^+)-A(d^-)$ with
\begin{align*}
A(d^+)&\eqdef \lim_{\delta\nearrow 0} A(d+\delta)\,,&
A(d^-)&\eqdef \lim_{\delta\nearrow 0} A(d-\delta)\,,
\end{align*}
and the Oleinik condition: for any $v$ strictly between $\uU(d^+)$ and $\uU(d^-)$
\[
\frac{f(v)-f(\uU(d^-))}{v-\uU(d^-)}\geq\frac{f(v)-f(\uU(d^+))}{v-\uU(d^+)}\,. 
\]
Note that assuming the Rankine-Hugoniot condition, the Oleinik condition implies $f'(\uU(d^-))\geq \sigma$ and $f'(\uU(d^+))\leq \sigma$.

This motivates the following definition.

\begin{Definition}\label{A.generic-traveling-wave}
Let $(\uU,\sigma,D)$ define a piecewise regular entropy-admissible traveling-wave solution to~\eqref{eq-u}. We say that the corresponding traveling wave is \emph{non-degenerate} provided that
\begin{enumerate}
\item for any $\uu_\star\in \uU(\RR\setminus D)$,
\[
\qquad f'(\uu_\star)=\sigma \Longrightarrow f''(\uu_\star)\neq 0\]
and, if $g'(\uu_\star)=0$,\footnote{This possibility is eventually ruled out in Proposition~\ref{P.structure}.} $F_\sigma$ defined by~\eqref{def-F} is locally Lipschitz near $\uu_\star$;
\item at any $d\in D$ 
\[f'(\uU(d^-))-\sigma>0\qquad\textrm{and}\qquad f'(\uU(d^+))-\sigma<0\]
and for any $v$ strictly between $\uU(d^+)$ and $\uU(d^-)$
\[
\frac{f(v)-f(\uU(d^-))}{v-\uU(d^-)}>\frac{f(v)-f(\uU(d^+))}{v-\uU(d^+)}\,;
\]
\item if $\uU$ possesses a finite limit\footnote{Recall that it amounts to $\uU$ being defined in a neighborhood of $+\infty$ (resp. $-\infty$), that is, one of the connected components of $\RR\setminus D$ not to be bounded from above (resp. below).} $\uu_{\infty}$ at $+\infty$ or $-\infty$, then
\[f'(\uu_\infty)\neq\sigma \quad\textrm{and}\quad g'(\uu_\infty)\neq0\,.
\]
\end{enumerate}
\end{Definition}

Note that the second and third conditions discard the possibility that $\uU$ could be constant on a connected component of $\RR\setminus D$ with value a zero of $f'-\sigma$, thus the first condition may be equivalently written as: for any $\uu_\star\in \uU(\RR\setminus D)$,
\[
\qquad f'(\uu_\star)=\sigma \Longrightarrow \left(f''(\uu_\star)\neq 0\quad\textrm{and}\quad 
g'(\uu_\star)\neq0
\right)\,.
\]
Concerning the third point, note that the non-characteristic condition $f'(\uu_\infty)\neq\sigma$ is sufficient to deduce $g(\uu_\infty)=0$ so that the third condition is really the condition that $\uu_\infty$ is non-characteristic and is a simple zero of $g$.

\begin{Proposition}\label{P.structure}
Let $(\uU,\sigma,D)$ define a non-degenerate piecewise regular entropy-admissible traveling-wave solution to~\eqref{eq-u}.
\begin{enumerate}
\item If $\uU$ possesses a finite limit $\uu_{\infty}$ at $+\infty$ or $-\infty$ then $g(\uu_{\infty})=0$.
\item If $\uu_\star\in \uU(\RR\setminus D)$ is a characteristic value, that is, $f'(\uu_\star)=\sigma$, then $g(\uu_\star)=0$, $g'(\uu_\star)\neq0$, and on connected components of $\RR\setminus D$ where $\uU$ takes the value $\uu_\star$, $\uU$ is strictly monotonic with monotonicity given by the sign of $g'(\uu_\star)/f''(\uu_\star)$.
\item If $\uU$ is constant, with value $\uu$, on a connected component $X$ of $\RR\setminus D$, then $g(\uu)=0$, $X$ is unbounded and if $\sup X<+\infty$ (resp. $\inf X>-\infty$) $f'(\uu)-\sigma>0$ (resp. $f'(\uu)-\sigma<0$).
\item On a bounded connected component of $\RR\setminus D$, 
\subitem $\uU$ passes through a characteristic value an odd number of times; 
\subitem the signs of $g'$ alternate along these characteristic values, starting with positive value; 
\subitem the signs of $f'-\sigma$ alternate between characteristic values, starting with negative value.
\item On a connected component of $\RR\setminus D$ bounded from above but not from below on which $\uU$ is not constant,
\subitem  $\uU$ passes through an even (resp. odd) number of characteristic values if $g'(\uu_{-\infty})<0$ (resp. $g'(\uu_{-\infty})>0$), with $\uu_{-\infty}$ the limit of $\uU$ at $-\infty$; 
\subitem the signs of $g'$ alternate along these characteristic values, finishing with positive value; 
\subitem the signs of $f'-\sigma$ alternate between characteristic values, finishing with positive value;
\subitem $f'(\uu_{-\infty})-\sigma$ has the sign of $g'(\uu_{-\infty})$.
\item On a connected component of $\RR\setminus D$ bounded from below but not from above on which $\uU$ is not constant, 
\subitem $\uU$ passes through an even (resp. odd) number of characteristic values if $-g'(\uu_{+\infty})<0$ (resp. $g'(\uu_{+\infty})>0$), with $\uu_{+\infty}$ the limit of $\uU$ at $+\infty$;
\subitem the signs of $g'$ alternate along these characteristic values, starting with positive value; 
\subitem the signs of $f'-\sigma$ alternate between characteristic values, starting with negative value;
\subitem $f'(\uu_{+\infty})-\sigma$ has the sign of $-g'(\uu_{+\infty})$.
\item If $D=\emptyset$, 
\subitem $\uU$ passes through an even (resp. odd) number of characteristic values if ${g'(\uu_{-\infty})g'(\uu_{+\infty})<0}$ (resp. $g'(\uu_{-\infty})g'(\uu_{+\infty})>0$), with $\uu_{\pm\infty}$ the limits of $\uU$ at $\pm\infty$; 
\subitem the signs of $g'$ alternate along $\uu_{-\infty}$, characteristic values, and $\uu_{+\infty}$; 
\subitem the signs of $f'-\sigma$ alternate between characteristic values; 
\subitem $f'(\uu_{+\infty})-\sigma$ (resp. $f'(\uu_{-\infty})-\sigma$) has the sign of $- g'(\uu_{+\infty})$ (resp. $g'(\uu_{-\infty})$). 
\end{enumerate}
\end{Proposition}

\begin{proof}
Along the proof we use the vector-field $F_\sigma$ introduced in the proof of Proposition~\ref{P.monotonic}. Note that the monotonicity of $\uU$ in a given connected component of $\RR\setminus D$ is given by the sign of $F_\sigma$ at any value taken on the connected component under consideration. As we have already observed we also know that if $\uU$ is constant on a connected component of $\RR\setminus D$ its value there is not a characteristic value. This proves the second point.

Likewise, in the first point, since $\uu_\infty$ is not a characteristic value, $F_\sigma$ extends to a neighborhood of $\uu_\infty$ thus $\uu_\infty$ must be a zero of $F_\sigma$ that is not a characteristic value, i.e. $g(\uu_\infty)=0$. This proves the first point of the proposition.

To prove the remaining points, since we already know the sign of $f'(\uU(\cdot))-\sigma$ near a discontinuity of $\uU$, we only need to connect its sign near $\pm\infty$ or near a characteristic point (that is, a point $x_\star$ where $\uu_\star\eqdef \uU(x_\star)$ is a characteristic value, that is $f'(\uu_\star)=\sigma$) to the sign of $g'(\uU(\cdot))$. At a characteristic point $x_\star$, with $\uu_\star\eqdef \uU(x_\star)$ we also have $(f'(\uU(\cdot)))'(x_\star)=f''(\uU(x_\star))\,\uU'(x_\star)=g'(\uu_\star)\neq 0$ so that
\begin{equation}\label{signs-characteristic}
f'(\uU(x))-\sigma\,\stackrel{x\to x_\star}{\sim}\,g'(\uu_\star)\,(x-x_\star) \,.
\end{equation}
Near $\pm\infty$, if $\uU$ is defined but not constant, $\uU'$ does not vanish and
\begin{equation}\label{signs-infinity}
f'(\uU(x))-\sigma\,=\,\frac{g(\uU(x))}{\uU'(x)}\,\stackrel{x\to \pm\infty}{\sim}\,g'(\uu_{\infty})\,\frac{\uU(x)-\uu_{\infty}}{\uU'(x)}\,,
\end{equation}
from which stems the claim on signs near $\pm\infty$. The proof is complete.
\end{proof}

\subsection{Notions of stability}

\subsubsection{Nonlinear stability}

We now introduce suitable notions of stability. Our stability results provide a detailed description of the dynamics so that they can be read without much preliminary abstract discussion. In contrast, much more care is needed to ensure that our instability results reflect a genuine instability and not the misuse of a deceptive notion of stability.

With this respect, we recall that it is well-known that the relevant notion of stability, even for smooth traveling waves of smoothing equations, must encode control on deformations of shape but allow for resynchronization of positions. As a preliminary we make two remarks illustrating the dramatic effect of disregarding synchronization. The simplest observation is that since any non-constant traveling wave comes in a family of traveling waves obtained by translating it spatially, direct \emph{asymptotic} stability can not hold if translation operates continuously on the background traveling wave for the topology at hand. Even worse, if near the background wave lie infinitesimally close waves with infinitesimally close but distinct speeds a direct comparison concludes instability since an infinitesimally small initial perturbation will result in a macroscopic shift, whereas the variation in shape remains infinitesimal. For waves with the simplest possible spatial structure, it is sufficient to tune one position parameter and thus to investigate the possibility for a solution $u$ to be written in the form
\[
u(t,x+\sigma t+\psi(t))=\uU(x)+\tu(t,x)
\]
with $(\tu(t,\cdot),\psi'(t))$ small provided $\tu(0,\cdot)$ is sufficiently small initially. This encodes the notion of orbital stability. When the spatial structure of the wave at hand is more complex the notion of stability needs to be even more flexible. The extreme case is well illustrated by periodic waves for which there is essentially an infinite number of positions to adjust. For specific discussions on those the reader is referred to~\cite{R,JNRZ-conservation,R_Roscoff,R_linKdV}. Since here we are dealing with waves with possibly quite wild spatial variations we do need to use a notion of stability at least as versatile as in the periodic case.

Moreover, as already observed in~\cite[Section~4]{JNRYZ}, in a context where discontinuities are present and one aims at using topologies controlling piecewise smoothness, it is even more crucial to synchronize all discontinuities. Indeed, the relevant notion of stability is a close parent to the notion of proximity obtained with the Skohokhod metric on functions with discontinuities, which allows for a near-identity synchronization of jumps.

With this is in mind, let $(\uU,\sigma,D)$ define a non-degenerate piecewise regular entropy-admissible traveling-wave solution to~\eqref{eq-u} and introduce a relevant nonlinear stability framework. Adapting the notion of space-modulated stability, coined in~\cite{JNRZ-conservation} and already used for discontinuous waves in~\cite[Section~4]{JNRYZ}, to a non periodic context, we investigate the existence of an entropic solution $u$ to~\eqref{eq-u} in the form
\begin{equation}\label{Ansatz}
u(t,x+\sigma t+\psi(t,x))=\uU(x)+\tu(t,x)
\end{equation}
with $(\tu,\d_x\psi,\d_t\psi)(t,\cdot)$ small provided they are sufficiently small initially. Note that we aim at a space shift $\psi(t,\cdot)$ regular on $\RR$ and a shape deformation $\tu(t,\cdot)$ regular on $\RR\setminus D$ with limits from both sides at each $d\in D$. We stress that our \emph{positive} stability results, as those in~\cite{DR1}, allow for classes of initial perturbations even larger but they include such configurations as special cases. Our instability results shall show that there does not exist any $\psi$ able to bring $u$ close to $\uU$ in the sense associated with~\eqref{Ansatz}.

\subsubsection{Spectral stability}\label{S.spectral-def}

We also want to exhibit spectral instabilities. To do so, we need to identify spectral problems consistent with the foregoing notion of dynamical stability. First we insert the ansatz~\eqref{Ansatz}.

\begin{Definition}\label{A.piecewise-solution}
On a time interval $I\subset \RR$ we say that an entropy solution to~\eqref{eq-u}, $u\in L^\infty_{\rm loc}(I\times\RR)$, is \emph{piecewise regular with invariant regularity structure} or \emph{invariably piecewise regular} if there exist a closed discrete set $D$ and a local phase shift $\psi\in \cC^1(I\times \RR)$ such that for any\footnote{Obviously in the present definition, the separation of $\sigma t$ from $\psi(t,x)$ is immaterial and done purely to match with~\eqref{Ansatz}.} $t\in I$, $x\in\RR\mapsto x+\sigma t+\psi(t,x)\in\RR$ is bijective and such that $(t,x) \mapsto u(t,x+\sigma t+\psi(t,x))$ is $\cC^1$ on $I\times( \RR\setminus D)$.
\end{Definition}

For $u$ as in~\eqref{Ansatz} being a invariably piecewise regular entropy solution to~\eqref{eq-u} reduces to interior equations
\begin{align*}
\d_t(\tu-\psi\uU')+&\d_x((f'(\uU)-\sigma)(\tu-\psi\uU'))
-g'(\uU)(\tu-\psi\uU')\\
=&-\d_x(f(\uU+\tu)-f(\uU)-f'(\uU)\tu)+g(\uU+\tu)-g(\uU)-g'(\uU)\tu\\
&+\d_x\psi\,(g(\uU+\tu)-g(\uU))-\d_t(\d_x\psi\,\tu)+\d_x(\d_t\psi\,\tu)
\end{align*}
on $\RR\setminus D$, and at any $d\in D$ the Rankine-Hugoniot condition
\[
\d_t\psi\,[\,\uU\,]_d\,-\,[\,(f'(\uU)-\sigma)\tu\,]_d
\,=\,[\,f(\uU+\tu)-f(\uU)-f'(\uU)\tu)\,]_d\,-\,\d_t\psi\,[\,\tu\,]_d
\]
and the Oleinik entropy condition (which we omit to write down). Since we only consider waves satisfying strict entropy condition, entropy condition do not show up at the linearized level.

The foregoing discussion suggests to consider at least a subclass of the following linearized problem
\begin{align*}
\d_t(\tu-\psi\uU')+\d_x((f'(\uU)-\sigma)(\tu-\psi\uU'))
-g'(\uU)(\tu-\psi\uU')&=\tA+\d_x(\tB)&&\textrm{on }\RR\setminus D\,,\\
\d_t\psi\,[\,\uU\,]_d\,-\,[\,(f'(\uU)-\sigma)\tu\,]_d
&=\,-\,[\,\tB\,]_d&&\textrm{at any }d\in D\,.
\end{align*}
Since this linear problem is time-independent, it is natural to analyze it through the family of spectral problems
\begin{align*}
\lambda\,(\tu-\psi\uU')+\d_x((f'(\uU)-\sigma)(\tu-\psi\uU'))
-g'(\uU)(\tu-\psi\uU')&=A+\d_x(B)&&\textrm{on }\RR\setminus D\,,\\
\lambda\psi\,[\,\uU\,]_d\,-\,[\,(f'(\uU)-\sigma)\tu\,]_d
&=\,-\,[\,B\,]_d&&\textrm{at any }d\in D\,,
\end{align*}
(with new $(\psi,\tu)$ playing the role of the value at $\lambda$ of the Laplace transform in time of the old $(\psi,\tu)$). 

For the sake of tractability we relax the foregoing problem into the problem of the determination of the spectrum of a given operator. To do so we choose $\cX$ a functional space of (classes of) locally integrable functions on $\RR\setminus D$ and $\cY$ a space of functions on $D$. We enforce the rather weak\footnote{It is sufficient to know that $\d_x((f'(\uU)-\sigma)w)-g'(\uU)w\in \cX$ implies that $\d_x((f'(\uU)-\sigma)w)$ (defined on $\RR\setminus D$) is locally integrable near any $d\in D$.} condition that for any $w\in\cX$ such that $\d_x((f'(\uU)-\sigma)w)-g'(\uU)w\in \cX$, $(f'(\uU)-\sigma)w$ possesses limits from the left and from the right at any point\footnote{Since $(f'(\uU)-\sigma)$ possesses nonzero limits there this is equivalent to $w$ possessing limits there.} $d\in D$. Then one may define on $\cX\times\cY$, the operator with maximal domain
\[
\cL(w,(y_d)_{d\in D})\eqdef \left(-\d_x((f'(\uU)-\sigma)w)+g'(\uU)w,
\left(y_d\frac{[\,(f'(\uU)-\sigma)\uU'\,]_d}{[\,\uU\,]_d}\,+\,\frac{[\,(f'(\uU)-\sigma)w\,]_d}{[\,\uU\,]_d}\right)_{d\in D}\right)\,.
\] 

\begin{Definition}
We call $\cX\times\cY$-\emph{spectrum} of the linearization about the wave $\uU$ the spectrum of $\cL$. We say that the wave $\uU$ is spectrally unstable if there exists an element of the latter spectrum with positive real part.
\end{Definition}

\begin{Remark}
Note that when relaxing the original problem to the spectrum of $\cL$, that is, when replacing $(\tu,\psi)$ with $(\tw,(y_d)_{d\in D})=(\tu-\psi\uU',(\psi(d))_{d\in D})$ we have essentially reduced the role of $\psi$ to the synchronization of discontinuities. The expectation is that the corresponding inaccuracy only blurs the separation between algebraic growth/decay but does not impact the detection of exponential growth/decay by spectral arguments.
\end{Remark}

For background on unbounded operators and their spectra the reader is referred to~\cite{Davies}. We simply recall that to prove spectral instability it is sufficient to find $\lambda$ with positive real part and
\begin{itemize}
\item either an associated Weyl sequence, that is, a sequence $((w^n,(y_d^n)_{d\in D}))_{n\in\NN}$ of elements of the domain of $\cL$ such that
\[
\frac{\|(\lambda-\cL)(w^n,(y_d^n)_{d\in D})\|_{\cX\times\cY}}{\|(w^n,(y_d^n)_{d\in D})\|_{\cX\times\cY}}
\stackrel{n\to\infty}{\longrightarrow}0
\]
\item or a nonzero element of the kernel of the adjoint of $\lambda-\cL$.
\end{itemize}

\section{Instability mechanisms}\label{S.instabilities}

In the present section, we prove that the criteria expounded in the introduction do provide both spectral and nonlinear instability. 

Since the traveling waves we consider have very diverse global spatial structure, our instability analysis must be infinitesimally localized near the point under consideration (jump, infinity or characteristic). In particular, our arguments do apply to classes of waves that are actually larger than the class of non-degenerate piecewise regular traveling wave we focus on.

In the following propositions, $W^{k,p}$ denotes the Sobolev space of functions whose derivatives up to order $k$ are in $L^p$ and $BUC^k$ the space of functions whose derivatives up to order $k$ are bounded and uniformly continuous.

\subsection{Instabilities at infinity}

Though we have not found in the literature the exact instability results we need, the mechanism for near infinity instability is extremely classical.

\begin{Proposition}\label{P.spectral-instability-infinity}
Let $k\in\NN$ and $f\in \cC^{k+2}(\RR)$, $g\in\cC^{2}(\RR)\cap\cC^{k+1}(\RR)$ and $(\uU,\sigma,D)$ define a non-degenerate piecewise regular entropy-admissible traveling-wave solution to~\eqref{eq-u}. 
\begin{enumerate}
\item If $\uU$ admits a limit $\uu_{\infty}$ at $+\infty$ or $-\infty$ then $g'(\uu_{\infty})+i\RR$ is included in the $\cX\times\cY$-spectrum of the linearization about $\uU$ provided for some neighborhood $I$ of $+\infty$ (resp. $-\infty$) the norm of $\cX$ restricted to smooth functions compactly supported in $I$ is controlled by the $W^{k,p}(I)$-norm and controls the $L^q(I)$-norm, for some $1\leq p,\,q\,\leq\infty$ such that $(p,q)\neq(1,\infty)$. 
\item In particular, if $\uU$ admits a limit $\uu_{\infty}$ at $+\infty$ or $-\infty$ such that $g'(\uu_\infty)>0$ then $\uU$ is spectrally unstable in $BUC^k(\RR\setminus D)\times \ell^\infty(D)$.
\end{enumerate}
\end{Proposition}
\begin{proof}
Since the difference is purely notational we only treat the case where the limit is at $+\infty$. We first observe that the non degeneracy of the profile ODE at $+\infty$ includes that $f'(\uu_{\infty})-\sigma\neq0$ and, since $g'(\uu_\infty)\neq0$, implies that $\uU-\uu_{\infty}$ and its derivatives up to the order $k+2$ converge exponentially fast to zero at $+\infty$.

Now pick some $\chi$ non zero, smooth and compactly supported. Let $\xi\in\RR$ and define, for $\eps>0$, $(y_d^{(\eps)})_{d\in D}=(0)_{d\in D}$ and 
\[
w^{(\eps)}\,:\ \RR\setminus D\to\CC,\,x\mapsto e^{-\frac{i\,\xi x}{f'(\uu_{\infty})-\sigma}}\,\chi\left(\eps\,x-\frac{1}{\eps}\right)\,.
\]
Then 
\[
\frac{\|((g'(\uu_\infty)+i\xi)-\cL)(w^{(\eps)},(y_d^{(\eps)}))\|_{\cX\times\cY}}{\|(w^{(\eps)},(y_d^{(\eps)}))\|_{\cX\times\cY}}
\stackrel{\eps\to0}{\longrightarrow}0
\]
follows from the fact that if $\eps$ is sufficiently small $w^{(\eps)}$ is supported in $I$ and
\begin{align*}
\|(g'(\uu_\infty)+i\xi)w^{(\eps)}+\d_x((f'(\uU)-\sigma)w^{(\eps)})-g'(\uU)w^{(\eps)}\|_{W^{k,p}(I)}
&\lesssim \eps^{1-\frac1p}\\
\|w^{(\eps)}\|_{L^q(I)}&\gtrsim \eps^{-\frac1q}. 
\end{align*}
Hence $(w^{(\eps_n)},(0)_{d\in D})_{n\in\NN}$ with $(\eps_n)_{n\in\NN}$ positive and converging to zero defines a Weyl sequence, and the proof is complete.
\end{proof}

\begin{Proposition}\label{P.nl-instability-infinity}
Let  $f\in\cC^{2}(\RR)$, $g\in\cC^{2}(\RR)$ and  $(\uU,\sigma,D)$ define a non-degenerate piecewise regular entropy-admissible traveling-wave solution to~\eqref{eq-u}. 
If $\uU$ admits a limit $\uu_{\infty}$ at $+\infty$ (resp. $-\infty$) such that $g'(\uu_\infty)>0$, then $\uU$ is nonlinearly unstable in the following sense. There exists $\delta>0$,  and a sequence $(u_n)_{n\in\NN}$ of piecewise regular entropy solutions to~\eqref{eq-u}, each defined on $[0,T_n]$ with $T_n>0$, such that for any $I\subset\RR$ neighborhood of $+\infty$ (resp. $-\infty$), one has for $n$ sufficiently large
\begin{enumerate}
\item $u_n(0,\cdot)-\uU$ is smooth and compactly supported in $I$, and for any $k\in\NN$ and $1\leq q\leq \infty$,
\[\Norm{u_n(0,\cdot)- \uU}_{W^{k,q}(I)}\to 0\quad \text{ as }\quad n\to \infty,\]
\item $(t,x)\mapsto u_n(t,\cdot+\sigma t)-\uU(\cdot)\in\cC^1([0,T_n]\times \RR)$, has compact support in $[0,T_n]\times I$, and for any $\cC^1$ shift $\psi:\RR\to\RR$ such that $(\Id_\RR+\psi)(D)=D$ and $\Norm{\d_x\psi}_{L^\infty(\RR)}\leq 1/2$, and any $1\leq p\leq \infty$,
\[\Norm{u_n(T_n,\,\cdot\,+\sigma\,T_n+\psi(\cdot))-\uU}_{L^p(\RR)}\geq \delta.\]
\end{enumerate}
\end{Proposition}

The presence of constraints on possible shifts $\psi$ is due to the fact that, when $1\leq p<\infty$, we measure in a norm that does not weight heavily wrong discontinuities, whereas we would like to keep the discussion localized to the connected component at hand. One may obtain various weaker but easier-to-read statements, for instance replacing the foregoing with: for any  $\psi:\RR\to\RR$ 
\[\Norm{u_n(T_n,\,\cdot\,+\sigma\,T_n+\psi(\cdot))-\uU}_{W^{k,p}(\RR)}
+\Norm{\d_x\psi}_{W^{k,p}(\RR)}\geq \delta\]
after having fixed some $(k,p)\in \NN\times[1,\infty]$ such that $k-1/p>0$ and enforced accordingly regularity on $f$ and $g$. The latter variants hinge on the fact that the finiteness of the left-hand side of the latter inequality implies $(\Id_\RR+\psi)(D)=D$ and that this left-hand side controls $\Norm{\d_x\psi}_{L^\infty(\RR)}$.

\begin{proof}
Since the difference is purely notational we only treat the case where the limit is at $+\infty$. 

To begin with, we show how the $\psi$-dependent conclusion may be derived from a $\psi$-independent inequality. Let $\psi:\RR\to\RR$ be  such that $(\Id_\RR+\psi)(D)=D$ and $\Norm{\d_x\psi}_{L^\infty(\RR)}\leq 1/2$. Then $\Id_\RR+\psi$ is strictly monotonic with derivative everywhere at least $1/2$. Since by assumption $D$ is discrete and possesses a maximum this implies that for any $d\in D$, $\psi(d)=0$ and the image of any connected component of $\RR\setminus D$ by $\Id_\RR+\psi$ is the same connected component. Let us denote $X_\infty$ the connected component neighboring $+\infty$. Since $\uU$ is either strictly monotonic or constant on $X_\infty$, we have either $\uu_\infty=\sup_{X_\infty}\uU$ or $\uu_\infty=\inf_{X_\infty}\uU$. We will assume in the following that the former case holds, and let the reader make the obvious modifications in the opposite case. Then if $T_n$ and $u_n$ are such that $x\mapsto u_n(T_n,\cdot+\sigma T_n)-\uU(\cdot)\in\cC^1(\RR)$ and has compact support in $X_\infty$, then for any $1\leq p\leq\infty$,
\begin{align*}
\Norm{\big(u_n(T_n,\,\cdot\,+\sigma\,T_n)-\uu_\infty\big)_+}_{L^p(X_\infty)}
&\leq\,2^{\frac1p}\,\Norm{\big(u_n(T_n,\,\cdot\,+\sigma\,T_n+\psi(\cdot))-\uu_\infty\big)_+}_{L^p(X_\infty)}\\
&\leq\,2^{\frac1p}\,\Norm{u_n(T_n,\,\cdot\,+\sigma\,T_n+\psi(\cdot))-\uU}_{L^p(\RR)}
\end{align*}
where $(\,\cdot\,)_+$ denotes the positive part. Thus in the following we can concentrate on proving for some $\delta'>0$ and a well-chosen family of solutions a lower bound 
\[
\Norm{\big(u_n(T_n,\,\cdot\,+\sigma\,T_n)-\uu_\infty\big)_+}_{L^p(X_\infty)}\geq \delta'
\]
on the comparison with $\uu_\infty$.

Let us build such a family of solutions, which we find more convenient to parameterize by $\eps>0$ instead of $n\in\NN$. We modify $\uU$ only in $X_\infty$ (defined as above) and use characteristics to study the effect of the perturbation. To do so, we need both to prevent the formation of new shocks and to ensure the confinement of the perturbation in\footnote{More accurately in $\bigcup_t\,\{t\}\times (X_\infty+\sigma\,t)$.} $X_\infty$ on a time interval sufficiently long for the perturbation to grow.
We choose $\delta_0>0$ sufficiently small to enforce 
\begin{align*}
e^{\frac{\delta_0}{g'(\uu_\infty)}\,\max_{|u-\uu_\infty|\leq 8\,\delta_0} |g''(u)|}
\,&\leq 2\,.
\end{align*}
Pick $\chi:\RR\to\RR$ nonnegative and non zero, smooth and compactly supported in $(0,\infty)$ and set for $\eps>0$, $u_\eps(0,\cdot)=\uU+\eps^2\chi(\eps\,\cdot-\eps^{-1})$. For $\eps>0$ and $x\in (\eps^{-2},+\infty) $, we define $v_\eps(\cdot,x)$  and $X_\eps(\cdot,x)$ by the initial data $v_\eps(0,x)=u_\eps(0,x)$ and $X_\eps(0,x)=x$, and the differential equations
\[
\partial_t v_\eps(t,x)=g\big(v_\eps(t,x)\big) \quad  \text{ and } \quad\partial_t X_\eps(t,x)=f'\big(v_\eps(t,x)\big).
\]
By a continuity argument, we have for $\eps>0$ sufficiently small and any 
\[0\leq t\leq \frac{1}{g'(\uu_\infty)}
\ln\left(\frac{2\delta_0}{\eps^2\Norm{\chi}_{L^\infty}}\right)
\eqdef T_\eps,
\]
 that, for any $0\leq s\leq t$, 
 \[ \|v_\eps(s,\cdot)-\uu_{\infty}\|_{L^\infty}\leq 2\,e^{g'(\uu_\infty)\,s}\,\|v_\eps(0,\cdot)-\uu_{\infty}\|_{L^\infty}\leq 8\delta_0\,\]
and hence for any $x\geq\eps^{-2}$,
\begin{align*}
\frac12\,|v_\eps(0,x)-\uu_{\infty}|\,e^{g'(\uu_\infty)\,t}
\leq\,|v_\eps(t,x)-\uu_{\infty}|\leq 2\,|v_\eps(0,x)-\uu_{\infty}| e^{g'(\uu_\infty)\,t}\,.
\end{align*}
Moreover, differentiating spatially the defining differential equations and lowering  $\eps$ if necessary,
\begin{align*}
|\d_xv_\eps(t,x)|&\leq 2\,|\d_x v_\eps(0,x)|  e^{g'(\uu_\infty) t}\,,\\
\d_xX_\eps(t,x)&\geq 1-2\,|\partial_x v_\eps(0,x)|\,\frac{e^{g'(\uu_\infty) t}}{g'(\uu_\infty)}\,\max_{|u-\uu_\infty|\leq 8\,\delta_0} |f''(u)|\,\geq\frac12\,,\\
X_\eps(t,x)-\sigma t&\geq \eps^{-2}\, -\, \big(\sigma-\inf_{|u-\uu_\infty|\leq 8\,\delta_0} f'(u)\big)\,t \,>\, \sup(D)\,.
\end{align*}
Hence  for $\eps>0$ sufficiently small a solution $u_\eps$ to~\eqref{eq-u} on $[0,T_\eps]\times\RR$ is obtained by setting, for $0\leq t\leq T_\eps$ and $x\in\RR$,
\[
u_\eps(t,x)\,=\,
\begin{cases}\uU(x-\sigma t)&\qquad\textrm{if } x\leq X_\eps(t,\eps^{-2})\\
v_\eps(t,X_\eps(t,\cdot)^{-1}(x))
&\qquad\textrm{if } x\geq X_\eps(t,\eps^{-2})
\end{cases}
\]
and it satisfies
for any $k\in\NN$ and $1\leq q\leq \infty$
\begin{align*}
\Norm{\d_x^k(u_\eps(0,\cdot)-\uU)}_{L^{q}(\RR)}
&\,=\,\eps^{k+2-\frac1p}\,\Norm{\d_x^k\chi}_{L^{q}(\RR)}\,,
\end{align*}
and for any $1\leq p\leq \infty$
\begin{align*}
\Norm{\big(u_\eps(T_\eps,\cdot+\sigma T_\eps)-\uu_\infty\big)_+}_{L^{p}(\RR)}
&\,\geq\,
\frac{1}{2^{\frac1p}}
\Norm{\big(v_\eps(T_\eps,\cdot)-\uu_\infty\big)_+}_{L^{p}(\eps^{-2},+\infty)}
\\
&\,\geq\,
\frac{1}{2^{\frac1p}}\,
\frac{2\delta_0}{\eps^2\Norm{\chi}_{L^{\infty}(\RR)}} \Norm{\big(v_\eps(0,\cdot)-\uu_\infty\big)_+}_{L^{p}(\eps^{-2},+\infty)} \,.
\end{align*}
Lowering $\eps$ if necessary, one has
\begin{align*} \Norm{\big(v_\eps(0,\cdot)-\uu_\infty\big)_+}_{L^{p}(\eps^{-2},+\infty)}  &\,\geq\, \Norm{\eps^2\chi(\eps\,\cdot-\eps^{-1})}_{L^{p}(\eps^{-2},+\infty)}  - \Norm{\uU-\uu_\infty}_{L^{p}(\eps^{-2},+\infty)} \\
&\,\geq\, \frac12\eps^{2-1/p}\Norm{\chi}_{L^p(\RR)}
\,\geq\, \frac12\eps^{2}\Norm{\chi}_{L^p(\RR)}\,.
\end{align*}
This achieves the proof with
\[
\delta'\,=\,\delta_0\,\min_{1\leq p\leq\infty}\left(\frac{1}{2^{\frac1p}}\,
\frac{\Norm{\chi}_{L^{p}(\RR)}}{\Norm{\chi}_{L^{\infty}(\RR)}}\right)\,.
\]
\end{proof}

Note that in the latter our instability result is in some sense constructive and positive. We prove that there does exist a family of solutions with explicit initial data and explicit guaranteed time of existence whose growth encodes instability. A negative form showing that for some initial data there do not exist solutions globally defined and suitably small would be somewhat simpler to prove but less instructive.

\subsection{Instabilities at characteristic points}

Instabilities due to characteristic points seem to be pointed out here for the first time. This is probably partly due to the fact that they are of wave-breaking type. They manifest themselves in topologies encoding a sufficient amount of smoothness, and remain harmless in the $L^\infty$ topology. Moreover the condition yielding instability is somewhat counter-intuitive as it amounts to dissipativity of the source term near a characteristic value.

\begin{Proposition}\label{P.spectral-instability-characteristic}
Let $k\in\NN$ and $f\in\cC^{k+2}(\RR)$, $g\in\cC^{k+1}(\RR)$ and $(\uU,\sigma,D)$ define a non-degenerate piecewise regular entropy-admissible traveling-wave solution to~\eqref{eq-u}. 
\begin{enumerate}
\item If $x_\star\in\RR\setminus D$ is a characteristic point, that is $\uU(x_\star)=\uu_\star$ with $f'(\uu_\star)=\sigma$, then $-g'(\uu_\star)\,k$ belongs to the $\cX\times\cY$-spectrum of the linearization about $\uU$ provided that $\delta_{x_\star}$, $\cdots$, $\delta_{x_\star}^{(k)}$ act continuously on $\cX$. 
\item In particular, when $k\geq1$, if at a characteristic point $x_\star\in\RR\setminus D$, $\uu_\star=\uU(x_\star)$ is such that $g'(\uu_\star)<0$ then $\uU$ is spectrally unstable  in $BUC^k(\RR\setminus D)\times \ell^\infty(D)$.
\end{enumerate}
\end{Proposition}

Note that the stronger the regularity encoded in $\cX$ the stronger the instability proved in the foregoing proposition. In the limit case where the norm of $\cX$ would control an infinite number of derivatives at $x_\star$ the proposition yields ill-posedness (at the linear level).

\begin{proof}
It is sufficient to prove that there exists $\tw$ a non trivial combination of $\delta_{x_\star}$, $\cdots$, $\delta_{x_\star}^{(k)}$ such that
\[
-g'(\uu_\star)k\,\tw-(f'(\uU)-\sigma)\d_x\tw-g'(\uU)\tw\,=\,0
\]
since then $(\tw,(0)_{d\in D})$ provides a nontrivial element of the kernel of the adjoint of $-g'(\uu_\star)k-\cL$.

Now the claim follows recursively from the fact that 
\[
-(f'(\uU)-\sigma)\delta_{x_\star}^{(1)}-g'(\uU)\delta_{x_\star}
\,=\,0\,
\]
and, for any $\ell\in\NN$, $1\leq \ell\leq k$,
\[
-(f'(\uU)-\sigma)\delta_{x_\star}^{(\ell+1)}-g'(\uU)\delta_{x_\star}^{(\ell)}
\,\in\,\ell\,g'(\uu_\star)\,\delta_{x_\star}^{(\ell)}
\,+\,\Span(\{\delta_{x_\star},\cdots,\delta_{x_\star}^{(\ell-1)}\})\,.
\]
\end{proof}

\begin{Proposition}\label{P.nl-instability-characteristic}
Let  $f\in\cC^{2}(\RR)$, $g\in\cC^{1}(\RR)$ and $(\uU,\sigma,D)$ define a non-degenerate piecewise regular entropy-admissible traveling-wave solution to~\eqref{eq-u}, and assume that there exists $x_\star\in\RR\setminus D$ such that $g(\uu_\star)=0$ and $g'(\uu_\star)<0$. Then $\uU$ is nonlinearly unstable in the following sense. There exists a sequence $(u_n)_{n\in\NN}$ of piecewise regular entropy solutions to~\eqref{eq-u}, each defined on $[0,T_n)$ with $T_n>0$, such that for any $I\subset \RR$ neighborhood of $x_\star$, one has for $n$ sufficiently large
\begin{enumerate}
\item $u_n(0,\cdot)-\uU$ is smooth and compactly supported in $I$, and for any $k\in\NN$ and $1\leq q\leq \infty$, 
\[\Norm{u_n(0,\cdot)- \uU}_{W^{k,q}(I)}\to 0\quad \text{ as} \quad n\to \infty,\]
\item $(t,x)\mapsto u_n(t,\cdot+\sigma t)-\uU\in\cC^1([0,T_n)\times \RR)\cap L^\infty([0,T_n)\times \RR)$, has support in $[0,T_n)\times I$, and
\[ \Norm{\d_x u_n(t,\cdot+\sigma t)}_{L^{\infty}(I)}\to \infty\quad  \text { as $t\to T_n$}.\]
\end{enumerate}
\end{Proposition}
\begin{proof}
Again we find more convenient to parameterize our family of solutions by $\eps>0$ instead of $n\in\NN$, modify $\uU$ only in a neighborhood of $x_\star$ and use characteristics to study the effect of the perturbation. 

We choose $\delta_0>0$ sufficiently small to enforce 
\begin{align*}
\min_{|u-\uu_\star|\leq 2\,\delta_0} |f''(u)-f''(\uu_\star)|&\leq
\frac12|f''(\uu_\star)|\,,&
\min_{|u-\uu_\star|\leq 2\,\delta_0} |g'(u)-g'(\uu_\star)|&\leq
\frac12|g'(\uu_\star)|\,.
\end{align*}
Pick $\chi:\RR\to\RR$ smooth and compactly supported in $(-1,1)$ such that $\chi(0)=0$ and $\chi'(0)>0$ and set for $\eps>0$, $\eta_\eps\eqdef |\ln \eps|^{-1}$, $I_\eps\eqdef [x_\star-\eta_\eps,x_\star+\eta_\eps]$, and \[
u_\eps(0,\cdot)=\uU+\eps\,\uU'(x_\star)\,\chi\left(\frac{\cdot-x_\star}{\eta_\eps}\right)\,.
\] We may take $\eps$ sufficiently small to ensure that for any $x\in I_\eps$, $|u_\eps(0,x)-\uu_\star|\leq \,\delta_0$ and $\d_xu_\eps(0,x)$ has the sign of $\uU'(x_\star)$. 

Now, for $\eps>0$ and $x\in I_\eps $, we define $v_\eps(\cdot,x)$ and $X_\eps(\cdot,x)$ by the initial data $v_\eps(0,x)=u_\eps(0,x)$ and $X_\eps(0,x)=x$, and the differential equations
\[
\partial_t v_\eps(t,x)=g\big(v_\eps(t,x)\big) \quad  \text{ and } \quad\partial_t X_\eps(t,x)=f'\big(v_\eps(t,x)\big).
\]
For $\eps>0$ sufficiently small, a continuity argument shows that $v_\eps$ and $X_\eps$ are globally defined, and that 
\begin{itemize}
\item for any $t\geq0$ and $x\in I_\eps$, $v_\eps(t,x)-\uu_\star$ has the sign of $\uU'(x_\star)(x-x_\star)$, $\d_x v_\eps(t,x)$ has the sign of $\uU'(x_\star)$ and
\begin{align}
e^{\frac32\,g'(\uu_\star)\,t}\,|v_\eps(0,x)-\uu_\star|\leq\,|v_\eps(t,x)-\uu_\star|&\leq\,e^{\frac12\,g'(\uu_\star)\,t}\,|v_\eps(0,x)-\uu_\star|\,,
\nonumber\\\label{e:wb1}
e^{\frac32\,g'(\uu_\star)\,t}\,|\d_xv_\eps(0,x)|\leq\,|\d_xv_\eps(t,x)|&\leq\,e^{\frac12\,g'(\uu_\star)\,t}\,|\d_xv_\eps(0,x)|\,;
\end{align}
\item for any $t\geq 0$, $\d_xv(t,x_\star)=\d_xv_\eps(0,x)\,e^{g'(\uu_\star)\,t}$, $X_\eps(t,x_\star)=x_\star+\sigma\,t$ and
\begin{align}
\d_xX_\eps(t,x_\star)&=
1+\frac{f''(\uu_\star)}{g'(\uu_\star)} \partial_x v_\eps(0,x_\star) \left( e^{g'(\uu_\star)t}-1\right)
\,=\,e^{g'(\uu_\star)t}\left(1+\frac{\eps}{\eta_\eps}\chi'(0)\right)
-\frac{\eps}{\eta_\eps}\chi'(0)\nonumber\\\label{e:wb0}
&\stackrel{t\to+\infty}{\longrightarrow}
-\frac{\eps}{\eta_\eps}\chi'(0)\,;
\end{align}
\item $t \mapsto X_\eps(t,x_\star-\eta_\eps)-(x_\star+\sigma\,t)$ is increasing and $t \mapsto X_\eps(t,x_\star+\eta_\eps)-(x_\star+\sigma\,t)$ is decreasing.
\end{itemize}

This shows that for $\eps>0$ sufficiently small, with
\[
T_\eps\eqdef \sup\{\,t\geq0\,;\,\forall s\in[0,t],\,\forall x\in I_\eps, \d_xX_\eps(s,x)>0\,\}\,,
\]
a solution $u_\eps$ to~\eqref{eq-u} on $[0,T_\eps)\times\RR$ is obtained by setting, for $0\leq t< T_\eps$ and $x\in\RR$,
\[
u_\eps(t,x)\,=\,
\begin{cases}\uU(x-\sigma t)&\qquad\textrm{if }\quad x\leq X_\eps(t,x_\star-\eta_\eps)\quad\textrm{or}\quad
x\geq X_\eps(t,x_\star+\eta_\eps)\\
v_\eps(t,X_\eps(t,\cdot)^{-1}(x))
&\qquad\textrm{if }\quad X_\eps(t,x_\star-\eta_\eps)\leq x\leq X_\eps(t,x_\star+\eta_\eps)
\end{cases}
\]
and that it satisfies for any $k\in\NN$ and $1\leq q\leq \infty$
\begin{align*}
\Norm{\d_x^k(u_\eps(0,\cdot)-\uU)}_{L^{q}(\RR)}
&\,=\,\frac{\eps}{|\ln(\eps)|^{k-\frac1q}}\,|\uU'(x_\star)|\,\Norm{\d_x^k\chi}_{L^{q}(\RR)}\,,
\end{align*}
and for any $0\leq t<T_\eps$ and $x\in I_\eps$
\begin{equation}\label{e:wb2}
\d_xu_\eps(t,X_\eps(t,x))
\,=\,\frac{\d_xv_\eps(t,x)}{\d_xX_\eps(t,x)}\,.
\end{equation}
For any $\eps>0$ sufficiently small, it follows from the continuity of $\d_xX_\eps$ that $T_\eps>0$ and from~\eqref{e:wb0} that $T_\eps<+\infty$. Combined with~\eqref{e:wb1} and~\eqref{e:wb2} this yields the desired blow-up.
\end{proof}

\subsection{Instabilities of shock positions}

For instabilities due to discontinuities it is crucial to use adapted notions of stability, both at spectral and nonlinear levels. More, within the framework already introduced it is trivial to detect spectral instabilities.

\begin{Proposition}\label{P.spectral-instability-shock}
Let $f\in\cC^{2}(\RR)$, $g\in\cC^{1}(\RR)$ and $(\uU,\sigma,D)$ define a non-degenerate piecewise regular entropy-admissible traveling-wave solution to~\eqref{eq-u}. 
\begin{enumerate}
\item For any $d_0\in D$, $\frac{[\,g(\uU)\,]_{d_0}}{[\,\uU\,]_{d_0}}$ belongs to the $\cX\times\cY$-spectrum of the linearization about $\uU$. 
\item In particular, if for some $d_0\in D$, $\frac{[\,g(\uU)\,]_{d_0}}{[\,\uU\,]_{d_0}}>0$ then $\uU$ is spectrally unstable in $BUC^k({\RR\setminus D})\times \ell^\infty(D)$ for any $k\in \NN$.
\end{enumerate}
\end{Proposition}

\begin{proof}
One checks readily that $(w,(y_d)_{d\in D})=(0,(\delta_{d,d_0})_{d\in D})$ provides an eigenvector of $\cL$ for the eigenvalue $\frac{[\,g(\uU)\,]_{d_0}}{[\,\uU\,]_{d_0}}$.
\end{proof}

\begin{Proposition}\label{P.nl-instability-shock}
Let $k\in\NN^\star$, $f\in\cC^{k+1}(\RR)$, $g\in\cC^{k}(\RR)$ and $(\uU,\sigma,D)$ define a non-degenerate piecewise regular entropy-admissible traveling-wave solution to~\eqref{eq-u} satisfying for some  $d_0\in D$, $\frac{[\,g(\uU)\,]_{d_0}}{[\,\uU\,]_{d_0}}>0$. Then $\uU$ is nonlinearly unstable in the following sense. There exist $\delta>0$, a sequence $(u_n)_{n\in\NN}$ of piecewise regular entropy solutions to~\eqref{eq-u}, each defined on $[0,T_n]$ with $T_n>0$, and smooth phase shifts $(\psi_n)_{n\in\NN}$, each defined on $[0,T_n]$, such that $t\mapsto d_0+\sigma\,t+\psi_n(t,d_0)$ is constant equal to $d_0$, $(t,x)\mapsto u_n(t,x+\sigma t+\psi_n(t,x))-\uU(x)\in\cC^1([0,T_n]\times (\RR\setminus \{d_0\}))$ and for any $I\subset \RR$ neighborhood of $d_0$, one has for $n$ sufficiently large
\begin{enumerate}
\item $\psi_n(0,\cdot)$ and $u_n(0,\cdot+\psi_n(0,\cdot))-\uU$ are compactly supported in $I$, and for any $1\leq q\leq \infty$, 
\[\Norm{u_n(0,\cdot+\psi_n(0,\cdot))- \uU}_{W^{k,q}(I\setminus\{d_0\})}
+\Norm{\psi_n(0,\cdot)}_{W^{k+1,q}(I)}\to 0\quad \text{ as} \quad n\to \infty,\]
\item for any $\cC^1$ shift $\psi:\RR\to\RR$ such that $\psi\id{D}=\psi_n(T_n,\cdot)\id{D}$ and $\Norm{\d_x\psi}_{L^\infty(\RR)}\leq 1/2$, and any $1\leq p\leq \infty$,
\[\Norm{u_n(T_n,\,\cdot\,+\sigma\,T_n+\psi(\cdot))-\uU}_{L^p(\RR\setminus \{d_0\})}
\geq \delta.\] 
\end{enumerate}
\end{Proposition}
\begin{proof} 
To begin with introduce $\uU_\pm$ the maximal solutions to $\uU_\pm'=F_\sigma(\uU_\pm)$, where $F_\sigma$ is defined in~\eqref{def-F}, with $\uU_\pm(d_0)=\uU(d_0^\pm)$.
There exists $\eta>0$ such that with $I_\eta\eqdef [d_0-32\eta,d_0+32\eta]$ $\uU_-,\uU_+\in\cC^1(I_\eta)$ and for any $x\in I_\eta$, $\uU(x)=\uU_-(x)$ if $x<d_0$ and $\uU(x)=\uU_+(x)$ if $x>d_0$. To analyze the Rankine-Hugoniot condition near $d_0$, lowering $\eta$ if necessary, we also introduce the slope function
\[
s_0\,:\quad I_\eta\to\RR\,,\qquad
x\mapsto \frac{f(\uU_+(x))-f(\uU_-(x))}{\uU_+(x)-\uU_-(x)}
\]
and observe that 
\[
s_0(d_0)=\sigma\qquad\textrm{and}\qquad
s_0'(d_0)=\frac{[\,g(\uU)\,]_{d_0}}{[\,\uU\,]_{d_0}}\,.
\]
By continuity, lowering $\eta$ again we may also enforce that
\begin{align*}
\max_{I_\eta} |\uU'_--\uU'(d_0^-)|&\leq
\frac12|\uU'(d_0^-)|\,,&
e^{32\eta\,\frac{\max_{I_\eta}|s_0''|}{|s_0' (d_0)|}}&\leq 2\,,
\end{align*}
and for any $x\in I_\eta$ and $v$ between $\uU_-(x)$ and $\uU_+(x)$,
\[
\int_0^1 f'(\uU_-(x)+s\,(v-\uU_-(x)))\,\dd s
\,>\,
\int_0^1 f'(\uU_+(x)+s\,(v-\uU_+(x)))\,\dd s\,.\]

Once again we parameterize perturbed solutions with $\eps>0$. For $\eps\in(0,\eta)$, starting from $\psi_{0,\eps}(0)=\eps$ with 
\[
T_\eps
\eqdef 
\frac{1}{\frac{[\,g(\uU)\,]_{d_0}}{[\,\uU\,]_{d_0}}}
\ln\left(\frac{8\eta}{\eps}\right)
\]
we may solve on $[0,T_\eps]$ the Rankine-Hugoniot equation 
\[
\sigma+\psi_{0,\eps}'=s_0(d_0 +\psi_{0,\eps}(\cdot))
\]
and obtain that for any $t\in[0,T_\eps]$
\[
\frac12\,\eps\,e^{\frac{[\,g(\uU)\,]_{d_0}}{[\,\uU\,]_{d_0}}\,t}\leq\psi_{0,\eps}(t)\leq 2\,\eps\,e^{\frac{[\,g(\uU)\,]_{d_0}}{[\,\uU\,]_{d_0}}\,t}\,.
\]
Then for $\eps\in(0,\eta)$, a solution $u_\eps$ to~\eqref{eq-u} on $[0,T_\eps]\times\RR$ is obtained by setting, for $0\leq t\leq T_\eps$ and $x\in\RR$,
\[
u_{\eps}(t,x) = \begin{cases}
\uU_-(x-\sigma t) & \text{if }\ x-\sigma t\in I_\eta\qquad\text{and}\qquad x-\sigma t<d_0+\psi_{0,\eps}(t)\,,\\
\uU_+(x-\sigma t) & \text{if }\ x-\sigma t\in I_\eta\qquad\text{and}\qquad x-\sigma t>d_0+\psi_{0,\eps}(t)\,,\\
\uU(x-\sigma t) & \text{if }\ x-\sigma t\in \RR\setminus I_\eta\,,
\end{cases}
\]
and, picking $\chi:\RR\to\RR$ smooth and compactly supported such that $\chi(0)=1$, a suitable phase shift is obtained by setting for $0\leq t\leq T_\eps$ and $x\in\RR$,
\[
\psi_\eps(t,x)\,=\,\psi_{0,\eps}(t)\,\chi\left(|\ln(\eps)|\,(x-d_0)\right)
\]
provided that $\eps$ is sufficiently small.

Now we turn to the instability bound. Fix $\eps\in(0,\eta)$ sufficiently small. Let $\psi:\RR\to\RR$ be a $\cC^1$ shift such that $\psi\id{D}=\psi_\eps(T_\eps,\cdot)\id{D}$ and $\Norm{\d_x\psi}_{L^\infty(\RR)}\leq 1/2$. Then since, for $x\in[d_0-2\eta,d_0)$,
\[
x+\psi(x)\,\geq\,d_0+\psi(d_0)+\frac32\,(x-d_0)
\,\geq\,d_0\,+\,\eta\,,
\]
we have, for $1\leq p\leq\infty$
\[
\Norm{u_\eps(T_\eps,\,\cdot\,+\sigma\,T_\eps+\psi(\cdot))-\uU}_{L^p([d_0-2\eta,d_0))}
\,\geq\,(2\,\eta)^{\frac1p}\,\frac{\eta}{2}\,|\uU'(d_0^-)|\,.
\]
This achieves the proof with
\[
\delta\eqdef \frac{\eta}{2}\,|\uU'(d_0^-)|\,\min(\{1,2\,\eta\})\,.
\]
\end{proof}

\subsection{Specialization to periodic traveling waves}

We conclude this section by briefly addressing on a specific example the following question: when $\uU$ possesses some global symmetry, is it possible to restore some stability by constraining perturbations to share the same or a related symmetry? We discuss here the case where the symmetry is periodicity, that is, invariance under translations by a discrete set of periods.

A direct inspection of the proofs of Propositions~\ref{P.spectral-instability-characteristic},~\ref{P.nl-instability-characteristic},~\ref{P.spectral-instability-shock} and~\ref{P.nl-instability-shock} show that one may modify them to enforce that perturbations share the same periodicity as the background wave. At the spectral level a more relevant and somewhat more intricate question is: for which Floquet parameter does the instability occur?

To be more explicit, following~\cite[Section~4.2]{JNRYZ} we introduce the relevant Bloch transform/inverse Fourier series. Let $(\uU,\sigma,D)$ define a non-degenerate piecewise regular entropy-admissible traveling-wave solution to~\eqref{eq-u}, that is periodic with period $X_0\in(0,+\infty)$. Choose $x_0\in\RR$ a point of continuity of $\uU$ and set $D_0\eqdef D\cap(x_0,x_0+X_0)$ so that $D=D_0+X_0\,\ZZ$. For functions $w$ on $\RR$ and $y$ on $D$, we introduce representations
\begin{align*}
w(x)&=\int_{-\pi/X_0}^{\pi/X_0} e^{i\xi x}\check w(\xi,x) \dd\xi\,,
\quad x\in\RR\,,\\
y_{jX_0+d}&=\int_{-\pi/X_0}^{\pi/X_0} e^{i\xi jX_0}\check y_d(\xi) \dd\xi\,,
\quad (j,d)\in \ZZ\times D_0\,,
\end{align*}
where each $\check w(\xi,\cdot)$ is $X_0$-periodic. For sufficiently smooth $w$ and sufficiently localized $(y_d)_d$, the former transforms are defined pointwise by
\begin{align*}
\check{w}(\xi,x)&\eqdef \sum_{k\in\ZZ}e^{i\frac{2k\pi}{X_0}x}\,\widehat{w}\big(\tfrac{2\pi k}{X_0}+\xi\big)=
\frac{X_0}{2\pi}\sum_{k\in\ZZ}e^{-i\xi(x+kX_0)}w(x+kX_0)\,,\\
\check{y}_d(\xi)&\eqdef 
\frac{X_0}{2\pi}\sum_{j\in\ZZ}e^{-ijX_0\,\xi}\,y_{jX_0+d}\,,
\end{align*}
where $\ \widehat{\,\,}\ $ denotes (a suitable choice of) the Fourier transform, given by
\[
\widehat{w}(\xi)\,=\,\frac{1}{2\pi}\int_{\RR} e^{-i\xi x}w(x) \dd x\,,
\quad \xi\in\RR\,.
\]
General definitions follow by a density argument in $L^2$ (respectively, $\ell^2$) based on Parseval identities
$$
\|\check{w}\|_{L^2\left(\left(-\frac{\pi}{X_0},\frac{\pi}{X_0}\right);L^2((x_0,x_0+X_0))\right)}
=\sqrt{\frac{X_0}{2\pi}}\|w\|_{L^2(\RR)}
\,,
\quad
\|\check{y}\|_{L^2\left(\left(-\frac{\pi}{X_0},\frac{\pi}{X_0}\right);\ell^2(D_0)\right)}=\sqrt{\frac{X_0}{2\pi}}\|(y_d)_d\|_{\ell^2(\ZZ)}\,.
$$
In particular the Bloch transform identifies $L^2(\RR)$ with $L^2\left(\left(-\pi/X_0,\pi/X_0\right);L^2((x_0,x_0+X_0))\right)$,  and this may be extended, for $k\in\NN$, to identify $H^k(\RR\setminus D)$ with $L^2\left(\left(-\frac{\pi}{X_0},\frac{\pi}{X_0}\right);H_{\rm per}^k((x_0,x_0+X_0)\setminus D_0)\right)$ by
observing
$$
\|(\d_x+i\xi)^\ell\check{w}\|_{L^2\left(\left(-\frac{\pi}{X_0},\frac{\pi}{X_0}\right);L^2((x_0,x_0+X_0))\right)}
=\sqrt{\frac{X_0}{2\pi}}\,\|\d_x^\ell w\|_{L^2(\RR)}\,,\qquad \ell\in\NN\,,
$$
where $H_{\rm per}^k((x_0,x_0+X_0)\setminus D_0)$ is the $H^k((x_0,x_0+X_0)\setminus D_0)$-closure of $X_0$-periodic functions of $\cC^\infty(\RR\setminus D)$, hence is a set of $H^k((x_0,x_0+X_0)\setminus D_0)$ functions satisfying suitable periodic boundary conditions as soon as $k>1/2$.

Applying the above transformations to the resolvent problems for $\cL$ diagonalize them into single-cell problems parameterized by the Floquet exponent $\xi$. As a consequence, the spectrum of $\cL$ on $H^k(\RR\setminus D)\times\ell^2(D)$ (for some $k\in\NN$) is seen to coincide with union over $\xi\in [-\pi/X_0,\pi/X_0]$ of the spectra of $\cL_\xi$ on $H_{\rm per}^k((x_0,x_0+X_0)\setminus D_0)\times\ell^2(D_0)$, defined with maximal domain by 
\begin{align*}
\cL_\xi&(w,(y_d)_{d\in D_0})\\
&\eqdef \left(-(\d_x+i\xi)((f'(\uU)-\sigma)w)+g'(\uU)w,
\left(y_d\frac{[\,(f'(\uU)-\sigma)\uU'\,]_d}{[\,\uU\,]_d}\,+\,\frac{[\,(f'(\uU)-\sigma)w\,]_d}{[\,\uU\,]_d}\right)_{d\in D_0}\right)\,.
\end{align*}

\begin{Proposition}\label{P.instability-periodic}
Let $k\in\NN$ and $f\in\cC^{k+2}(\RR)$, $g\in\cC^{k+1}(\RR)$ and $(\uU,\sigma,D)$ define a  non-degenerate piecewise regular entropy-admissible traveling-wave solution to~\eqref{eq-u} of period $X_0>0$. 
\begin{enumerate}
\item If $x_\star\in\RR\setminus D$ is a characteristic point then every $-g'(\uu_\star)\,\ell$ with $\ell\in\NN$ and $0\leq\ell\leq k-1$ belongs to the $H_{\rm per}^k((x_0,x_0+X_0)\setminus D_0)\times\ell^2(D_0)$-spectrum of the Bloch symbol $\cL_\xi$ of the linearization about $\uU$ at any Floquet $\xi\in [-\pi/X_0,\pi/X_0]$.
\item For any $d_0\in D$, $\frac{[\,g(\uU)\,]_{d_0}}{[\,\uU\,]_{d_0}}$ belongs to the $H_{\rm per}^k((x_0,x_0+X_0)\setminus D_0)\times\ell^2(D_0)$-spectrum of the Bloch symbol $\cL_\xi$ of the linearization about $\uU$ at any Floquet $\xi\in [-\pi/X_0,\pi/X_0]$.
\item Nonlinear instabilities of Propositions~\ref{P.nl-instability-characteristic} and~\ref{P.nl-instability-shock} may also be obtained under perturbations of period $X_0$.
\end{enumerate}
\end{Proposition}

Since it is a relatively straightforward adaptation of the proofs expounded here-above we leave the proof\footnote{Including the precise statements of the periodic versions of the nonlinear instabilities.} of the foregoing proposition to the reader. We point out however that, in the spectral part related to characteristic points, eigenvectors of the adjoint operators do depend on the Floquet exponent $\xi$ whereas the corresponding eigenvalues do not.

\bigskip

The instability mechanisms identified in the foregoing subsections leave hardly any room for a stabilization by the choice of a suitable topology (in the framework of piecewise smooth solutions). The only reasonable exception we see is the stabilization of instabilities at infinities by the introduction of spatial weights, essentially as monostable fronts of reaction-diffusion systems are proved to be stable in suitable exponentially weighted spaces. Note however that, though quite common, the latter choice breaks the invariance by spatial translations of the original problem.

\section{Classification of traveling waves}

The remaining part of the paper is essentially devoted to stability results. Yet, first, we anticipate those and pause to combine them with the instability results obtained so far so as to obtain a full classification of stable non-degenerate piecewise regular entropy-admissible traveling-wave solutions to~\eqref{eq-u} under the following generic assumption on $(f,g)$.

\begin{Assumption}\label{A.generic}
For any $(u,v)\in \RR^2$, if 
\[
g(u)=0\,,\quad g(v)=0\,,\quad 
g'(u)\geq0\,,\quad g'(v)\geq0\,,\quad
\textrm{and}\quad
f'(u)=f'(v)\,,
\]
then $u=v$.
\end{Assumption}

We stress that, for any $(k,\ell)\in\NN^\star$, the set where Assumption~\ref{A.generic} holds is indeed a dense $G_\delta$ set in $\cC^k(\RR)\times\cC^\ell(\RR)$. To check that it is a $G_\delta$ set, we observe that it coincides with $\cap_{n\in\NN^\star} \Phi_n^{-1}((0,+\infty))$, where for $n\in\NN^\star$,
\[
\Phi_n\,:\quad (f,g)\,\longrightarrow\,
\min_{\substack{(u,v)\\|u|\leq n\,,\ |v|\leq n\,,\\|u-v|\geq \frac1n}}
\left(|g(u)|+|g(v)|+(g'(u))_-+(g'(u))_-+(g'(v))_++|f'(u)-f'(v)|\right)
\]
where $(\,\cdot\,)_-$ denotes the negative part. The proof of the density claim is easy to derive from the density of polynomial functions and is left to the reader.

Note that Assumption~\ref{A.generic} is designed to ensure, in view of Propositions~\ref{P.spectral-instability-characteristic} and~\ref{P.nl-instability-characteristic}, that stable non-degenerate waves take at most one characteristic value $\uu_\star$.

\medskip

The following Theorem proves the classification announced in the introduction and illustrated in Figure~\ref{F.classes}.

\begin{Theorem}\label{T.classification}
Under Assumption~\ref{A.generic}, for any $(\uU,\sigma,D)$ defining a  non-degenerate piecewise regular entropy-admissible traveling-wave solution to~\eqref{eq-u}, spectral and nonlinear stability coincide and they are equivalent to $\uU$ taking one of the following forms  
\begin{enumerate}
\item $\uU$ is constant with value $\uu$ such that $g'(\uu)<0$;
\item $\uU$ is a Riemann shock --- that is, for some $x_0\in\RR$ and $(\uu_-,\uu_+)\in\RR^2$ with $\uu_-\neq\uu_+$, $\uU$ is constant equal to $\uu_-$ on $(-\infty,x_0)$ and constant equal to $\uu_+$ on $(x_0,+\infty)$, --- with values satisfying $g'(\uu_-)<0$ and $g'(\uu_+)<0$;
\item $\uU$ is a continuous front --- that is, $\uU$ is continuous with distinct limits $\uu_{-\infty}$ at $-\infty$, $\uu_{+\infty}$ at $+\infty$, $\uu_{-\infty}\neq\uu_{+\infty}$, --- with endstates satisfying $g'(\uu_{-\infty})<0$ and $g'(\uu_{+\infty})<0$;
\item $\uU$ has only one discontinuity at a point $d_0$, joining a constant part to a non-constant part, with endstates $\uu_{\pm\infty}$ at $\pm\infty$ satisfying $g'(\uu_{\pm\infty})<0$ and jumps satisfying $\frac{[g(\uU)]_{d_0}}{[\uU]_{d_0}}<0$;
\item $\uU$ has two discontinuities at points $d_-<d_+$, being constant equal to $\uu_-$ on $(-\infty,d_-)$, constant equal to $\uu_+$ on $(d_+,+\infty)$, and non constant on $(d_-,d_+)$, with endstates satisfying $g'(\uu_{\pm})<0$ and jumps satisfying $\frac{[g(\uU)]_{d_{\pm}}}{[\uU]_{d_{\pm}}}<0$.
\end{enumerate}
\end{Theorem}

\begin{proof}
As announced, the proof of the stability part of the statement is postponed to later sections. We only prove here that non-degenerate piecewise regular entropy-admissible waves that enter in none of the above categories satisfy one of the instability criteria of the foregoing sections. 

Let $(\uU,\sigma,D)$ define a  non-degenerate piecewise regular entropy-admissible traveling-wave solution to~\eqref{eq-u}. Since the non-degeneracy conditions imply that, on bounded connected components of $\RR\setminus D$, $\uU$ cannot be constant, we only to prove that if near a discontinuity point $d_0\in D$, $\uU$ is constant neither on the left nor on the right, then the wave is unstable.

Let $d_0$ be such a point and $X_-$ and $X_+$ be the neighboring connected components of $\RR\setminus D$, respectively on the left and on the right of $d_0$. We first show that the absence of characteristic point on either $X_-$ or $X_+$ implies instability. Indeed it follows from Proposition~\ref{P.structure} that if on $X_{\pm}$ there is no characteristic point then $X_{\pm}$ is unbounded and the corresponding endstate $\uu_\infty$ satisfies $g'(\uu_\infty)>0$.

Now, let $x_-$ denote the largest characteristic point of $X_-$ and $x_+$ denote the smallest characteristic point of $X_+$. If $g'(x_+)\geq 0$ and $g'(x_-)\geq0$, it follows from Assumption~\ref{A.generic} and Proposition~\ref{P.structure} that $\uU(x_-)=\uU(x_+)\eqdef \uu_\star$ and $g'(\uu_\star)>0$. On $(x_-,d_0)$, since $g(\uU(\cdot))$ does not vanish, it has the sign of $g'(\uU(x_-))\,\uU'(x_-)=(g'(\uu_\star))^2/f''(\uu_\star)$, and likewise $\uU(\cdot)-\uu_\star$ has the sign of $\uU'(x_-)=g'(\uu_\star)/f''(\uu_\star)$. Similarly on $(d_0,x_+)$, $g(\uU(\cdot))$ has the sign of $-g'(\uU(x_+))\,\uU'(x_+)=-(g'(\uu_\star))^2/f''(\uu_\star)$ and $\uU(\cdot)-\uu_\star$ has the sign of $-\uU'(x_+)=-g'(\uu_\star)/f''(\uu_\star)$. As a consequence, $[g(\uU)]_{d_0}$ has the sign of $-(g'(\uu_\star))^2/f''(\uu_\star)$ and $[\uU]_{d_0}=[\uU-\uu_\star]_{d_0}$ has the sign of $-g'(\uu_\star)/f''(\uu_\star)$ so that $[g(\uU)]_{d_0}/[\uU]_{d_0}$ has the sign of $g'(\uu_\star)$, thus is positive.

In any case $\uU$ is spectrally (resp. nonlinearly) unstable either by Proposition~\ref{P.spectral-instability-infinity} (resp.~\ref{P.nl-instability-infinity}), Proposition~\ref{P.spectral-instability-characteristic} (resp.~\ref{P.nl-instability-characteristic}) or Proposition~\ref{P.spectral-instability-shock} (resp.~\ref{P.nl-instability-shock}).
\end{proof}

\begin{Corollary}
Assuming that $f$ is either strictly convex or strictly concave, for any $(\uU,\sigma,D)$ defining a non-degenerate piecewise regular entropy-admissible traveling-wave solution to~\eqref{eq-u}, spectral and nonlinear stability coincide and they are equivalent to $\uU$ taking one of the following forms
\begin{itemize}
\item $\uU$ is constant with value $\uu$ such that $g'(\uu)<0$;
\item $\uU$ is a Riemann shock with values $\uu_\pm$ satisfying $g'(\uu_\pm)<0$;
\item $\uU$ is a continuous front --- with endstates $\uu_{\pm\infty}$ satisfying $g'(\uu_{\pm\infty})<0$.
\end{itemize}
\end{Corollary}

\begin{proof}
First we point out that Assumption~\ref{A.generic} is automatically satisfied when $f$ is either strictly convex or strictly concave. Thus it only remains to prove that under the latter assumption the last and former-to-last cases of Theorem~\ref{T.classification} cannot happen.

Let $(\uU,\sigma,D)$ define a non-degenerate wave satisfying all the conditions of one of the two categories to be discarded except for the jump conditions. Pick $d\in D$. On one hand, we observe that since $[f'(\uU)]_d<0$, $[\uU]_d$ has the sign of $-f''$. On the other hand, a computation similar to the one at the end of the proof of Theorem~\ref{T.classification} yields that $[g(\uU)]_d$ has the sign of $-f''$. Thus $[g(\uU)]_d/[\uU]_d>0$. Hence the result.
\end{proof}

\section{Stable continuous fronts}\label{S.continous_fronts}

We begin the part of the paper devoted to stability results by the investigation of the dynamics near stable continuous fronts. This is the key missing point to carry out the strategy introduced in~\cite{DR1} and complete the stability analysis of more complex patterns.

\subsection{Detailed structure}

Given the importance of the pattern under consideration, we find convenient to store the detailed conditions on $(f,g)$ from which arises the existence of a stable continuous front.

\begin{Assumption}\label{A.front}
Assume that $(\uu_-,\uu_\star,\uu_+)\in\RR^3$, $\uu_-<\uu_\star<\uu_+$, that on a neighborhood of $[\uu_-,\uu_+]$, $f$ is $\cC^3$ and $g$ is $\cC^2$, and that the following conditions hold
\begin{enumerate}
\item $ $\\\vspace{-3em}
\begin{align*}
g(\uu_-)&=0\,,
&g(\uu_\star)&=0\,,
&g(\uu_+)&=0\,,\\
g'(\uu_-)&<0\,,
&g'(\uu_\star)&>0\,,
&g'(\uu_+)&<0\,;
\end{align*}
\item for any $u \in (\uu_-,\uu_+)\setminus\{\uu_\star\}$, $g(u)\neq 0$;
\item $f''(\uu_\star)\neq0$ and, for any $u \in [\uu_-,\uu_+]\setminus\{\uu_\star\}$, $f'(u)\neq f'(\uu_\star)$.
\end{enumerate}
Then we set $\sigma\eqdef f'(\uu_\star)$, label $\{\uu_{-\infty},\uu_{+\infty}\}=\{\uu_-,\uu_+\}$ according to
\begin{align*}
f'(\uu_{-\infty})&<\sigma\,,
&f'(\uu_{+\infty})&>\sigma\,,
\end{align*}
and define $\uU$ as the solution to $\uU(0)=\uu_\star$ and, for any $x\in\RR$, 
\[
\uU'(x)=F_\sigma (\uU(x))\,,
\]
where $F_\sigma\,:\ [\uu_-,\uu_+]\to\RR$ is as in~\eqref{def-F}.
\end{Assumption}

Note that, without loss of generality by the translation invariance of the equation, we have enforced that the single characteristic point of $\uU$ is $0$.

Notice also that the fact that $\uU$ is indeed globally defined stems from a scalar phase-portrait argument. Likewise, the following lemma follows from standard ordinary differential arguments.

\begin{Lemma}\label{L.front}
Let $(\uU,\sigma)$ be as in Assumption~\ref{A.front}. Then $(\uU,\sigma,\emptyset)$ defines a non-degenerate piecewise regular entropy-admissible traveling-wave solution to~\eqref{eq-u}, $\uU\in \cC^2(\RR)$,
\[
\dsp\lim_{x\to\pm\infty}\frac{\uU''(x)}{\uU'(x)}=\frac{g'(\uu_{\pm\infty})}{f'(\uu_{\pm\infty})-\sigma}\,,
\]
and there exists $C>0$ such that for any $x\in \RR$,
\begin{align*}
\abs{\uU(x)-\uu_{\pm\infty} }&\leq C\, \exp\left(\frac{g'(\uu_{\pm\infty})}{f'(\uu_{\pm \infty})-\sigma}x\right)\,,
\end{align*}
Moreover, for any $k\in\NN$ such that $f$ is $\cC^{k+1}$ and $g$ is $\cC^k$, there exists $C'>0$ such that for any $x\in \RR$, and any $\ell\in\NN$, $1\leq \ell\leq k$,
\begin{align*}
\abs{\uU^{(\ell)}(x)}&\leq C'\, \min\left(\left\{
\exp\left(\frac{g'(\uu_{+\infty})}{f'(\uu_{+ \infty})-\sigma}x\right),\exp\left(\frac{g'(\uu_{-\infty})}{f'(\uu_{- \infty})-\sigma}x\right)\right\}\right)\,.
\end{align*}
\end{Lemma}

\begin{Remark}[Instability]
It follows from Proposition~\ref{P.structure} that non-degenerate piecewise regular entropy-admissible traveling-wave solutions to~\eqref{eq-u} that are continuous are either constants or continuous strictly monotonic fronts. Combining Section~\ref{S.instabilities} with Proposition~\ref{P.structure}, one derives that such continuous strictly monotonic fronts are unstable unless they are generated as in Assumption~\ref{A.front}. Indeed, Propositions~\ref{P.spectral-instability-infinity} and~\ref{P.nl-instability-infinity} yield that stables ones must have limits $\uu_{\pm\infty}$ at $\pm\infty$ satisfying $g'(\uu_{-\infty})<0$ and $g'(\uu_{+\infty})<0$, and thus by Proposition~\ref{P.structure} they must also have an odd number of characteristic points with $g'$ alternating sign on those characteristic values. Hence from Propositions~\ref{P.spectral-instability-characteristic} and~\ref{P.nl-instability-characteristic} stems that such stable fronts possess exactly one characteristic point. 
\end{Remark}

\subsection{Spectral stability}

Let $(\uU,\sigma)$ be given by Assumption~\ref{A.front}. Specializing the discussion of Section~\ref{S.spectral-def} to the simpler case where the discontinuity set is empty, we consider the operator
\[
\cL\eqdef -(f'(\uU)-\sigma)\cD, \quad \cD\eqdef \left(\d_x -\frac{\uU''}{\uU'}\right)
=\uU'\,\d_x\left(\frac{1}{\uU'}\,\cdot\,\right)
\]
with maximal domain for various choices of functional space $\cX$.

To serve as such functional spaces, we introduce for $n\in\NN$
\[ X^n_\star(\RR) \eqdef \{ a\in BUC^n(\RR)\ ; \ a(0)=0\} .\]
We shall mostly consider $\cL$ on $X^1_\star(\RR)$. As a preliminary we first elucidate the interplay between this constrained spectrum and the original unconstrained spectrum.

\begin{Lemma}\label{L.lin-0}
Let $(\uU,\sigma)$ be given by Assumption~\ref{A.front} and assume that $k\in\NN$ is such that $f$ is $\cC^{k+2}$ and $g$ is $\cC^{k+1}$. Then the spectrum of $\cL$ on $BUC^k(\RR)$ is the union of $\{0\}$ and the spectrum of $\cL$ on $X^k_\star(\RR)$.
\end{Lemma}

\begin{proof}
For the sake of clarity, in the present proof we denote $\cL$ the operator on $BUC^k(\RR)$ and $\cL_\star$ the operator on $X^k_\star(\RR)$. To begin with, note that $0$ does belong to the spectrum of $\cL$ since $\uU'$ belongs to its kernel. 

Now we fix $\lambda\in\CC^\star$. First consider the spectral problem $(\lambda-\cL)v=A$. Evaluating at $x=0$ shows that it contains the constraint $v(0)=A(0)/\lambda$ so that if $\lambda$ is in the resolvent set of $\cL$, $(\lambda-\cL)^{-1}$ leaves $X^k_\star(\RR)$ invariant thus its restriction provides a resolvent for $\cL_\star$. Reciprocally, note that if $\lambda$ is in the resolvent set of $\cL_\star$ then a resolvent for $\cL$ is obtained through
\begin{equation}\label{eq:unstar}
(\lambda-\cL)^{-1}A\,=\,\frac{1}{\lambda}\,A(0)\,\frac{\uU'}{\uU'(0)}
+(\lambda-\cL_\star)^{-1}\left(A-\,A(0)\,\frac{\uU'}{\uU'(0)}\right)\,.
\end{equation}
\end{proof}

\begin{Remark}[Spectral projector]
Our subsequent analysis contains that indeed $0$ does not belong to the spectrum of $\cL$ on $X^k_\star(\RR)$ when $k\geq1$, so that considering the limit $\lambda\to0$ in~\eqref{eq:unstar} gives that $0$ is a simple eigenvalue of $\cL$ on $BUC^k(\RR)$ when $k\geq1$, with associated spectral projector $\Pi$ given by 
\[
\Pi A\,=\,A(0)\,\frac{\uU'}{\uU'(0)}\,.
\]
In particular it turns out that $X^k_\star(\RR)$ is simply $(\Id-\Pi)BUC^k(\RR)$. Obviously these observations are consistent with the fact that $\uU'$ lies in the kernel of $\cL$ (acting on any $BUC^k(\RR)$, $k\in\NN$) and that $\delta_0$ lies in the kernel of its adjoint.
\end{Remark}

To carry out our nonlinear analysis along the strategy in~\cite{DR1}, we need to analyze a class of operators close to $\cL$, sufficiently large to induce no regularity loss in Duhamel formulations but sufficiently small to retain the main features of $\cL$ expounded above. Our choice is to analyze operators $\cL_a$ defined as $\cL$ but with formula 
\[
\cL_a\eqdef -a\,\left(\d_x -\frac{\uU''}{\uU'}\right)
\,=\,-a\,\uU'\,\d_x\left(\frac{1}{\uU'}\,\cdot\,\right)
\]
when $a$ is sufficiently close to $(f'(\uU)-\sigma)$ in $X^1_\star(\RR)$.

Unlike the analysis in~\cite{DR1}, here the derivation of higher-order regularity estimates requires a specific analysis. To perform it we introduce, for $k\in\NN$, $L_{a,k}$ operating on $BUC^0(\RR)$ with maximal domain, defined by
\[
L_{a,k}\eqdef -a\,\left(\d_x -\frac{\uU''}{\uU'}\right)-k\,a'\,,
\]
and note that when $k\in\NN$ is such that $a\in BUC^k(\RR)$, $f\in\cC^{k+3}(\RR)$ and $g\in\cC^{k+2}(\RR)$, 
\begin{equation}\label{eq:augmented-derivative}
\begin{pmatrix}
\cL_a\,v\\\d_x(\cL_a\,v)\\\vdots\\\d_x^k(\cL_a\,v)
\end{pmatrix}
\,=\,
\begin{pmatrix}
\cL_a\,v\\L_{a,1}\,\d_xv\\\vdots\\L_{a,k}\,\d_x^kv
\end{pmatrix}
+*
\begin{pmatrix}
v\\\d_xv\\\vdots\\\d_x^kv
\end{pmatrix}
\end{equation}
where $*$ denotes an operator strictly lower triangular and bounded on $X^1_\star(\RR)\times (BUC^0)^{k-1}(\RR)$.

Since it is simpler, we begin by considering spectrum problems for $L_{a,k}$, $k\in\NN$. Elementary considerations --- for instance based on a direct comparison with $-a'(0)\,x\,\d_x-k\,a'(0)$ as in~\cite[Appendix~A]{JNRYZ} --- show that if $a\in X^1_\star$, $a'(0)>0$ and $\Re(\lambda)>-k\,a'(0)$, there exists an open interval $I_0$ containing $0$ such that for any $A\in BUC^0(I_0)$ there exists a unique $v\in BUC^0(I_0)$ such that $(\lambda-L_{a,k})v=A$, and $v$ is given by
\begin{equation}\label{eq:local-spectral}
v(x)=\begin{cases}\displaystyle
\qquad\frac{A(0)}{\lambda+k\,a'(0)}\,,&\qquad x=0\,,\\
\displaystyle
\int_{0}^x\frac{e^{-\int_y^x\frac{1}{a}(\lambda+k a'-a\frac{\uU''}{\uU'})}}{a(y)}
A(y)\dd y\,,&\qquad x\neq0\,.
\end{cases}
\end{equation}
Note that if moreover $a$ does not vanish on $\RR^\star$ then the uniqueness part (with associated formula~\eqref{eq:local-spectral}) of the latter discussion may be extended from $I_0$ to $\RR$ by standard ordinary differential equations arguments, though the existence part may fail.

The non-vanishing of $a$ when $a$ is sufficiently close to $(f'(\uU)-\sigma)$ in $X^1_\star(\RR)$ follows from the following lemma.

\begin{Lemma}\label{L.non-vanishing}
Let $(\uU,\sigma)$ be given by Assumption~\ref{A.front}. There exist $c>0$ and $C>0$ such that for any $a\in X^1_\star(\RR)$
\begin{align*}
a(x)&\,\geq (c-C\Norm{a-(f'(\uU)-\sigma)}_{W^{1,\infty}})\,\min(\{1,|x|\})
&\qquad \textrm{when }x>0\,,\\
a(x)&\,\leq -(c-C\Norm{a-(f'(\uU)-\sigma)}_{W^{1,\infty}})\,\min(\{1,|x|\})
&\qquad \textrm{when }x<0\,.
\end{align*}
\end{Lemma}
\begin{proof}
Immediate from
\[
\abs{a(x)-(f'(\uU(x))-\sigma)}
\,\leq\,
\min\left(\left\{
\Norm{a-(f'(\uU)-\sigma)}_{L^{\infty}},
|x|\,\Norm{a'-(f'(\uU))'}_{L^{\infty}}
\right\}\right)\,.
\]
\end{proof}
From Lemma~\ref{L.non-vanishing}, we derive that when $a$ is sufficiently close to $(f'(\uU)-\sigma)$ in $X^1_\star(\RR)$,
\begin{equation}\label{eq:non-vanishing}
\exists c_0>0\,,\qquad
\forall x\in\RR\,,\qquad
\begin{cases}
a(x)\,\geq c_0\,\min(\{1,|x|\})
&\qquad \textrm{when }x>0\,,\\
a(x)\,\leq -c_0\,\min(\{1,|x|\})
&\qquad \textrm{when }x<0\,.
\end{cases}
\end{equation}

To complete the argument and provide resolvent bounds of contraction type, we introduce weights independent of $a$ according to the following lemma.

\begin{Lemma}\label{L.chi}
Let $(\uU,\sigma)$ be given by Assumption~\ref{A.front} and denote for any $k\in\NN$
\begin{equation}\label{eq:theta_k}
\theta_k\eqdef\min(\{k\,g'(\uu_\star),-g'(\uu_+),-g'(\uu_-)\})\,.
\end{equation}
There exists $C_0>0$ such that for any  $k\in\NN$, there exists $\chi_k\in L^1(\RR)\cap BUC^0(\RR)$ such that for any $a\in \cC^1(\RR)\cap W^{1,\infty}(\RR)$,
\begin{equation}\label{eq:theta_ak}
\theta_{a,k}\eqdef\inf_{\RR} \left(k\,a'-a\frac{\uU''}{\uU'}+a\chi_k\right)\geq \theta_k - C_0(1+k)\Norm{a-(f'(\uU)-\sigma)}_{W^{1,\infty}(\RR)}\,.
\end{equation}
\end{Lemma} 
\begin{proof}
We define
\begin{equation}\label{def-chi}
\chi_k(x)=\begin{cases}
0&\textrm{when }x=0\,,\\
\max\left(\left\{\frac{\theta_k-k\,(f'(\uU))'(x)+(f'(\uU(x))-\sigma)\frac{\uU''(x)}{\uU'(x)}}{(f'(\uU(x))-\sigma)},0\right\}\right)
&\textrm{when }x>0\,,\\
\min\left(\left\{\frac{\theta_k-k\,(f'(\uU))'(x)+(f'(\uU(x))-\sigma)\frac{\uU''(x)}{\uU'(x)}}{(f'(\uU(x))-\sigma)},0\right\}\right)
&\textrm{when }x<0\,.
\end{cases}
\end{equation}
Now we observe that
\[ \lim_{x\to 0 } k\,(f'(\uU))'(x)=k\,g'(\uu_\star)\,,\]
and that the convergences
\[k\,(f'(\uU))'(x)-(f'(\uU(x))-\sigma)\frac{\uU''(x)}{\uU'(x)}
\stackrel{x\to\pm\infty}{\longrightarrow} -g'(\uu_{\pm\infty})\]
are exponentially fast. Thus $\chi_k\in L^1(\RR)\cap BUC^0(\RR)$. Since the function $\chi_k$ is designed to ensure
\[ k(f'(\uU))'-(f'(\uU)-\sigma)\frac{\uU''}{\uU'}+(f'(\uU)-\sigma)\chi_k\geq \theta_k\,,\]
the result follows. 
\end{proof}

\begin{Lemma}\label{L.lin-1}
Let $(\uU,\sigma)$ be defined by Assumption~\ref{A.front}, $k\in \NN$ and a corresponding $\chi_k$ is given by Lemma~\ref{L.chi}. If $a\in X^1_\star(\RR)$ satisfies~\eqref{eq:non-vanishing} then for any $\lambda\in\CC$ such that 
\[
\Re(\lambda)>-\theta_{a,k}\,,\] 
with $\theta_{a,k}$ as in~\eqref{eq:theta_ak} and any $A\in BUC^0(\RR)$, there exists a unique $v\in BUC^0(\RR)$ such that 
\[
(\lambda-L_{a,k})\,v\,=\,A
\]
and, moreover,
\[
\Norm{e^{-\int_{0}^{\cdot}\chi_k}v}_{L^{\infty}(\RR)}
\,\leq\,\frac{1}{\Re(\lambda)+\theta_{a,k}}\,\Norm{e^{-\int_{0}^{\cdot}\chi_k}A}_{L^{\infty}(\RR)}\,.
\]
\end{Lemma}
\begin{proof}
Thanks to~\eqref{eq:non-vanishing} and $\Re(\lambda)>-\theta_{a,k}\geq -k\,a'(0)$, one may rely directly on the arguments expounded above and focus on~\eqref{eq:local-spectral} that yields for $x\neq0$
\[
e^{-\int_{0}^{x}\chi_k}v(x)
\,=\,
\int_{0}^x\frac{e^{-\int_y^x\frac{1}{a}(\lambda+k a'-a\frac{\uU''}{\uU'}+a\,\chi_k)}}{a(y)} e^{-\int_{0}^{y}\chi_k}\,A(y)\,\dd y
\]
thus
\begin{align*}
&\abs{e^{-\int_{0}^{x}\chi_k}v(x)}\\
&\,\leq\,
\frac{\Norm{e^{-\int_{0}^{\cdot}\chi_k}A}_{L^{\infty}(\RR)}}{\Re(\lambda)+\theta_{a,k}}
\,\times\begin{cases}\displaystyle
\int_{0}^x\frac{e^{-\int_y^x\frac{1}{a}(\Re(\lambda)+k a'-a\frac{\uU''}{\uU'}+a\,\chi_k)}}{a(y)}\,\left(\Re(\lambda)+k a'-a\frac{\uU''}{\uU'}+a\,\chi_k\right)(y)\,\dd y&
\qquad x>0\\\displaystyle
-\int_{x}^0\frac{e^{-\int_y^x\frac{1}{a}(\Re(\lambda)+k a'-a\frac{\uU''}{\uU'}+a\,\chi_k)}}{a(y)}\,\left(\Re(\lambda)+k a'-a\frac{\uU''}{\uU'}+a\,\chi_k\right)(y)\,\dd y&
\qquad x<0
\end{cases}\\
&\,\leq\,\frac{\Norm{e^{-\int_{0}^{\cdot}\chi_k}A}_{L^{\infty}(\RR)}}{\Re(\lambda)+\theta_{a,k}}\,.
\end{align*}
Hence the result.
\end{proof}

Now we turn to the consideration of $\cL_a$ on $X^1_\star$. The key observation is that
\begin{equation}\label{eq:key-derivative}
\cD\,\cL_a\,v\,\eqdef\,\uU'\d_x\left(\frac{1}{\uU'}\cL_a\,v\right)
\,=\,L_{a,1}\, \cD\, v
\,=\,L_{a,1}\left(\d_xv-\frac{\uU''}{\uU'}v\right)\,.
\end{equation}
To go further, we show that, when $v\in X^1_\star$, $\Norm{v}_{W^{1,\infty}}$ is controlled by a multiple of $\Norm{\d_xv-\tfrac{\uU''}{\uU'}v}_{L^{\infty}}$.

\begin{Lemma}\label{L.norm-equivalence}
Let $(\uU,\sigma)$ be defined by Assumption~\ref{A.front} and a corresponding $\chi\eqdef\chi_1$ be obtained from Lemma~\ref{L.chi} (with $k=1$). Then $\Norm{\,\cdot\,}_{\star}$, given by
\begin{equation}\label{def-norm-star}
\Norm{v}_{\star}\,\eqdef\,
\Norm{e^{-\int_{0}^{\cdot}\chi}\left(\d_xv-\tfrac{\uU''}{\uU'}v\right)}_{L^{\infty}(\RR)}\,,\qquad
v\in X^1_\star(\RR)\,,
\end{equation}
defines a norm on $X^1_\star(\RR)$ equivalent to $\Norm{\,\cdot\,}_{W^{1,\infty}(\RR)}$.
\end{Lemma}
\begin{proof}
That $\Norm{\,\cdot\,}_{\star}$ defines a seminorm is obvious. Since $\chi\in L^1(\RR)$ and $\uU''/\uU'\in L^\infty(\RR)$ by  Lemma~\ref{L.front}, it is sufficient to prove that when $v\in X^1_\star(\RR)$, $\Norm{v}_{L^{\infty}(\RR)}$ is controlled by a multiple of $\Norm{\d_xv-\tfrac{\uU''}{\uU'}v}_{L^{\infty}(\RR)}$. Yet, when $v\in X^1_\star(\RR)$, with $w\eqdef\d_xv-\tfrac{\uU''}{\uU'}v=\uU'\,\d_x(v/\uU')$, we have for any $x\in\RR$,
\begin{equation}\label{eq:v-w}
v(x)\,=\,\int_0^x\,\frac{\uU'(x)}{\uU'(y)}\,w(y)\,\dd y
\,.
\end{equation}
Thus, since $\uU'$ does not vanish, it is sufficient to prove that 
\[
\sup_{x\in\RR}\quad\left|\int_0^x\,\frac{\uU'(x)}{\uU'(y)}\,\dd y\right| <+\infty\,.
\]
To prove the latter we observe that from Lemma~\ref{L.front} stems that for $R_0$ sufficient large there exists $c_0>0$ such that 
\[
-c_0\,\frac{\uU''(x)}{\uU'(x)}\,\geq\,1\qquad\textrm{when }x\geq R_0
\qquad\textrm{and}\qquad
c_0\,\frac{\uU''(x)}{\uU'(x)}\,\geq\,1\qquad \textrm{when }x\leq -R_0\,.
\]
From this follows for $x\geq R_0$
\[
\int_{R_0}^x\,\frac{\uU'(x)}{\uU'(y)}\,\dd y
\,\leq\,c_0\uU'(x)\int_{R_0}^x\,\left(-\frac{\uU''(y)}{(\uU'(y))^2}\right)\,\dd y
\,\leq\,c_0
\]
and for $x\leq -R_0$
\[
\int_{x}^{-R_0}\,\frac{\uU'(x)}{\uU'(y)}\,\dd y
\,\leq\,-c_0\uU'(x)\int_{x}^{-R_0}\,\left(-\frac{\uU''(y)}{(\uU'(y))^2}\right)\,\dd y
\,\leq\,c_0\,.
\]
Hence the claim.
\end{proof}

Note that in particular the foregoing lemma contains that for any $(A,B)\in (X^1_\star(\RR))^2$, $A=B$ if and only if $\d_xA-\tfrac{\uU''}{\uU'}A=\d_xB-\tfrac{\uU''}{\uU'}B$. Reciprocally it follows readily from~\eqref{eq:v-w} that for any $w\in BUC^0$ there exists a $v\in X^1_\star$ such that $\d_xv-\tfrac{\uU''}{\uU'}v=w$.

The derivation of the following lemma from~\eqref{eq:key-derivative} and Lemmas~\ref{L.lin-1} and~\ref{L.norm-equivalence} is now immediate.
\begin{Lemma}\label{L.lin-2}
Let $(\uU,\sigma)$ be given by Assumption~\ref{A.front} and $a\in X^1_\star(\RR)$ satisfying~\eqref{eq:non-vanishing}. Let $\chi\eqdef\chi_1$ be obtained from Lemma~\ref{L.chi} (with $k=1$). For any $\lambda\in\CC$ such that 
\[
\Re(\lambda)>-\theta_{a,1}\,,\] 
with $\theta_{a,1}$ as in~\eqref{eq:theta_ak}, for any $A\in X^1_\star(\RR)$, there exists a unique $v\in X^2_\star(\RR)$ such that 
\[
(\lambda-\cL_{a})\,v\,=\,A
\]
and, moreover,
\[
\Norm{v}_{\star}
\,\leq\,\frac{1}{\Re(\lambda)+\theta_{a,1}}\,\Norm{A}_{\star}
\]
where $\Norm{\,\cdot\,}_{\star}$ is defined by~\eqref{def-norm-star}. 
\end{Lemma}

From the strictly lower triangular structure in~\eqref{eq:augmented-derivative} and Lemmas~\ref{L.lin-0},~\ref{L.chi},~\ref{L.lin-1} and~\ref{L.lin-2}, stems spectral stability in $BUC^k(\RR)$, $k\in\NN$.

\begin{Proposition}\label{P.spectral-stability-front}
Let $(\uU,\sigma)$ be defined by Assumption~\ref{A.front} and $k\in\NN$ be such that $f\in\cC^{k+3}(\RR)$ and $g\in\cC^{k+2}(\RR)$. Then the spectrum of $\cL$ on $BUC^k(\RR)$ is included in
\[ 
\{\,\lambda\,;\,\Re(\lambda)\leq -\theta\,\}\cup\{0\}
\]
with
\begin{equation}\label{eq:theta}
\theta\eqdef\min(\{\,g'(\uu_\star),-g'(\uu_+),-g'(\uu_-)\})>0\,,
\end{equation}
and $0$ is a simple eigenvalue with eigenvector $\uU'$.
\end{Proposition}

\subsection{Linear estimates}

With the resolvent estimates of the time-independent operators obtained in Lemmas~\ref{L.lin-1} (with $k=0$) and~\ref{L.lin-2}, we may apply general theorems on evolution systems. See for instance~\cite[Chapter~5, Theorem~3.1]{Pazy} with $X=X^0_\star(\RR)$ and $Y=X^1_\star(\RR)$.

\begin{Proposition}\label{P.linear}
Let $(\uU,\sigma)$ be given by Assumption~\ref{A.front}. There exist $\eps_0>0$ and $C_0$ such that if $T>0$ and $a\in \cC^1([0,T];X^1_\star(\RR))$ satisfies for any $t\in[0,T]$ 
\[
\Norm{a(t,\cdot)-(f'(\uU)-\sigma)}_{W^{1,\infty}(\RR)}\leq \eps_0
\]
then the family of operators $\cL_{a(t,\cdot)}$ generates an evolution system $\cS_a$ on $X^0_\star(\RR)$ and for any\\
$0\leq s\leq t<T$, $\cS_a(s,t)(X^1_\star(\RR))\subset X^1_\star(\RR)$ and for any $v_0\in X^1_\star(\RR)$
\[
\Norm{\cS_a(s,t)\,v_0}_{W^{1,\infty}(\RR)}\leq 
C_0\,e^{-\theta\,(t-s)}\,e^{C_0\,\int_s^t\Norm{a(\tau,\cdot)-(f'(\uU)-\sigma)}_{W^{1,\infty}(\RR)}\,\dd\tau}\,\Norm{v_0}_{W^{1,\infty}(\RR)}\,,
\]
with $\theta$ as in~\eqref{eq:theta}.
\end{Proposition}

Likewise our study of the decay of higher-order derivatives will be derived from the following lemma. Here to use~\cite[Chapter~5, Theorem~3.1]{Pazy} with $X=BUC^0(\RR)$ and $Y=BUC^1(\RR)$, we rely on~\cite[Chapter~5, Theorem~2.3]{Pazy} to reduce the verification of assumption~$(H_2)$ there to another application of Lemma~\ref{L.lin-1}, with index $k+1$.

\begin{Lemma}\label{L.linear}
Let $(\uU,\sigma)$ be given by Assumption~\ref{A.front}. For any $k\in\NN$, there exists $C_k>0$ such that if $T>0$ and $a\in \cC^1([0,T];X^2_\star(\RR))$ are such that for any $t\in[0,T]$, $a(t,\cdot)$ satisfies~\eqref{eq:non-vanishing}, then the family of operators $L_{a(t,\cdot),k}$ generates an evolution system $S_{a,k}$ on $BUC^0(\RR)$ and for any $0\leq s\leq t<T$ and any $v_0\in BUC^0(\RR)$
\[
\Norm{S_{a,k}(s,t)\,v_0}_{L^\infty(\RR)}\leq 
C_k\,e^{-\theta_k\,(t-s)}\,e^{C_k\,\int_s^t\Norm{a(\tau,\cdot)-(f'(\uU)-\sigma)}_{W^{1,\infty}(\RR)}\,\dd\tau}\,\Norm{v_0}_{L^\infty(\RR)}\,,
\]
with $\theta_k$ as in~\eqref{eq:theta_k}.
\end{Lemma} 

\subsection{Nonlinear stability under shockless perturbations}

To deduce nonlinear stability from Proposition~\ref{P.linear} we still need a lemma ensuring that, thanks to invariance by spatial translation of~\eqref{eq-u}, at the nonlinear level one may also restrict to perturbations in $X^1_\star(\RR)$.
\begin{Lemma}\label{L.shift}
Let $(\uU,\sigma)$ be given by Assumption~\ref{A.front}. 
For any $C_0>1$, there exists $\eps_0>0$ such that for any $v_0\in W^{1,\infty}(\RR)$ satisfying
\[
\Norm{v_0}_{W^{1,\infty}(\RR)}\leq \eps_0,
\]
there exists a unique $x_{\star,v_0}\in \RR$ such that 
\[
\uU(x_{\star,v_0})+v_0(x_{\star,v_0})=\uu_\star\,,
\]
and it satisfies
\[
\abs{x_{\star,v_0}}\leq  \frac{C_0}{\abs{\uU'(0)}}\Norm{v_0}_{L^{\infty}(\RR)}\,.
\]
One has $\uU+v_0=(\uU+\widetilde v_0)(\cdot-x_{\star,v_0})$ with $\widetilde v_0\in X^1_\star(\RR)$ and
\[\Norm{\widetilde v_0}_{W^{1,\infty}(\RR)} \leq \Norm{v_0}_{W^{1,\infty}(\RR)}\ + \ \Norm{\uU'}_{W^{1,\infty}(\RR)} |x_{\star,v_0}|\,.
\]
\end{Lemma}
\begin{proof}
Since  $\uU$ is strictly monotonic and $\uU'(0)\neq0$, we may fix $r>0$ and $\delta>0$ such that
\[
|\uU'(x)|\geq C_0^{-1}|\uU'(0)|\text{ when }|x|\leq r
\qquad\textrm{and}\qquad
|\uU(x)-\uu_\star|\geq \delta\text{ when }|x|\geq r\,.
\]
Hence by choosing $\eps_0$ to enforce
\[
\Norm{\d_xv_0}_{L^{\infty}(\RR)}\leq \frac12 C_0^{-1}|\uU'(0)|
\qquad\textrm{and}\qquad
\Norm{v_0}_{L^{\infty}(\RR)}\leq \frac12 \delta\,,
\]
we have that the function $\uU+v_0-\uu_\star$ vanishes exactly once. Indeed the latter function has the sign of $\uU-\uu_\star$ on $(-\infty,-r]$ and on $[r,\infty)$ and is strictly monotonic and changing sign on $[-r,r]$. The estimate on $x_{\star,v_0}$ stems from $x_{\star,v_0}\in [-r,r]$ and $\uU(x_{\star,v_0})-\uU(0)=v_0(x_{\star,v_0})$. The last estimate on $\widetilde v_0\eqdef\uU(\cdot+x_{\star,v_0})-\uU+v_0(\cdot+x_{\star,v_0})$ is immediate.
\end{proof}

\begin{Theorem}\label{T.stable-front}
Let $(\uU,\sigma)$ be given by Assumption~\ref{A.front}, assume that on a neighborhood of $[\uu_-,\uu_+]$, $f$ is $\cC^4$ and $g$ is $\cC^3$, and as in~\eqref{eq:theta} set
\[\theta\eqdef \min(\{g'(\uu_\star),-g'(\uu_{+\infty}),-g'(\uu_{-\infty})\})>0.\]
There exist $\eps>0$ and $C>0$ such that for any $\psi_0\in\RR$ and any $v_0\in BUC^1(\RR)$ satisfying
\begin{equation}\label{perturbation0}
\Norm{v_0}_{W^{1,\infty}(\RR)}\leq \eps,
\end{equation}
the initial datum $u\id{t=0}=\uU(\cdot-\psi_0)+v_0$ generates a unique global classical solution to~\eqref{eq-u}, $u\in BUC^1(\RR^+\times\RR)$, and there exists $\psi_\infty\in\RR$ such that 
\begin{equation}\label{phase-perturbation}
|\psi_\infty-\psi_0|\leq C\,\Norm{v_0}_{L^{\infty}(\RR)}
\end{equation}
and for any $t\geq 0$
\begin{equation}\label{sonic-point}
u(t,\sigma t+\psi_\infty)=\uu_\star
\end{equation}
and
\begin{equation}\label{asymptotic-stability}
\Norm{u(t,\cdot+(\sigma t+\psi_\infty))-\uU}_{W^{1,\infty}(\RR)}\leq 
C\,e^{-\theta t}\,\Norm{v_0}_{W^{1,\infty}(\RR)}\,.
\end{equation}
\end{Theorem}
\begin{proof}
Translating spatially by $\psi_0$ and replacing $v_0$ with $v_0(\cdot+\psi_0)$ reduces the theorem to the case $\psi_0=0$ . Then, assuming $\psi_0=0$ and applying lemma~\ref{L.chi} show that the general case with $\psi_\infty=x_{\star,v_0}$ may be deduced from the theorem restricted to perturbations $v_0\in X^1_\star(\RR)$ yielding asymptotic phase shifts $\psi_\infty=0$. From now on we restrict to $\psi_0=0$ and $v_0\in X^1_\star(\RR)$.

The local well-posedness at the level of classical solutions is well-known so that we may focus on proving global bounds, which also imply global existence. To begin with we point out that considering, as in the proof of Proposition~\ref{P.nl-instability-characteristic}, the characteristic curve arising from the characteristic point $0$ shows that a classical solution satisfying~\eqref{sonic-point} initially satisfies it for any later time. Now, to study a solution $u$ on an interval $I$, we put it under the form 
\[u(t,x)=\uU(x-\sigma t)+v(t,x-\sigma t)\]
and observe that $v$ satisfies for $t\in I$, $v(t,\cdot)\in X^1_\star(\RR)$ and
\begin{equation}\label{eq-tu}
\d_tv(t,\cdot)-\cL_{f'(\uU+v(t,\cdot))-\sigma}\,v(t,\cdot)
\,=\,\cN(v(t,\cdot)),
\end{equation}
where for any function $w$
\begin{align*}
\cN(w)&\eqdef\left(f'(\uU)-f'(\uU+w)\right)\frac{\uU''}{\uU'}\,w+g(\uU+w)-g(\uU)-g'(\uU)w\\
&\qquad-\left(f'(\uU+w)-f'(\uU)-f''(\uU)w\right)\uU'\,.
\end{align*}
Choose $\delta>0$ such that $f\in\cC^4([\uu_--\delta,\uu_++\delta])$ and $g\in\cC^3([\uu_--\delta,\uu_++\delta])$. Then there exists $C_1$ such that for any $w\in X^1_\star(\RR)$ satisfying $\Norm{w}_{L^{\infty}(\RR)}\leq \delta$, we have $f'(\uU+w)-\sigma\in X^1_\star(\RR)$, $\cN(w)\in X^1_\star(\RR)$ and
\begin{align*}
\Norm{(f'(\uU+w)-\sigma)-(f'(\uU)-\sigma)}_{W^{1,\infty}(\RR)}&\leq 
C_1\,\Norm{w}_{W^{1,\infty}(\RR)}\,,\\
\Norm{\cN(w)}_{W^{1,\infty}(\RR)}&\leq C_1\,\Norm{w}_{W^{1,\infty}(\RR)}^2\,.
\end{align*}
Moreover pick $C_0\geq 1$ and $\eps_0>0$ as in Proposition~\ref{P.linear}.

If $v_0\in X^1_\star(\RR)$ satisfies 
\[
\Norm{v_0}_{W^{1,\infty}(\RR)}\leq\frac{1}{2\,C_0}\min\left(\left\{\delta,\frac{\eps_0}{C_1}\right\}\right)
\]
then when $T\geq0$ is such that the corresponding solution is defined on $[0,T]$ and satisfies for any $t\in [0,T]$
\[
\Norm{v(t,\cdot)}_{W^{1,\infty}(\RR)}
\leq 2\,C_0\,e^{-\theta\,t}\,\Norm{v_0}_{W^{1,\infty}(\RR)}
\]
there holds that for any $t\in [0,T]$
\[
v(t,\cdot)=
\cS_{f'(\uU+v)-\sigma}(0,t)(v_0)
+\int_0^t\cS_{f'(\uU+v)-\sigma)}(s,t)\,\cN(v(s))\,\dd s\,,\]
thus, by the Gr\"onwall lemma, for any $t\in [0,T]$,
\[
\Norm{v(t,\cdot)}_{W^{1,\infty}(\RR)}
\leq C_0\,e^{-\theta\,t}\,\Norm{v_0}_{W^{1,\infty}(\RR)}\,\times 
\exp\left(\frac{2\,C_0\,C_1}{\theta}\,\Norm{v_0}_{W^{1,\infty}(\RR)}
\left(1+e^{\frac{2\,C_0\,C_1}{\theta}\,\Norm{v_0}_{W^{1,\infty}(\RR)}}\right)\right)\,.
\]

By choosing $C\eqdef 2\,C_0$ and taking $\eps>0$ sufficiently small to guarantee
\begin{align*}
\eps&\leq\frac{1}{2\,C_0}\min\left(\left\{\delta,\frac{\eps_0}{C_1}\right\}\right)&
\textrm{and}&&
\exp\left(\frac{2\,C_0\,C_1}{\theta}\,\eps
\left(1+e^{\frac{2\,C_0\,C_1}{\theta}\,\eps}\right)\right)
&<2
\end{align*}
one concludes the proof by a continuity argument.
\end{proof}

We also prove that the exponential decay in time holds for higher order derivatives without further restriction on sizes of perturbations.
\begin{Proposition}\label{P.high-order-front} Under the assumptions of Theorem~\ref{T.stable-front}, if one assumes additionally that $f\in\cC^{k+3}(\RR)$, $g\in \cC^{k+2}(\RR)$ with $k\in\NN$, $k\geq 2$ then there exists $C_k>0$, depending on $f$, $g$, $\uU$ and $k$ but not on the initial data $v_0$, such that if $v_0\in BUC^k(\RR)$ additionally to constraints in Theorem~\ref{T.stable-front}, then the global unique classical solution to~\eqref{eq-u} emerging from the initial data $\uU(\cdot-\psi_0)+v_0$ satisfies $u\in BUC^k(\RR^+\times\RR)$ and, with the same $\psi_\infty$ as in Theorem~\ref{T.stable-front}, for any $t\geq0$,
\begin{align*}
&\Norm{\d_x^k u(t,\cdot+(\sigma\,t+\psi_\infty))-\d_x^k\uU}_{L^{\infty}(\RR)}\\
&\ \leq 
C_k \left(e^{-\theta_1\,t}\,t^{(k-1)\delta_{\geq k_0}(1)}\,\Norm{v_0}_{W^{1,\infty}(\RR)}
+\sum_{\ell=2}^ke^{-\theta_\ell\,t}\,t^{(k-\ell)\delta_{\geq k_0}(\ell)}\,\Norm{\d_x^\ell v_0}_{L^{\infty}(\RR)}\right)\,,
\end{align*}
where $(\theta_\ell)_{\ell\in\NN}$ are as in~\eqref{eq:theta_k}, $k_0\in\NN^\star$ is the smallest index such that $\theta_{k_0}=\theta_{k_0+1}$ and $\delta_{\geq k_0}(\ell)=1$ if $\ell\geq k_0$ and $\delta_{\geq k_0}(\ell)=0$ if $\ell< k_0$. 
\end{Proposition}

\begin{proof}
Again propagation of regularity is classical so that we may focus on proving bounds. We use here notation and reductions from the proof of Theorem~\ref{T.stable-front}.

Under the above conditions, by relying on~\eqref{eq:augmented-derivative}, Lemma~\ref{L.linear} and composition properties of Sobolev norms, we obtain that there exist a constant $C_k'>0$, depending on $k$, $f$, $g$, $\uU$ and the choice of $\eps$ in Theorem~\ref{T.stable-front} but not on $v_0$, such that for any $t\geq0$
\begin{align*}
\Norm{\d_x^k v(t,\cdot)}_{L^{\infty}(\RR)}
&\leq C_k'\,e^{-\theta_k\,t}\,\Norm{\d_x^k v_0}_{L^{\infty}(\RR)}
+C_k'\int_0^t\,e^{-\theta_k\,(t-s)}\,\Norm{v(s,\cdot)}_{W^{k-1,\infty}(\RR)}\dd s\\
&+C_k'\int_0^t\,e^{-\theta_k\,(t-s)}\,\Norm{v(s,\cdot)}_{L^{\infty}(\RR)}\,\Norm{\d_x^kv(s,\cdot)}_{L^{\infty}(\RR)}\dd s\,.
\end{align*}
Applying the Gr\"onwall lemma and arguing recursively proves the result.
\end{proof}

\subsection{Nonlinear stability under perturbations with small shocks}

Following the strategy already used in~\cite{DR1}, we may extend Theorem~\ref{T.stable-front} to the case where the perturbation contains a finite number of well-separated strictly entropic discontinuities under strict convexity/concavity assumptions of the advective flux $f$. The basic tenet is that infinitesimally small discontinuities travel approximately along characteristic curves of the background wave. Thus the motions of these small shocks may be predicted accurately. In order to lighten the notational complexity, as in~\cite{DR1}, we limit our setting to the case of one discontinuity.

For stable continuous fronts, small shocks introduced in perturbed initial data persist forever but their position drift toward infinity and their amplitude converge exponentially fast to zero. We provide a description of the solution $u$ as regular on
\[
\Omega^\phi\eqdef\RR_+\times\RR\setminus\{\,(t,\phi(t))\,|\,t\geq0\,\}
\]
where $\phi$ follows the position of the shock. In order for $u$ to satisfy the equation in distributional sense we require $u$ to satisfy it in a classical sense on $\Omega^\phi$ and that it also satisfies the Rankine-Hugoniot condition, for any $t\geq0$
\begin{equation}\label{RH}
 f(u_r(t))-f(u_l(t))=\phi'(t)  (u_r(t)-u_l(t))
\end{equation}
where $u_l(t)=u(t,\phi(t)^-)\eqdef\lim_{\delta\searrow 0} u(t,\phi(t)-\delta)$ and $u_r(t)=u(t,\phi(t)^+)\eqdef\lim_{\delta\searrow 0} u(t,\phi(t)+\delta)$. We will also enforce its admissibility as an entropy solution through the (strict) Lax condition
\begin{equation}\label{Lax}
f'(u_l(t))\,>\,\phi'(t)\,>\,f'(u_r(t))
\end{equation}
that implies the (strict) Oleinik condition provided that $f''$ does not vanish on $[u_l(t),u_r(t)]$.

\begin{Proposition}\label{P.small-shock}
Let $(\uU,\sigma)$ be given by Assumption~\ref{A.front}, assume that on a neighborhood of $[\uu_-,\uu_+]$, $f$ is $\cC^4$ and $g$ is $\cC^3$, and that $\delta_0>0$ (resp. $\delta_0<0$) is such that $f''(\uU)$ does not vanish on $[\delta_0,+\infty)$ and $f''(\uu_{+\infty})\neq0$ (resp. on $(-\infty,\delta_0]$ and $f''(\uu_{-\infty})\neq0$) and as in~\eqref{eq:theta} set
\[\theta\eqdef \min(\{g'(\uu_\star),-g'(\uu_\infty),-g'(\uu_{-\infty})\})>0.\]
There exists $\eps>0$ and $C>0$ such that for any $\psi_0\in\RR$ and any $v_0\in BUC^1(\RR^\star)$ satisfying
\begin{equation}\label{perturbation0-shock}
\Norm{v_0}_{W^{1,\infty}(\RR^\star)}\leq \eps\,,\qquad\qquad
f'(\uU(\delta_0)+v_0(0^-))<f'(\uU(\delta_0)+v_0(0^+))\,,
\end{equation}
the initial datum $u\id{t=0}=\uU(\cdot-\psi_0)+v_0(\cdot-(\psi_0+\delta_0))$ generates a unique global entropy solution to~\eqref{eq-u}, and there exists $\phi\in\cC^2(\RR^+)$ with initial data $\phi(0)=\psi_0+\delta_0$ such that $u\in BUC^1(\Omega^\phi)$ and 
\begin{enumerate}
\item there exists $\psi_\infty\in\RR$ such that 
\[
|\psi_\infty-\psi_0|\leq C\,\Norm{v_0}_{L^{\infty}(\RR)}
\]
and for any $t\geq 0$
\[
u(t,\sigma t+\psi_\infty)=\uu_\star
\]
and
\[
\Norm{u(t,\cdot)-\uU(\cdot-(\sigma t+\psi_\infty))}_{W^{1,\infty}(\RR\setminus\{\phi(t)\})}
\leq C\,e^{-\theta t}\,\Norm{v_0}_{W^{1,\infty}(\RR^\star)}
\]
\item there exists $(\phi_\infty,\phi_{{\rm as},\infty})\in\RR^2$ such that 
\begin{align*}
|\phi_\infty-\phi_\infty^0|&\leq C\,\Norm{v_0}_{W^{1,\infty}(\RR^\star)}\,,&
|\phi_{{\rm as},\infty}|&\leq C\,\Norm{v_0}_{W^{1,\infty}(\RR^\star)}\,,
\end{align*}
and, for any $t\geq 0$,
\begin{align*}
\abs{\phi(t)-(\phi_{{\rm as},\infty}+\phi_{\rm as}(t))}&\leq C\,e^{-\theta t}\,\Norm{v_0}_{W^{1,\infty}(\RR^\star)}\,,&
\abs{\phi(t)-(\phi_\infty+f'(\uu_\infty)\,t)}&\leq C\,e^{-\theta t}\,,\\
\abs{\phi'(t)-\phi_{\rm as}'(t)}&\leq C\,e^{-\theta t}\,\Norm{v_0}_{W^{1,\infty}(\RR^\star)}\,,&
\abs{\phi'(t)-f'(\uu_\infty)}&\leq C\,e^{-\theta t}\,,
\end{align*}
with $\uu_\infty=\uu_{+\infty}$ (resp. $\uu_\infty=\uu_{-\infty}$), $\phi_{\rm as}$ the solution to
\[
\phi_{\rm as}(0)=\psi_0+\delta_0
\qquad\textrm{and}\qquad
(\forall t\geq0\,,\quad 
\phi_{\rm as}'(t)=f'(\uU(\phi_{\rm as}(t)-(\sigma t+\psi_\infty))))\,,
\]
and
\[
\phi_\infty^0\eqdef \psi_0+\delta_0+\int_{\delta_0}^{+\infty}\left(1-\frac{f'(\uu_{+\infty})-\sigma}{f'(\uU(\xi))-\sigma}\right)\,\dd \xi\qquad
\Bigg(\textrm{resp.}\quad+\int^{\delta_0}_{-\infty}\left(\frac{f'(\uu_{-\infty})-\sigma}{f'(\uU(\xi))-\sigma}-1\right)\,\dd \xi\Bigg)
\]
\end{enumerate}
\end{Proposition}

We point out that~\eqref{perturbation0-shock} could be replaced with the more patently perturbative
\[
\Norm{v_0}_{W^{1,\infty}(\RR^\star)}\leq \eps\,,\qquad\qquad
\sign(f''(\uU(\delta_0))\,\uU'(\delta_0))\times(v_0(0^+)-v_0(0^-))>0\,.
\]

It is also clear from the proof that also holds a higher-order version of Proposition~\ref{P.small-shock}, in the spirit of Proposition~\ref{P.high-order-front}.

\begin{proof}
Again, by using invariance by translation, we reduce to the case $\psi_0=0$. Furthermore, since both cases are completely analogous, we restrict to the case $\delta_0<0$. 

The strategy of the proof is the following. Given $v_0$ satisfying~\eqref{perturbation0-shock}, we define two extensions $v_{0,\pm}$, defined on $\RR$, satisfying $v_{0,+}=v_0(\cdot-\delta_0)$ on $(\delta_0,+\infty)$ and $v_{0,-}=v_0(\cdot-\delta_0)$ on $(-\infty,\delta_0)$, and fulfilling the hypotheses of Theorem~\ref{T.stable-front}. We may then consider $u_\pm$ the two global unique classical solutions to~\eqref{eq-u}  emerging from the initial data $u_\pm\id{t=0}=\uU +v_{0,\pm}$. The solution $u$ is constructed by patching together $u_+$ and $u_-$ along the curve $t\mapsto(t,\phi(t))$ defined through the Rankine-Hugoniot condition.

We first provide an extension lemma.

\begin{Lemma}\label{L.extension-small-shock}
\begin{enumerate}
\item There exists $C'\geq1$ such that for any $\delta_0<0$ and any $w_{0}\in BUC^1((0,+\infty))$, there exists $v_{0,+}\in BUC^1(\RR)$ such that 
\begin{align*}
v_{0,+}\id{(\delta_0,+\infty)}&=w_0(\cdot-\delta_0),&\qquad
\Norm{v_{0,+}}_{W^{1,\infty}(\RR)}
&\leq C'\,\Norm{w_0}_{W^{1,\infty}((0,+\infty))}\,.
\end{align*}
\item For any $\delta_0<0$, there exists $C'\geq1$ such that for any $v_{0}\in BUC^1(\RR^\star)$ and any $\psi_\infty>\delta_0/2$, there exists $v_{0,-}\in BUC^1(\RR)$ such that
\begin{align*} 
v_{0,-}\id{(-\infty,\delta_0)}&=(v_0\id{(-\infty,0)})(\cdot-\delta_0),&\qquad
\Norm{v_{0,-}}_{W^{1,\infty}(\RR)}
&\leq C'\,\Norm{v_0}_{W^{1,\infty}(\RR^\star)}\,,
\end{align*} 
and $v_{0,-}(\psi_\infty)=v_0(\psi_\infty)$.
\end{enumerate}
\end{Lemma}

The foregoing lemma is a variation on classical and easy-to-prove extension lemmas (see for instance~\cite{Adams}) and its proof is left to the reader. Now we apply its first part with $w_0=v_0\id{(0,+\infty)}$ and obtain some $v_{0,+}$. Then we apply Theorem~\ref{T.stable-front} with initial perturbation $v_{0,+}$ and receive a corresponding solution $u_+$ and a corresponding (asymptotic) shift $\psi_\infty$. Note that by taking $\eps$ sufficiently small we may ensure $|\psi_\infty|\leq\delta_0/4$. Afterwards we apply the second part of the extension lemma to $v_0$ and $\psi_\infty$ and receive some $v_{0,-}$. At last we apply again Theorem~\ref{T.stable-front} this time with initial perturbation $v_{0,-}$ and receive a corresponding solution $u_-$, with the asymptotic shift $\psi_\infty$ prescribed by $v_{0,+}$ since $\uU(\psi_\infty)+v_{0,-}(\psi_\infty)=\uU(\psi_\infty)+v_{0,+}(\psi_\infty)=\uu_\star$.

We shall construct our solution, $u$, through the following formula
\begin{equation}\label{patch}
u(t,x)=
\begin{cases} 
u_-(t,x) & \text{ if } x<\phi(t),\\
u_+(t,x) & \text{ if } x>\phi(t),
\end{cases}
\end{equation}
where the discontinuity curve, described by $\phi$, is defined through the Rankine-Hugoniot condition
\[
\left(u_+(t,\phi(t))-u_-(t,\phi(t))\right)\phi'(t)\,=\, f(u_+(t,\phi(t)))-f(u_-(t,\phi(t)))\,.
\]

To this aim, we introduce the slope function associated with $f$,
\begin{equation}\label{slope-function}
s_f:\,\RR\times\RR\to\RR,\quad (a,b)\mapsto \int_0^1 f'\big(a+\tau(b-a)\big)\,\dd \tau.
\end{equation}
Since $s_f\in \cC^1(\RR\times\RR)$, the map $(t,x)\mapsto s_f(u_-(t,x),u_+(t,x))$ belongs to $BUC^1(\RR_+\times\RR)$, hence there exists a unique $\phi\in\cC^2(\RR_+)$ satisfying $\phi(0)=\delta_0$ and for any $t\geq 0$, 
\[\phi'(t)\,=\, s_f(u_-(t,\phi(t)),u_+(t,\phi(t)))\,.\]
As a result we obtain that $u$ defined by~\eqref{patch} satisfies~\eqref{eq-u} on $\Omega^\phi$, the Rankine-Hugoniot condition along $\{(t,\phi(t));t\geq0\}$, and thus is a weak solution to~\eqref{eq-u}. By design, the initial condition $u(0,\cdot)=\uU+v_0(\cdot-\delta_0)$ also holds. It only remains to study the asymptotic behavior of $\phi$ and the entropy condition.

We now verify the claimed estimates on $\phi$. To begin with we study $\phi_{\rm as}$. First note that since 
\[
\sup_{(-\infty,\frac{3\delta_0}{4}]}f'\circ\uU\ <\ \sigma
\]
a continuity argument shows for some $c>0$ independent of $v_0$ (satisfying the assumptions of the proposition) that for any $t\geq 0$,
\begin{align}\label{ineq:phias-intermediate}
\phi_{\rm as}'(t)&\leq \sigma-c\,,
&\quad\phi_{\rm as}(t)&\leq \delta_0+(\sigma-c)\,t\,.
\end{align}
Moreover it follows from Lemma~\ref{L.front}, that for some $C'$ independent of $v_0$, for any $t\geq0$,
\begin{equation}\label{est-phias}
|\phi_{\rm as}'(t)-f'(\uu_{-\infty})|\leq C'\exp\left( \frac{g'(\uu_{-\infty})}{f'(\uu_{- \infty})-\sigma}(\phi_{\rm as}(t)-(\psi_\infty+\sigma t))\right)\,.
\end{equation}
Inserting~\eqref{ineq:phias-intermediate} in~\eqref{est-phias} shows that for some $C_0>0$ independent of $v_0$, for any $t\geq0$,
\begin{equation}\label{ineq:phias-intermediate-bis}
\phi_{\rm as}(t)\leq \delta_0+C_0+f'(\uu_{-\infty})\,t\,.
\end{equation}
Then by inserting~\eqref{ineq:phias-intermediate-bis} in~\eqref{est-phias} we deduce that for some $C''$ independent of $v_0$, for any $t\geq 0$,
\[
|\phi_{\rm as}'(t)-f'(\uu_{-\infty})|\leq C''\,e^{g'(\uu_{-\infty})\,t}\,.
\]
In particular, integrating the latter shows for some $C'''$ independent of $v_0$, for any $t\geq0$,
\[
|\phi_{\rm as}(t)-(\phi_{\infty,{\rm as}}+f'(\uu_{-\infty})t)|
\leq\, C'''\,e^{g'(\uu_{-\infty})\,t}\,,
\]
with 
\[
\phi_{\infty,{\rm as}}=\delta_0
+\int^{\delta_0-\psi_\infty}_{-\infty}\left(\frac{f'(\uu_{-\infty})-\sigma}{f'(\uU(\xi-\psi_\infty))-\sigma}-1\right)\,\dd \xi\,,\]
so that
\[\abs{\phi_{\infty,{\rm as}}-\phi^0_{\infty}}\leq C'''\,\Norm{v_0}_{L^{\infty}(\RR^\star)}\,.
\]

Now, using again Lemma~\ref{L.front} shows that there exists a constant $C_0'>0$ independent of $v_0$ (satisfying the assumptions of the proposition) such that, for any $t\geq0$,
\begin{align}\label{ineq:phi-intermediate}
\left|\phi'(t)-\phi_{\rm as}'(t)\right|
&\leq C_0'\,e^{-\theta\,t}\,\Norm{v_0}_{W^{1,\infty}(\RR^\star)}
+C_0'\,
\left|\phi(t)-\phi_{\rm as}(t)\right|
\,e^{\frac{g'(\uu_{-\infty})}{f'(\uu_{- \infty})-\sigma}(\max(\{\phi_{\rm as}(t),\phi(t)\})-\sigma t)}\,,
\end{align}
Thus for some constant $C_0''>0$ independent of $v_0$, the Gr\"onwall lemma shows that, if $t$ is such that, for any $0\leq s\leq t$, $\phi(s)\leq \delta_0+C_0+1+f'(\uu_{-\infty})\,s$, then 
\[
\left|\phi(t)-\phi_{\rm as}(t)\right|\leq C_0''\,\Norm{v_0}_{W^{1,\infty}(\RR^\star)}\,.
\]
Hence, enforcing $C_0''\eps<1$, a continuity argument shows that for any $t\geq0$, 
\[
\phi(t)\leq \delta_0+C_0+1+f'(\uu_{-\infty})\,t\,,\qquad
\left|\phi(t)-\phi_{\rm as}(t)\right|\leq C_0''\,\Norm{v_0}_{W^{1,\infty}(\RR^\star)}\,.
\] 
Inserting the latter bounds in~\eqref{ineq:phi-intermediate} provides the expected bound on $\left|\phi'(t)-\phi_{\rm as}'(t)\right|$. At last, integrating this bound and combining with bounds on $\phi_{\rm as}$ conclude the study of $\phi$.

To achieve the proof, we need to ensure that lessening $\eps$ if necessary, the constructed weak solution is an entropy solution. By assumption there is a convex neighborhood of $\overline{\{\uU(x)\ ;\ x\leq\delta_0\}}$ on which $f''$ does not vanish and taking $\eps$ sufficiently small we may ensure that for any $t\geq0$, both $u_+(t,\phi(t))$ and $u_-(t,\phi(t))$ belong to this neighborhood. Thus it is sufficient to check that $w(t)\eqdef u_+(t,\psi(t))- u_-(t,\psi(t))$ has the correct sign for any $t\geq0$. Since $w$ takes the correct sign at $t=0$, it amounts to checking that $w$ does not vanish. This follows from the Cauchy-Lipschitz theorem since $w(0)\neq0$ and $(\forall t\geq0, w'(t)=\Phi(t,w(t)))$ with $\Phi$ such that ($\forall t\geq0$, $\Phi(t,0)=0$). Indeed the latter claim follows from a straightforward computation with $\Phi$ explicitly given by 
\begin{align*}
\Phi\,:\,(t,z)\mapsto&\ s_g(u_+(t,\psi(t)),u_-(t,\psi(t)))\,z\\
&+\Big(s_f(u_+(t,\psi(t)),u_+(t,\psi(t))-z)-s_f(u_+(t,\psi(t)),u_+(t,\psi(t)))\Big)\d_x u_+(t,\psi(t))\\
&-\Big(s_f(u_-(t,\psi(t))+z,u_-(t,\psi(t)))-s_f(u_-(t,\psi(t)),u_-(t,\psi(t)))\Big)\d_x u_-(t,\psi(t)).
\end{align*}
and one readily checks that $\Phi$ is jointly continuous, and uniformly Lipschitz in $z$.
\end{proof}

\subsection{Transverse stability in the multidimensional framework}\label{S.multiD}

In the present subsection we discuss possible generalizations of Theorem~\ref{T.stable-front} to multidimensional settings. In particular we temporarily replace~\eqref{eq-u} with
\begin{equation}\label{eq-u-multiD}
\d_t u+\Div \big(f(u)\big)=g(u)
\end{equation}
where the spatial variable $x$ belongs to $\RR^d$, $d\in\NN$, and $f:\RR\to\RR^d$.

Starting from spatial dimension $2$ the range of possible geometries for discontinuities becomes too wide to be reasonably covered here, even if one restricts to a few typical cases. Therefore, as in~\cite{DR1}, in the multidimensional case we only consider shockless perturbations.

For the sake of clarity and without loss of generality, we fix the direction of propagation of the reference plane front, split spatial variables accordingly $x=(\xi,y)\in\RR\times\RR^{d-1}$, and correspondingly $f=(\fpar,\fperp)$. We consider a plane wave $\uu$, 
\[\uu(t,x)=\uU(\xi-(\psi_0+\sigma t))\,,\] 
with $\psi_0\in\RR$, and $(\uU,\sigma)$ generating for nonlinearities $(\fpar,g)$ a one-dimensional stable continuous front in the sense of Assumption~\ref{A.front}.

Note that the profile of the plane wave $\uu$ takes the characteristic value $\uu_\star$ on an hyperplane. A general initial perturbation may bend this characteristic hyperplane whereas the time dynamics cannot restore the unperturbed shape so that the plane wave $\uu$ cannot be asymptotically stable (even in an orbital sense). The best one may expect is that
\begin{enumerate} 
\item near the plane wave under consideration there exists a genuinely multidimensional family of traveling waves, continuously parameterized by the $\uu_\star$-level set;
\item this family is asymptotically stable in the sense that a solution arising from the perturbation of an element of the family converges to a possibly different element of the same wave family.
\end{enumerate}
We anticipate that the study of the latter would require arguments significantly different from the rest of our other investigations. Yet the reader is referred to Subsection~\ref{S.multiple-characteristics} for the analysis of a similar situation.

To bypass the foregoing, we restrict here to initial perturbations that are localized away from the characteristic hyperplane and thus do not alter its shape. Incidentally let us point out that this restriction is conceptually similar to those made in~\cite{YangZumbrun20} where perturbations are assumed to be initially supported away from discontinuities.

\begin{Theorem}\label{T.multiD}
Let $(\uU,\sigpar)$ be given by Assumption~\ref{A.front} with nonlinearities $(\fpar,g)$, assume that on a neighborhood of $[\uu_-,\uu_+]$, $f$ is $\cC^4$ and $g$ is $\cC^3$, and set
\[\theta\eqdef \min(\{-g'(\uu_{+\infty}),-g'(\uu_{-\infty})\})>0.\]
For any $r_0>0$, there exist $\eps>0$ and $C>0$ such that for any $\psi_0\in\RR$ and any $v_0\in BUC^1(\RR^d)$ satisfying
\begin{align*}
\Norm{v_0}_{W^{1,\infty}(\RR^d)}&\leq \eps\,,&
\supp v_0&\subset\,\{\,(\xi,y)\,;\,|\xi|\geq r_0\,\}\,,
\end{align*}
the initial datum $u\id{t=0}=(\uU+v_0)(\cdot-\psi_0)$ generates a unique global classical solution to~\eqref{eq-u-multiD}, $u\in BUC^1(\RR^+\times\RR^d)$, such that, with $\sigma\eqdef f'(\uu_\star)$, for any $t\geq 0$,
\[
\supp (u(t,\cdot+(\sigma t+\psi_0))-\uU)\,\subset\,\{\,(\xi,y)\,;\,|\xi|\geq r_0\,\}\,,
\]
and
\begin{align*}
\Norm{u(t,\cdot+(\sigma t+\psi_0))-\uU}_{L^{\infty}(\RR^d)}
&\leq\,C\,e^{-\theta t}\,\Norm{v_0}_{L^{\infty}(\RR^d)}\,,\\
\Norm{u(t,\cdot+(\sigma t+\psi_0))-\uU}_{W^{1,\infty}(\RR^d)}
&\leq\,C\,e^{-\theta t}\,\Norm{v_0}_{W^{1,\infty}(\RR^d)}\,.
\end{align*}
\end{Theorem}

\begin{proof}
Since the details of the proof are completely similar to the ones involved in the proof of Theorem~\ref{T.stable-front}, we only outline its main features.

To derive a formulation as perturbative as possible, we introduce first $\Fperp$ and $\Phiperp$ through 
\begin{align*}
\Fperp(u)&\eqdef\begin{cases}\frac{\fperp'(u)-\fperp'(\uu_\star)}{\fpar'(u)-\sigpar}&\textrm{if}\ u\neq \uu_\star\\
\frac{\fperp''(\uu_\star)}{\fpar''(\uu_\star)}&\textrm{otherwise}
\end{cases}\,,&
\Phiperp(\xi)&\eqdef\int_0^\xi\,\Fperp(\uU(\eta))\,\dd\eta\,,
\end{align*}
so as to consider $v$ defined from the solution to study, $u$, by
\[
v(t,x)=u(t,x+\sigma\,t+(0,\Phiperp(\xi)))-\uU(\xi)\,.
\]
Then we observe that from~\eqref{eq-u-multiD} stems
\[
\d_tv(t,\cdot)-\cL_{\fpar'(\uU+v(t,\cdot))-\sigpar,\Fperp(\uU+v(t,\cdot))-\Fperp(\uU)}\,v(t,\cdot)
\,=\,\cN(v(t,\cdot)),
\]
where for any function $w$
\begin{align*}
\cN(w)&\eqdef\left(\fpar'(\uU)-\fpar'(\uU+w)\right)\frac{\uU''}{\uU'}\,w+g(\uU+w)-g(\uU)-g'(\uU)w\\
&\qquad-\left(\fpar'(\uU+w)-\fpar'(\uU)-\fpar''(\uU)w\right)\uU'\,.
\end{align*}
and for any functions $a$ and $A$
\[
\cL_{a,A}w\eqdef -a\,\left(\d_\xi w
+A\cdot \nabla_y w-\frac{\uU''}{\uU'} w\right)
\,=\,
-a\,\left(\uU'\d_\xi\left(\frac{w}{\uU'}\right)
+A\cdot \nabla_y w\right)
\,.
\]
From the latter follows, for $1\leq\ell\leq (d-1)$,
\begin{align*}
\d_t&(\d_{y_\ell}v(t,\cdot))-\cL_{\fpar'(\uU+v(t,\cdot))-\sigpar,\Fperp(\uU+v(t,\cdot))-\Fperp(\uU)}\,(\d_{y_\ell}v(t,\cdot))\\
&\,=\,\d_{y_\ell}(\cN(v(t,\cdot)))
-(\fpar'(\uU+v)-\sigpar)\,\d_{y_\ell}v(t,\cdot)\,\Fperp'(\uU+v(t,\cdot))\cdot \nabla_y v(t,\cdot)\\
&-\fpar''(\uU+v)\,\d_{y_\ell}v(t,\cdot)\,\left(\d_\xi v(t,\cdot)
+(\Fperp(\uU+v(t,\cdot))-\Fperp(\uU))\cdot \nabla_y v(t,\cdot)-\frac{\uU''}{\uU'} v(t,\cdot)\right),
\end{align*}
and
\begin{align*}
\d_t&\left(\uU'\d_\xi\left(\frac{v(t,\cdot)}{\uU'}\right)\right)
-L_{\fpar'(\uU+v(t,\cdot))-\sigpar,\Fperp(\uU+v(t,\cdot))-\Fperp(\uU)}\,\left(\uU'\d_\xi\left(\frac{v(t,\cdot)}{\uU'}\right)\right)\\
&\,=\,
\uU'\d_\xi\left(\frac{\cN(v(t,\cdot))}{\uU'}\right)
-\d_\xi\left[(\fpar'(\uU+v)-\sigpar)\,(\Fperp(\uU+v(t,\cdot))-\Fperp'(\uU))\right]\cdot \nabla_y v(t,\cdot),
\end{align*}
with 
\[
L_{a,A}w\eqdef -a\,\left(\d_\xi w+A\cdot \nabla_y w\right)
+\left(a\frac{\uU''}{\uU'}-\d_\xi a\right)w\,.
\]

To close the proof along the arguments of the proof of Theorem~\ref{T.stable-front}, it is sufficient to prove sharp exponential decay for evolution systems generated by $L_{a(t,\cdot),A(t,\cdot)}$ and $\cL_{a(t,\cdot),A(t,\cdot)}$ acting on 
\[
BUC^0_{r_0}(\RR^d)\eqdef
\left\{\,v\in BUC^0(\RR^d)\,;\,
\supp v\subset\,\{\,(\xi,y)\,;\,|\xi|\geq r_0\,\}\right\}
\]
when both $a(t,\cdot)-(\fpar'(\uU)-\sigpar)$ and $A(t,\cdot)$ are sufficiently small uniformly in time. In turn the latter is essentially a consequence of the following lemma.
\end{proof}

\begin{Lemma}
Let $(\uU,\sigpar)$ be given by Assumption~\ref{A.front} with nonlinearities $(\fpar,g)$, assume that on a neighborhood of $[\uu_-,\uu_+]$, $f$ is $\cC^4$ and $g$ is $\cC^3$, and set
\[\theta\eqdef \min(\{-g'(\uu_{+\infty}),-g'(\uu_{-\infty})\})>0.\]
For any $r_0>0$, there exist $\eps>0$, $C>0$ and $\chi, \tchi\in L^1(\RR)\cap BUC^0(\RR)$ such that for any $(a,A)\in \cC^1(\RR^d)\cap W^{1,\infty}(\RR^d)$ satisfying
\begin{align*}
\Norm{a-(\fpar'(\uU)-\sigpar)}_{W^{1,\infty}(\RR^d)}
\,\leq\,\eps
\end{align*}
the following holds with
\[
\theta_a\eqdef\theta-C\Norm{a-(\fpar'(\uU)-\sigpar)}_{W^{1,\infty}(\RR^d)}
\]
\begin{enumerate}
\item For any $\lambda\in\CC$ such that 
\[
\Re(\lambda)>-\theta_a\,,
\] 
and any $\alpha\in BUC_{r_0}^0(\RR^d)$, there exists a unique $v\in BUC_{r_0}^0(\RR^d)$ such that 
\[
(\lambda-\cL_{a,A})\,v\,=\,\alpha
\]
and, moreover,
\[
\Norm{x\mapsto e^{-\int_{0}^{\xi}\chi}v(x)}_{L^{\infty}(\RR^d)}
\,\leq\,\frac{1}{\Re(\lambda)+\theta_{a}}\,\Norm{x\mapsto e^{-\int_{0}^{\xi}\chi}\alpha(x)}_{L^{\infty}(\RR^d)}\,.
\]
\item For any $\lambda\in\CC$ such that 
\[
\Re(\lambda)>-\theta_a\,,
\] 
and any $\alpha\in BUC_{r_0}^0(\RR^d)$, there exists a unique $v\in BUC_{r_0}^0(\RR^d)$ such that 
\[
(\lambda-L_{a,A})\,v\,=\,\alpha
\]
and, moreover,
\[
\Norm{x\mapsto e^{-\int_{0}^{\xi}\tchi}v(x)}_{L^{\infty}(\RR^d)}
\,\leq\,\frac{1}{\Re(\lambda)+\theta_{a}}\,\Norm{x\mapsto e^{-\int_{0}^{\xi}\tchi}\alpha(x)}_{L^{\infty}(\RR^d)}\,.
\]
\end{enumerate}
\end{Lemma}

\begin{proof}
One may proceed as in the one-dimensional case, with generalized formula
\begin{align*}
e^{-\int_{0}^{\xi}\chi}
v(x)&\eqdef \int_{-\infty}^0 e^{\int_s^0\left(\left(a\frac{\uU''}{\uU'}-a\chi\right)(\Xi(\tau;x))-\lambda\right)\,\dd \tau}
e^{-\int_{0}^{\Xi(s;x)}\chi}\,
\alpha(X(s;x))\,\dd s\,,&
\textrm{in case i.}\,,\\
e^{-\int_{0}^{\xi}\tchi}
v(x)&\eqdef \int_{-\infty}^0 e^{\int_s^0\left(\left(a\frac{\uU''}{\uU'}-\d_\xi a-a\tchi\right)(\Xi(\tau;x))-\lambda\right)\,\dd \tau}
e^{-\int_{0}^{\Xi(s;x)}\tchi}\,
\alpha(X(s;x))\,\dd s\,,&
\textrm{in case ii.}\,,
\end{align*}
where $X(\cdot;x)=(\Xi,Y)(\cdot;x)$ is such that $X(0;x)=x$ and 
\[
\forall s\in\RR,\qquad\d_sX(s;x)=a(X(s;x))\,(1,A(X(s;x)))\,.
\]
\end{proof}

\section{General stable waves}

In this section we extend our stability results initiated in Section~\ref{S.continous_fronts} to stable waves of a more general form. We first consider in Section~\ref{S.other_classes} classes of stable waves involved in Theorem~\ref{T.classification} and thus conclude its proof.
Then, relaxing Assumption~\ref{A.generic}, we consider in Section~\ref{S.multiple-characteristics} some waves possessing several characteristic points.

\medskip

\subsection{Stable waves of generic equations}\label{S.other_classes}

To begin with, we recall some results from~\cite{DR1}, namely~\cite[Proposition~2.2]{DR1} (for constant states) and\footnote{Actually one of the variants of~\cite[Theorem~3.2]{DR1} along the lines of~\cite[Remark~3.3]{DR1}.}~\cite[Theorem~3.2]{DR1} (for Riemann shocks).

\begin{Proposition}[\cite{DR1}]\label{P.constant-classical}
Let $\uu\in\RR$ and $f$, $g$ $\cC^2$ in a neighborhood of $\uu$ such that 
\[
g(\uu)=0\qquad\textrm{and}\qquad g'(\uu)<0\,.
\]
Then for any $C_0>1$, there exists $\eps>0$ such that for any $v_0\in BUC^1(\RR)$ satisfying 
\[\Norm{v_0}_{W^{1,\infty}(\RR)}\leq \eps\,,\]
the initial data $u\id{t=0}=\uu+v_0$ generates a global unique classical solution to~\eqref{eq-u}, ${u\in BUC^1(\RR^+\times\RR)}$, and it satisfies for any $t\geq0$
\begin{align*}
\Norm{u(t,\cdot)-\uu}_{L^{\infty}(\RR)}&\leq \Norm{v_0}_{L^{\infty}(\RR)}C_0\,e^{g'(\uu)\,t}\ ;\\
\Norm{\d_x u(t,\cdot)}_{L^{\infty}(\RR)}&\leq \Norm{\d_x v_0}_{L^{\infty}(\RR)} C_0\,e^{g'(\uu)\,t}\,.
\end{align*}
\end{Proposition}

\begin{Assumption}\label{A.Riemann}
Assume that $(\uu_-,\uu_+)\in\RR^2$, $\uu_-<\uu_+$, that on a neighborhood of $[\uu_-,\uu_+]$, $f$ and $g$ are $\cC^2$, and that the following conditions hold
\begin{enumerate}
\item $ $\\\vspace{-3em}
\begin{align*}
g(\uu_-)&=0\,,
&g(\uu_+)&=0\,,\\
g'(\uu_-)&<0\,,
&g'(\uu_+)&<0\,;
\end{align*}
\item  with $\sigma\eqdef (f(\uu_+)-f(\uu_-))/(\uu_+-\uu_-)$, we have $(f'(\uu_+)-\sigma)(f'(\uu_+)-\sigma)<0$ and labeling $\{\uu_{-\infty},\uu_{+\infty}\}=\{\uu_-,\uu_+\}$ according to
\begin{align*}
f'(\uu_{-\infty})&>\sigma\,,
&f'(\uu_{+\infty})&<\sigma\,,
\end{align*}
for any $u \in (\uu_-,\uu_+)$, 
\[
\frac{f(u)-f(\uu_{-\infty})}{u-\uu_{-\infty}}>
\frac{f(u)-f(\uu_{+\infty})}{u-\uu_{+\infty}}\,.
\]
\end{enumerate}
Then we define $\uU$ as for any $x\in\RR$, 
\[
\uU(x)=\begin{cases}
\uu_{-\infty} &\text{ if }  x<0\\
\uu_{+\infty} &\text{ if }  x>0
\end{cases}\,.
\]
\end{Assumption}

\begin{Theorem}[\cite{DR1}]\label{T.shock}
Let $(\uU,\sigma)$ be given by Assumption~\ref{A.Riemann}. For any $C_0>1$, there exist $\eps>0$ and $C>0$ such that for any $\psi_0\in\RR$ and $v_0\in BUC^1(\RR^\star)$ satisfying
\[
\begin{array}{rl}
\Norm{v_0}_{W^{1,\infty}(\RR^\star)}&\leq\eps\,,
\end{array}
\]
there exists $\psi\in\cC^2(\RR^+)$ with initial data $\psi(0)=\psi_0$ such that the entropy solution to~\eqref{eq-u}, $u$, generated by the initial data $u(0,\cdot)=(\uU+v_0)(\cdot-\psi_0)$ is global, belongs to $BUC^1(\Omega^\psi)$ and satisfies for any $t\geq 0$
\begin{align*}
\Norm{u(t,\cdot+\psi(t))-\uu_{\pm\infty}}_{L^{\infty}(\RR^\pm)}&\leq \Norm{v_0}_{L^{\infty}(\RR^\pm)} C_0\,e^{g'(\uu_{\pm\infty})\,t}\,,\\
\Norm{\d_x u(t,\cdot+\psi(t))}_{L^{\infty}(\RR^\pm)}&\leq \Norm{\d_xv_0}_{L^{\infty}(\RR^\pm)} C_0\,e^{g'(\uu_{\pm\infty})\,t}\,,
\end{align*}
and moreover there exists $\psi_\infty$ such that 
\[
\abs{\psi_\infty-\psi_0}\,\leq \Norm{v_0}_{L^{\infty}(\RR^\star)} C\,,
\]
and for any $t\geq 0$
\[
\abs{\psi(t)-(\psi_\infty+t\,\sigma)}\,\leq \Norm{v_0}_{L^{\infty}(\RR^\star)} C\, e^{\max(\{g'(\uu_{+\infty}),g'(\uu_{-\infty})\})\, t}\,.
\]
\end{Theorem}

The reader is referred to~\cite{DR1} for other versions of the foregoing stability results including perturbations with small shocks, higher-regularity descriptions and multi-dimensional counterparts. 

\bigskip

The remaining stable waves involved in Theorem~\ref{T.classification} are neither continuous nor piecewise constant, and as such involve both characteristic points and discontinuities. For this kind of pattern, even when initial perturbations are smooth and supported away from discontinuities, we need to apply more than a simple uniform translation so as to synchronize the perturbed solution with the background wave since both characteristic points and discontinuity locations require fitting. This leads to results of space-modulated asymptotic stability instead of orbital asymptotic stability.

For the reader's convenience, we collect in Assumptions~\ref{A.mixed-single},~\ref{A.mixed-single-bis} and~\ref{A.mixed-double} the detailed conditions on $(f,g)$ from which arises the existence of such stable waves, and represent each case in Figure~\ref{F.composites}.

\begin{figure}[!phtb]
\begin{center}
\begin{subfigure}{.6\textwidth}
\includegraphics[width=\textwidth]{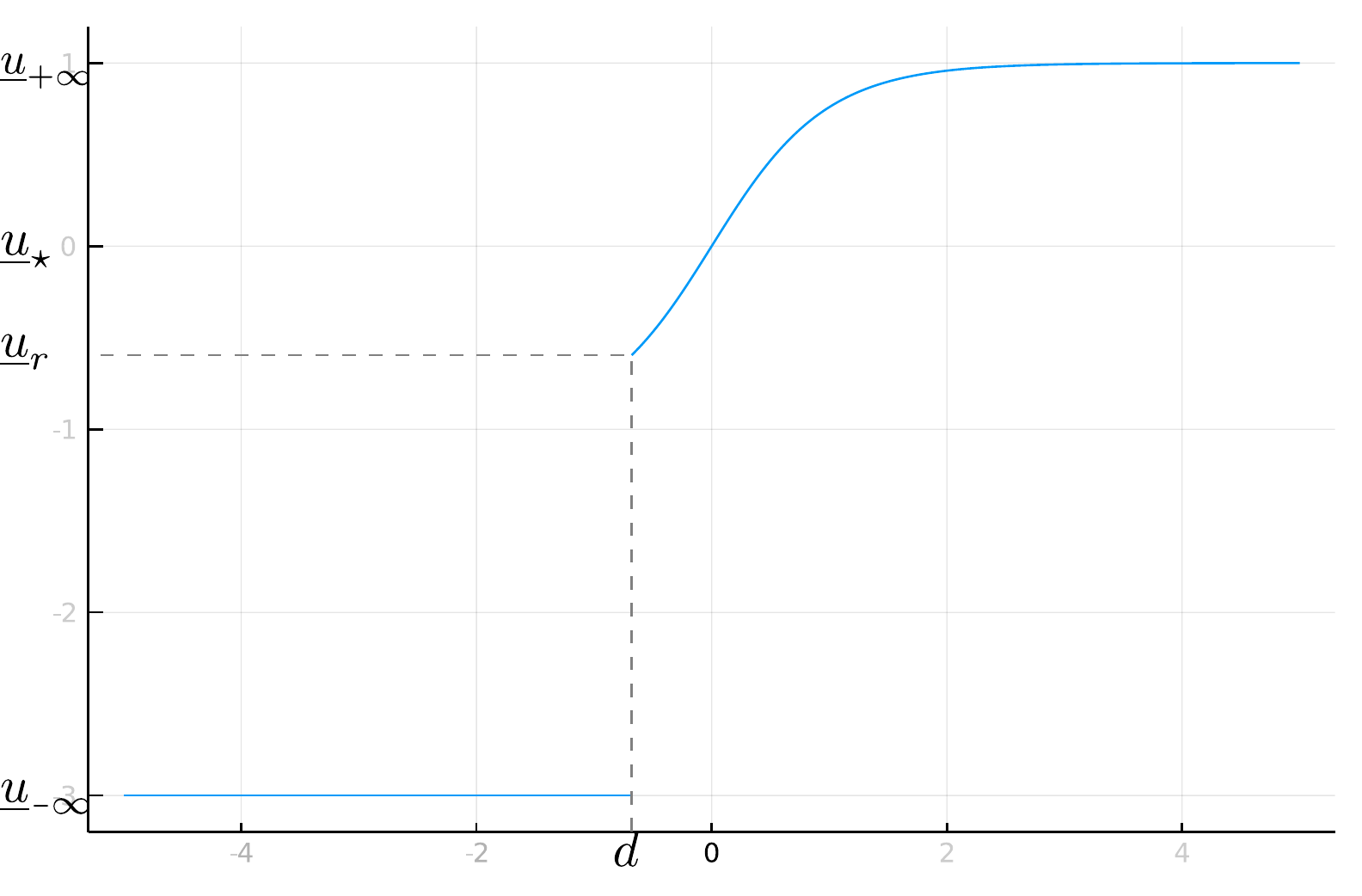}
\caption{A profile satisfying Assumptions~\ref{A.mixed-single}.}
\end{subfigure}
\begin{subfigure}{.6\textwidth}
\includegraphics[width=\textwidth]{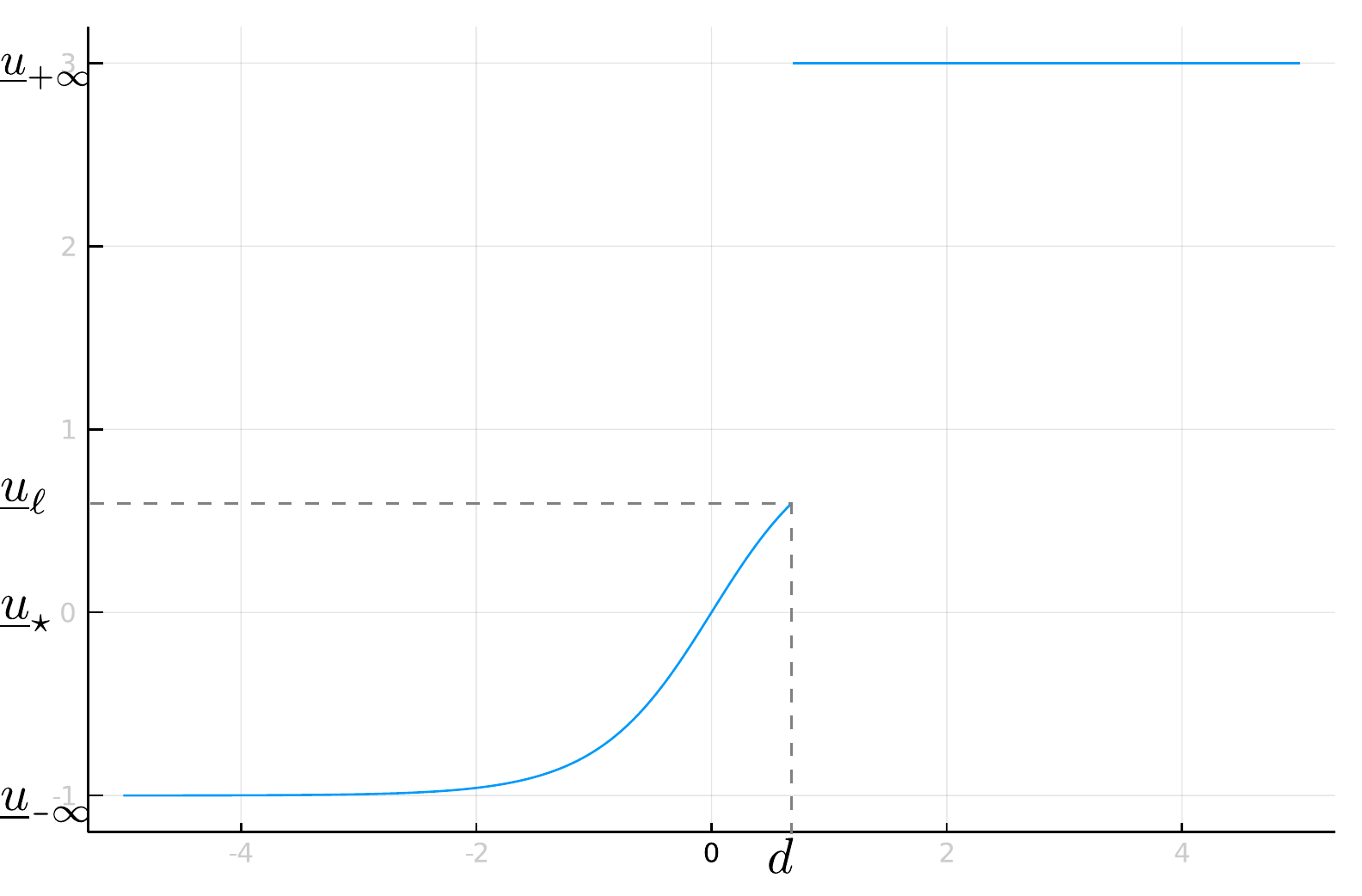}
\caption{A profile satisfying Assumptions~\ref{A.mixed-single-bis}.}
\end{subfigure}
\begin{subfigure}{.6\textwidth}
\includegraphics[width=\textwidth]{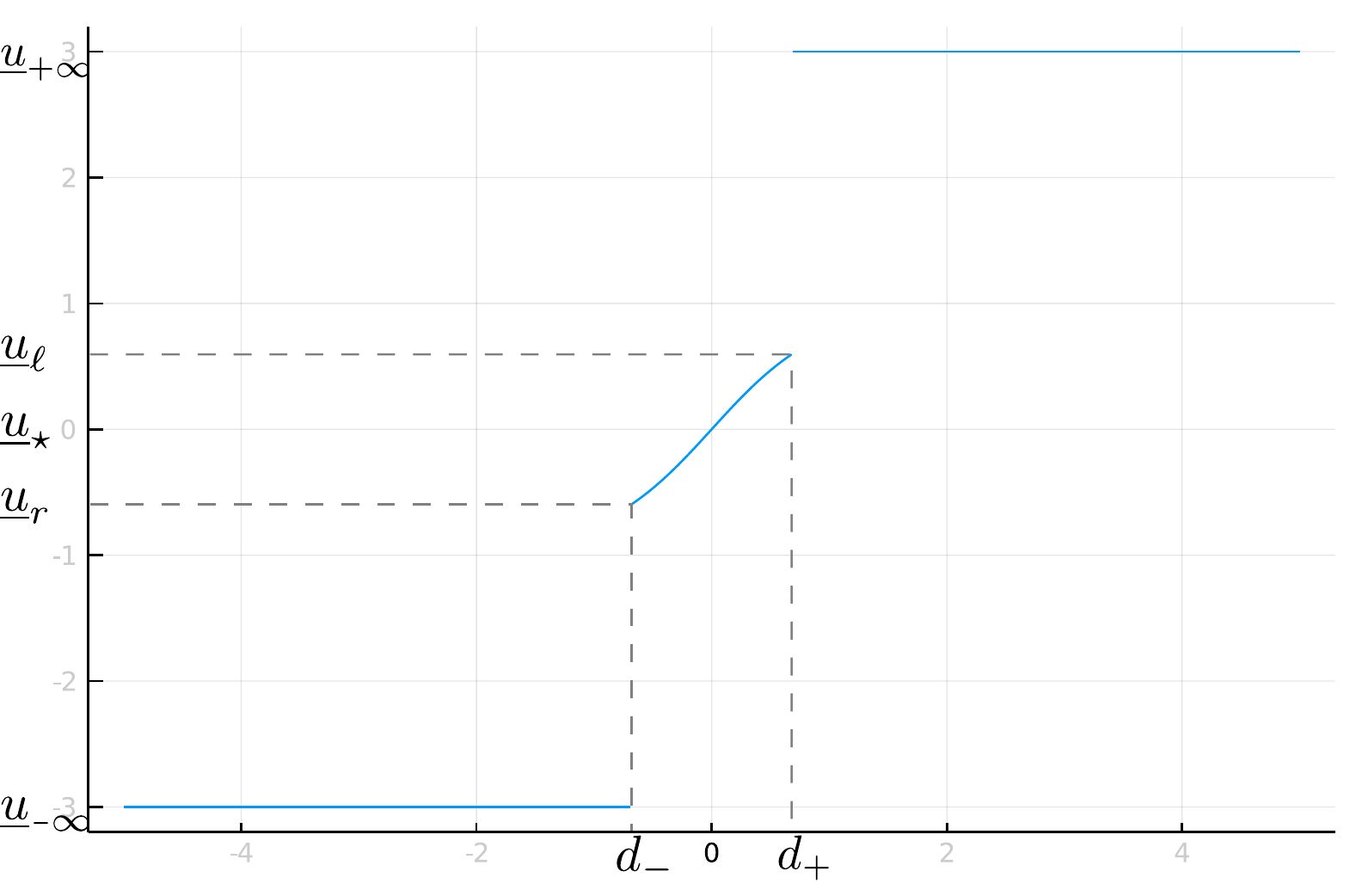}
\caption{A profile satisfying Assumptions~\ref{A.mixed-double}.}
\end{subfigure}
\end{center}
\caption{Stable waves of class~\ref{class4} or~\ref{class5} in Theorem~\ref{T.classification}. The functions $f$ and $g$ used to trace the profiles are as in Figure~\ref{F.classes}, specifically, $f(u) = -\cos(\tfrac74\,u)$ and  $g(u) = \sin(\pi\,u)$. }
\label{F.composites}
\end{figure}

\begin{Assumption}\label{A.mixed-single}
Let $(\uu_{-\infty},\uu_r,\uu_\star,\uu_{+\infty})\in\RR^4$ be 4 distinct values and assume that on a neighborhood of $[\min(\{\uu_{-\infty},\uu_r,\uu_{+\infty}\}),\max(\{\uu_{-\infty},\uu_r,\uu_{+\infty}\})]$, $f$ is $\cC^4$ and $g$ is $\cC^3$, and that the following conditions hold
\begin{enumerate}
\item $ $\\\vspace{-3em}
\begin{align*}
g(\uu_{-\infty})&=0\,,
&g(\uu_\star)&=0\,,
&g(\uu_{+\infty})&=0\,,\\
g'(\uu_{-\infty})&<0\,,
&g'(\uu_\star)&>0\,,
&g'(\uu_{+\infty})&<0\,;
\end{align*}
\item for any $u \in [\min(\{\uu_r,\uu_{+\infty}\}),\max(\{\uu_r,\uu_{+\infty}\})]\setminus\{\uu_\star,\uu_{+\infty}\}$, $g(u)\neq 0$;
\item $f''(\uu_\star)\neq0$, and, for any $u \in [\min(\{\uu_r,\uu_{+\infty}\}),\max(\{\uu_r,\uu_{+\infty}\})]\setminus\{\uu_\star\}$, $f'(u)\neq f'(\uu_\star)$;
\item  with $\sigma\eqdef f'(\uu_\star)$, we have 
\begin{align*}
f(\uu_r)-f(\uu_{-\infty})&=\sigma\,(\uu_r-\uu_{-\infty})\,,
&\frac{g(\uu_r)}{\uu_r-\uu_{-\infty}}&<0\,,\\
f'(\uu_{-\infty})&>\sigma\,,
&f'(\uu_r)&<\sigma\,,
\end{align*}
and, for any $u \in (\min(\{\uu_{-\infty},\uu_r\}),\max(\{\uu_{-\infty},\uu_r\}))$, 
\[
\frac{f(u)-f(\uu_{-\infty})}{u-\uu_{-\infty}}>
\frac{f(u)-f(\uu_r)}{u-\uu_r}\,.
\]
\end{enumerate}
Then we define 
\begin{enumerate}
\item $\uU_+$ on an open interval $\uI_+$ as the maximal solution to $\uU_+(0)=\uu_\star$ and, for any $x\in \uI_+$, 
\[
\uU_+'(x)=F_\sigma (\uU_+(x))\,,
\]
where $F_\sigma$ is defined on a neighborhood of $[\min(\{\uu_r,\uu_{+\infty}\}),\max(\{\uu_r,\uu_{+\infty}\})]$ as in~\eqref{def-F};
\item $d\in(-\infty,0)\cap \uI_+$ by $\uU_+(d)=\uu_r$; 
\end{enumerate}
and we set $D=\{d\}$ and, for $x\in\RR$, 
\[
\uU(x)=\begin{cases}
\uu_{-\infty} &\text{ if }  x<d\\
\uU_+(x) &\text{ if }  x>d
\end{cases}\,.
\]
\end{Assumption}

Note that in the foregoing Assumption the fact that $d$ is well-defined stems from the strict monotonicity of $\uU_+$ and sign considerations on $f'-\sigma$. Likewise one may check that $[d,+\infty)\subset \uI_+$ and $\lim_{x\to+\infty}\uU_+(x)=\uu_{+\infty}$.

\begin{Assumption}\label{A.mixed-single-bis}
Let $(\uu_{-\infty},\uu_\star,\uu_\ell,\uu_{+\infty})\in\RR^4$ be 4 distinct values and assume that on a neighborhood of $[\min(\{\uu_{-\infty},\uu_\ell,\uu_{+\infty}\}),\max(\{\uu_{-\infty},\uu_\ell,\uu_{+\infty}\})]$, $f$ is $\cC^4$ and $g$ is $\cC^3$, and that the following conditions hold
\begin{enumerate}
\item $ $\\\vspace{-3em}
\begin{align*}
g(\uu_{-\infty})&=0\,,
&g(\uu_\star)&=0\,,
&g(\uu_{+\infty})&=0\,,\\
g'(\uu_{-\infty})&<0\,,
&g'(\uu_\star)&>0\,,
&g'(\uu_{+\infty})&<0\,;
\end{align*}
\item for any $u \in [\min(\{\uu_{-\infty},\uu_\ell\}),\max(\{\uu_{-\infty},\uu_\ell\})]\setminus\{\uu_\star,\uu_{-\infty}\}$, $g(u)\neq 0$;
\item $f''(\uu_\star)\neq0$, and, for any $u \in [\min(\{\uu_{-\infty},\uu_\ell\}),\max(\{\uu_{-\infty},\uu_\ell\})]\setminus\{\uu_\star\}$, $f'(u)\neq f'(\uu_\star)$;
\item  with $\sigma\eqdef f'(\uu_\star)$, we have 
\begin{align*}
f(\uu_{+\infty})-f(\uu_\ell)&=\sigma\,(\uu_{+\infty}-\uu_\ell)\,,
&\frac{-g(\uu_\ell)}{\uu_{+\infty}-\uu_\ell}&<0\,,\\
f'(\uu_\ell)&>\sigma\,,
&f'(\uu_{+\infty})&<\sigma\,,
\end{align*}
and, for any $u \in (\min(\{\uu_\ell,\uu_{+\infty}\}),\max(\{\uu_\ell,\uu_{+\infty}\}))$, 
\[
\frac{f(u)-f(\uu_\ell)}{u-\uu_\ell}>\frac{f(u)-f(\uu_{+\infty})}{u-\uu_{+\infty}}\,.
\]
\end{enumerate}
Then we define 
\begin{enumerate}
\item $\uU_-$ on an open interval $\uI_-$ as the maximal solution to $\uU_-(0)=\uu_\star$ and, for any $x\in \uI_-$, 
\[
\uU_-'(x)=F_\sigma (\uU_-(x))\,,
\]
where $F_\sigma$ is defined on a neighborhood of $[\min(\{\uu_{-\infty},\uu_\ell\}),\max(\{\uu_{-\infty},\uu_\ell\})]$ as in~\eqref{def-F};
\item $d\in(0,+\infty)\cap \uI_-$ by $\uU_-(d)=\uu_\ell$; 
\end{enumerate}
and we set $D=\{d\}$ and, for $x\in\RR$, 
\[
\uU(x)=\begin{cases}
\uu_{+\infty} &\text{ if }  x>d\\
\uU_-(x) &\text{ if }  x<d
\end{cases}\,.
\]
\end{Assumption}

\begin{Assumption}\label{A.mixed-double}
Let $(\uu_{-\infty},\uu_r,\uu_\star,\uu_\ell,\uu_{+\infty})\in\RR^5$ be 5 distinct values and assume that on a neighborhood of $[\min(\{\uu_{-\infty},\uu_r\uu_\ell,\uu_{+\infty}\}),\max(\{\uu_{-\infty},\uu_r,\uu_\ell,\uu_{+\infty}\})]$, $f$ is $\cC^4$ and $g$ is $\cC^3$, and that the following conditions hold
\begin{enumerate}
\item $ $\\\vspace{-3em}
\begin{align*}
g(\uu_{-\infty})&=0\,,
&g(\uu_\star)&=0\,,
&g(\uu_{+\infty})&=0\,,\\
g'(\uu_{-\infty})&<0\,,
&g'(\uu_\star)&>0\,,
&g'(\uu_{+\infty})&<0\,;
\end{align*}
\item for any $u \in [\min(\{\uu_r,\uu_\ell\}),\max(\{\uu_r,\uu_\ell\})]\setminus\{\uu_\star\}$, $g(u)\neq 0$;
\item $f''(\uu_\star)\neq0$, and, for any $u \in [\min(\{\uu_r,\uu_\ell\}),\max(\{\uu_r,\uu_\ell\})]\setminus\{\uu_\star\}$, $f'(u)\neq f'(\uu_\star)$;
\item  with $\sigma\eqdef f'(\uu_\star)$, we have 
\begin{align*}
f(\uu_{+\infty})-f(\uu_\ell)&=\sigma\,(\uu_{+\infty}-\uu_\ell)\,,
&\frac{-g(\uu_\ell)}{\uu_{+\infty}-\uu_\ell}&<0\,,\\
f'(\uu_\ell)&>\sigma\,,
&f'(\uu_{+\infty})&<\sigma\,,
\end{align*}
and, for any $u \in (\min(\{\uu_\ell,\uu_{+\infty}\}),\max(\{\uu_\ell,\uu_{+\infty}\}))$, 
\[
\frac{f(u)-f(\uu_\ell)}{u-\uu_\ell}>\frac{f(u)-f(\uu_{+\infty})}{u-\uu_{+\infty}}\,;
\]
\item with the same $\sigma$, we also have 
\begin{align*}
f(\uu_r)-f(\uu_{-\infty})&=\sigma\,(\uu_r-\uu_{-\infty})\,,
&\frac{g(\uu_r)}{\uu_r-\uu_{-\infty}}&<0\,,\\
f'(\uu_{-\infty})&>\sigma\,,
&f'(\uu_r)&<\sigma\,,
\end{align*}
and, for any $u \in (\min(\{\uu_{-\infty},\uu_r\}),\max(\{\uu_{-\infty},\uu_r\}))$, 
\[
\frac{f(u)-f(\uu_{-\infty})}{u-\uu_{-\infty}}>
\frac{f(u)-f(\uu_r)}{u-\uu_r}\,.
\]
\end{enumerate}
Then we define 
\begin{enumerate}
\item $\uU_{\rm int}$ on an open interval $\uI_{\rm int}$ as the maximal solution to $\uU_{\rm int}(0)=\uu_\star$ and, for any $x\in \uI_{\rm int}$, 
\[
\uU_{\rm int}'(x)=F_\sigma (\uU_{\rm int}(x))\,,
\]
where $F_\sigma$ is defined on a neighborhood of $[\min(\{\uu_r,\uu_\ell\}),\max(\{\uu_r,\uu_\ell\})]$ as in~\eqref{def-F};
\item $d_+\in(0,+\infty)\cap \uI_{\rm int}$ by $\uU_{\rm int}(d_+)=\uu_\ell$; 
\item $d_-\in(-\infty,0)\cap \uI_{\rm int}$ by $\uU_{\rm int}(d_-)=\uu_r$;
\end{enumerate}
and we set $D=\{d_-,d_+\}$ and, for $x\in\RR$, 
\[
\uU(x)=\begin{cases}
\uu_{+\infty} &\text{ if }  x>d_+\\
\uu_{-\infty} &\text{ if }  x<d_-\\
\uU_{\rm int}(x) &\text{ if }  d_-<x<d_+
\end{cases}\,.
\]
\end{Assumption}

\begin{Theorem}\label{T.mixed}
Let $(\uU,\sigma,D)$ be given by either Assumption~\ref{A.mixed-single}, Assumption~\ref{A.mixed-single-bis} or Assumption~\ref{A.mixed-double} and set
\[
\theta\eqdef \min\left(\left\{g'(\uu_\star),-g'(\uu_{+\infty}),-g'(\uu_{-\infty})\right\}\cup\left\{-\frac{[g(\uU)]_d}{[\uU]_d};\,d\in D\,\right\}\right)>0.
\]
There exist $\eps>0$ and $C>0$ such that for any $(\psi_0,v_0)\in BUC^1(\RR)\times BUC^1(\RR\setminus D)$ satisfying
\[
\Norm{\d_x\psi_0-1}_{L^{\infty}(\RR)}+\Norm{v_0}_{W^{1,\infty}(\RR\setminus D)}\leq \eps,
\]
the initial datum $u\id{t=0}=(\uU+v_0)\circ \psi_0^{-1}$ generates a unique global entropic solution to~\eqref{eq-u} and there exist $\psi\in BUC^2(\RR_+\times\RR)$ and $\psi_\infty\in\RR$ such that 
\[
\abs{\psi_\infty-\psi_0(0)}\leq 
C\,\left(\Norm{\d_x\psi_0-1}_{L^{\infty}(\RR)}+\Norm{v_0}_{L^{\infty}(\RR)}\right)
\]
and for any $t\geq 0$
\begin{align*}
\psi(t,0)&=\psi_\infty+\sigma\,t\,,
&u(t,\psi_\infty+\sigma t)&=\uu_\star\,,
\end{align*}
$u(t,\psi(t,\cdot))\in BUC^1(\RR\setminus D)$ and
\begin{align*}
\Norm{\d_x\psi(t,\cdot)-1}_{L^{\infty}(\RR)}
+\Norm{u(t,\psi(t,\cdot))-\uU}_{W^{1,\infty}(\RR\setminus D)}
&\leq C\,e^{-\theta t}
\,\left(\Norm{\d_x\psi_0-1}_{L^{\infty}(\RR)}+\Norm{v_0}_{W^{1,\infty}(\RR\setminus D)}\right)\,,\\
\Norm{\d_t\psi(t,\cdot)-\sigma}_{L^{\infty}(\RR)}+\Norm{\psi(t,\cdot)-(\cdot+\sigma t +\psi_\infty)}_{L^{\infty}(\RR)}
&\leq C\,e^{-\theta t}
\,\left(\Norm{\d_x\psi_0-1}_{L^{\infty}(\RR)}+\Norm{v_0}_{W^{1,\infty}(\RR\setminus D)}\right)\,.
\end{align*}
Moreover, increasing $C$ if necessary, one may enforce that, for any $t\geq0$, $\psi(t,\cdot)-(\cdot+\sigma t +\psi_\infty)$ is supported in any prescribed neighborhood of $0$.
\end{Theorem}

Our strategy of proof is the same as in the proofs of Theorem~\ref{T.shock} in~\cite[Theorem~3.2]{DR1}, or of Proposition~\ref{P.small-shock} in the previous section. We first extend smooth parts of the initial datum to the whole line, apply Theorem~\ref{T.stable-front} and Proposition~\ref{P.constant-classical} to propagate in time these extended initial data, and glue the obtained functions along the shock location determined from the Rankine-Hugoniot condition. Yet here there is a priori no extension of $\uU\id{(d,+\infty)}$ into a stable front of~\eqref{eq-u}. Instead, we shall first perform an \emph{artificial} extension of the background profile itself, based on extensions of nonlinearities. We first state and prove the simple relevant lemma.

\begin{Lemma}\label{L.g-extend}
Let $a''<a'<a<b$, $\alpha<0$ and $h\,:\ [a'',b]\to \RR$ $\cC^3$ such that $h$ is negative on $[a',a]$.\\ 
There exists $\check{h}:\ [a'',b]\to \RR$ $\cC^3$ such that $\check{h}\id{[a',b]}=g\id{[a',b]}$, $\check{h}(a'')=0$, $\check{h}'(a'')=\alpha$ and $\check{h}$ is negative on $(a",a]$.
\end{Lemma}

\begin{proof}
Pick $\chi:\RR\to [0,1]$ smooth such that $\chi_{|\left[a''+\tfrac23(a'-a''),+\infty\right)}\equiv1$ and $\chi_{|\left(-\infty,a''+\tfrac13(a'-a'')\right]}\equiv0$. Then define $\check{h}$ through $\check{h}(x)=(1-\chi(x))\,\alpha\,(x-a)+\chi(x)\,h(x)$.
\end{proof}

Now we prove Theorem~\ref{T.mixed}.

\begin{proof}
For the sake of brevity we only treat the case arising from Assumption~\ref{A.mixed-single}. We stress that the changes needed to deal with other cases are purely notational. Moreover we point out that one may exchange Assumptions~\ref{A.mixed-single} and~\ref{A.mixed-single-bis} by switching $(x,f)$ to $(-x,-f)$. Besides, by using invariance by spatial translation as we have done in proofs of the later section to reduce to the case $\psi_0=0$, we may enforce without loss of generality that $\psi_0(0)=0$.

First, we choose a convex neighborhood of $[\min(\{\uu_r,\uu_{+\infty}\}),\max(\{\uu_r,\uu_{+\infty}\})]$ on which $f$ is $\cC^4$, $g$ is $\cC^3$, $f'-\sigma$ vanishes only at $\uu_\star$ and $g$ vanishes only at $\uu_\star$ and $\uu_{+\infty}$. Then we choose some $\check{\uu}_{-\infty}$ in the foregoing neighborhood so that we may apply Lemma~\ref{L.g-extend} with $\alpha=g'(\uu_{+\infty})$ (and either $(h,a'')=(g,\check{\uu}_{-\infty})$ or $(h,a'')=(-g(-\,\cdot\,),-\check{\uu}_{-\infty})$ depending on the relative positions of $\uu_r$ and $\uu_{+\infty}$) and obtain $\check{g}$ coinciding with $g$ on a convex neighborhood of $[\min(\{\uu_r,\uu_{+\infty}\}),\max(\{\uu_r,\uu_{+\infty}\})]$ and such that $(f,\check{g},\check{\uu}_{-\infty},\uu_\star,\uu_{+\infty})$ defines as in Assumption~\ref{A.front} a stable continuous front $(\check{\uU}_+,\sigma)$. In particular, for some positive $\eps_0$ and $\delta_0$, we have, for $x\in[d-\delta_0,+\infty)$, $\check{\uU}_+(x)=\uU_+(x)$, and for any $u$ such that $\abs{u-\check{\uU}_+(x)}\leq\eps_0$, $\check{g}(u)=g(u)$.

Now, we observe that if $\Norm{\d_x\psi_0-1}_{L^{\infty}(\RR)}\leq 1/2$, then for any $x\in\RR$, $|x|/2\leq \abs{\psi_0(x)}\leq 2\,|x|$. Combining this with the exponential localization of $\uU_+$, we deduce that for some $C'$ (not depending on $\psi_0$)
\begin{align*}
\Norm{\uU\circ\psi_0^{-1}-\uU_+}_{W^{1,\infty}((\psi_0(d),+\infty))}
&\leq C'\,\Norm{\d_x\psi_0-1}_{L^{\infty}(\RR)}
\end{align*}
provided that $\Norm{\d_x\psi_0-1}_{L^{\infty}(\RR)}$ is sufficiently small. Then, as in Lemma~\ref{L.extension-small-shock}, we may extend 
\[
\left(\uU\circ\psi_0^{-1}-\uU_++v_0\circ\psi_0^{-1}\right)\id{(\psi_0(d),+\infty)}
\]
into $\check{v}_{0,+}$ and $\left(v_0\circ\psi_0^{-1}\right)\id{(-\infty,\psi_0(d))}$ into $\check{v}_{0,-}$. Afterwards we apply Proposition~\ref{P.constant-classical} to $(\uu_{-\infty},\check{v}_{0,-})$ and Theorem~\ref{T.stable-front} to $(\check{\uU}_+,\check{v}_{0,+})$. In this way, provided that $\Norm{\d_x\psi_0-1}_{L^{\infty}(\RR)}+\Norm{v_0}_{W^{1,\infty}(\RR)}$ is sufficiently small, for some $C'$ (not depending on $(\psi_0,v_0)$), we receive $u_-$ solving $\eqref{eq-u}$ with initial datum $\uu_{-\infty}+\check{v}_{0,-}$ and $u_+$ solving $\eqref{eq-u}$ with $\check{g}$ instead of $g$ and initial datum $\check{\uU}_++\check{v}_{0,+}$ such that for any $t\geq0$
\begin{align*}
\Norm{u_-(t,\cdot)-\uu_{-\infty}}_{W^{1,\infty}(\RR)}
&\leq C'\,e^{g'(\uu_{-\infty})\,t}\,\Norm{v_0}_{W^{1,\infty}(\RR\setminus D)}\,,
\end{align*}
and
\begin{align*}
&\Norm{u_+(t,\cdot+\sigma\,t+\psi_\infty)-\check{\uU}_+}_{W^{1,\infty}(\RR)}\\
&\qquad\leq C'\,e^{-\min(\{g'(\uu_\star),-g'(\uu_{+\infty})\})\,t}\,
\left(\Norm{\d_x\psi_0-1}_{L^{\infty}(\RR)}+\Norm{v_0}_{W^{1,\infty}(\RR\setminus D)}\right)\,,
\end{align*}
for some $\psi_\infty$ such that
\[
\abs{\psi_\infty}
\leq C'\,\left(\Norm{\d_x\psi_0-1}_{L^{\infty}(\RR)}+\Norm{v_0}_{L^{\infty}(\RR)}\right)\,.
\]
Moreover in the foregoing, for any $t\geq0$, $u_+(t,\sigma t+\psi_\infty)=\uu_\star$.

As in the proof of Proposition~\ref{P.small-shock}, we consider the slope function $s_f$ associated with $f$ as in~\eqref{slope-function} and observe that, the map $(t,x)\mapsto s_f(u_-(t,x),u_+(t,x))$ belongs to $BUC^1(\RR_+\times\RR)$, hence there exists a unique $\phi\in\cC^2(\RR_+)$ satisfying $\phi(0)=\psi_0(d)$ and for any $t\geq 0$, 
\[\phi'(t)\,=\, s_f(u_-(t,\phi(t)),u_+(t,\phi(t)))\,.\]
We shall construct our solution, $u$, as in~\eqref{patch} through the formula
\[
u(t,x)=
\begin{cases} 
u_-(t,x) & \text{ if } x<\phi(t),\\
u_+(t,x) & \text{ if } x>\phi(t)\,.
\end{cases}
\]
Note that $\abs{\phi(0)-d}\leq |d|\,\Norm{\d_x\psi_0-1}_{L^{\infty}(\RR)}$ and if we prove that $\phi(t)$ remains sufficiently close to $\sigma\,t+\psi_\infty+d$ we deduce that $u$ is a weak solution to~\eqref{eq-u} (since the values of $u_+$ used in $u$ lie where $g$ and $\check{g}$ coincide). Likewise, the same condition yields that $(u_-(t,\phi(t)),u_+(t,\phi(t)))$ remains close to $(\uu_{-\infty},\uU_+(d))=(\uu_{-\infty},\uu_r)$, thus proving that $u$ satisfies Oleinik's condition is an entropy solution.

Therefore it only remains to study the asymptotic behavior of $\phi$ and to recast all the proved estimates so as to fit the claims in Theorem~\ref{T.mixed}. In order to study $\phi$, we introduce $\varphi:\RR_+\to\RR$, $t\mapsto \phi(t)-(d+\psi_\infty+\sigma\,t)$. Note that
\begin{align*}
s_f(\uu_{-\infty},\uu_r)&=\sigma\,,&
\d_x\left(s_f(\uu_{-\infty},\check{\uU}_+(\cdot))\right)(d)
&=\frac{[g(\uU)]_d}{[\uU]_d}\,.
\end{align*}
Combined with the asymptotic estimates on $u_-$ and $u_+$, this implies that for some $C''$, 
\[
\abs{\varphi(0)}\leq C'' \left(\Norm{\d_x\psi_0-1}_{L^{\infty}(\RR)}+\Norm{v_0}_{L^{\infty}(\RR)}\right)
\] 
and, for any $t\geq0$,
\begin{align*}
&\abs{\varphi(t)}
\leq \abs{\varphi(0)}\,e^{\frac{[g(\uU)]_d}{[\uU]_d}\,t}
+C''\,\int_0^t\,e^{\frac{[g(\uU)]_d}{[\uU]_d}\,(t-s)}\,(\varphi(s))^2\dd s\\
&\quad+C''\,
\left(\Norm{\d_x\psi_0-1}_{L^{\infty}(\RR)}+\Norm{v_0}_{W^{1,\infty}(\RR\setminus D)}\right)
\,\int_0^t\,e^{\frac{[g(\uU)]_d}{[\uU]_d}\,(t-s)}\,
e^{-s\,\min\left(\left\{g'(\uu_\star),-g'(\uu_{+\infty}),-g'(\uu_{-\infty})\right\}\right)}\dd s\,.
\end{align*}
Thus a continuity argument shows that, provided that $\Norm{\d_x\psi_0-1}_{L^{\infty}(\RR)}+\Norm{v_0}_{W^{1,\infty}(\RR\setminus D)}$ is sufficiently small, for some $C'''$ and any $t\geq0$,
\[
\abs{\varphi(t)}\leq 
C'''\,e^{-\theta\,t}\,
\left(\Norm{\d_x\psi_0-1}_{L^{\infty}(\RR)}+\Norm{v_0}_{W^{1,\infty}(\RR\setminus D)}\right)\,.
\]
This implies a similar bound on $\varphi'$.

At this stage we only need to introduce $\psi$ to fit the claims in Theorem~\ref{T.mixed}. We pick $\chi:\RR\to\RR$ smooth and compactly supported, constant equal to $1$ in a neighborhood of $d$ and constant equal to $0$ in a neighborhood of $0$. Then we set 
\[
\psi\,:\ \RR_+\times\RR\to\RR\,,\quad
(t,x)\mapsto x+\psi_\infty+\sigma\,t+\chi(x)\,\varphi(t)\,.
\] 
For some constant $C_0$ and any $t\geq0$, 
\begin{align*}
\Norm{u(t,\psi(t,\cdot))-\uU}_{W^{1,\infty}((-\infty,d))}
\leq C_0\Norm{u_-(t,\cdot)-\uu_{-\infty}}_{W^{1,\infty}(\RR)}
\end{align*}
and
\begin{align*}
&\Norm{u(t,\psi(t,\cdot))-\uU}_{W^{1,\infty}((d,+\infty))}\\
&\qquad\leq \Norm{u_+(t,\psi(t,\cdot))-\check{\uU}_+(\varphi(t)\,\chi(\cdot))}_{W^{1,\infty}((d,+\infty))}
+\Norm{\check{\uU}_+(\varphi(t)\,\chi(\cdot))-\check{\uU}_+}_{W^{1,\infty}((d,+\infty))}\\
&\qquad\leq C_0\Norm{u_+(t,\cdot+\sigma\,t+\psi_\infty)-\check{\uU}_+}_{W^{1,\infty}(\RR)}
+C_0\,\abs{\varphi(t)}\,.
\end{align*}
This achieves the proof.
\end{proof}

To conclude the proof of Theorem~\ref{T.classification}, it only remains to observe that spectral stability of stable constant states and of Riemann shocks as in Assumption~\ref{A.Riemann} is also proved in~\cite{DR1} and to show spectral stability of waves given by Assumptions~\ref{A.mixed-single},~\ref{A.mixed-single-bis} and~\ref{A.mixed-double}. The latter follows from the spectral stability of stable constant states and stable continuous fronts through the extension-patching argument as in the nonlinear stability proofs. The details are left to the reader and the reader is referred to Section~4 for some hints.

It is relatively straightforward to derive counterparts of Theorem~\ref{T.mixed} including perturbations with small shocks, higher-regularity descriptions or multidimensional perturbations supported away from characteristic points. The statement and proofs of those are left to the interested reader. Yet we give here a brief description of the dynamics of the small perturbing shocks. When perturbing a stable Riemann shock, the small shock merges in finite time with the background shock and thus somehow disappears. For waves as in Theorem~\ref{T.mixed}, small perturbing shocks either merge in finite time with a background discontinuity if there is any on the same side of the characteristic point of the reference wave or moves towards infinity as in Proposition~\ref{P.small-shock} if there is none.

\subsection{Stable waves with multiple characteristic points}\label{S.multiple-characteristics}

In the present subsection we prove that for non-degenerate piecewise regular traveling waves with a finite number of shocks the identification of instability mechanisms in Section~\ref{S.instabilities} is indeed comprehensive. Thus we relax Assumption~\ref{A.generic} and prove nonlinear stability for some waves possessing several characteristic points.

The main difference with analysis of the foregoing subsection is that such waves are not isolated, even if waves coinciding up to a spatial translation are identified. Indeed they come as elements of continuously parameterized families of waves. As a consequence each wave is not asymptotically stable by itself but these families are, in the sense that a solution arising from the perturbation of one such wave converges to a possibly different element of the same wave family.

The following assumption formalizes the class of waves we consider.

\begin{Assumption}\label{A.general-stable}
Consider a non-degenerate piecewise regular entropy-admissible traveling-wave solution to~\eqref{eq-u} defined by $(\uU,\sigma,D)$. Assume that
\begin{enumerate}
\item $D$ is finite and its limits, $\uu_{+\infty}=\lim_{+\infty}\uU$ and $\uu_{-\infty}=\lim_{-\infty}\uU$ satisfy
\begin{align*}
g'(\uu_{+\infty})&<0\,,&
g'(\uu_{-\infty})&<0\,;
\end{align*}
\item each characteristic value $\uu_\star$ taken by $\uU$ satisfies $g'(\uu_\star)>0$;
\item for any $d\in D$, either, near $d$, $\uU$ is constant on both sides or
\[\frac{[\,g(\uU)\,]_d}{[\,\uU\,]_d}\,<\,0\,.\]
\end{enumerate}
\end{Assumption}

A few remarks are in order. 
\begin{enumerate}
\item By Proposition~\ref{P.structure}, the second condition could alternatively be stated as: each connected component of $\RR\setminus D$ contains at most one characteristic point. 
\item Since $\uU$ is strictly monotonic on bounded connected components, in the third condition the first part of the alternative happens only for Riemann shocks. 
\item Proceeding as in the element of proof below Theorem~\ref{T.classification}, one infers that if $\uU$ has at least two characteristic points, then it passes through at least two different  characteristic values.
\item The cases when $\uU$ has less than two characteristic points are already covered by the results stated in the foregoing subsection and Section~\ref{S.continous_fronts}. 
\end{enumerate}

The following proposition proves that when there are at least two characteristic points there is nearby a family of similar waves.
\begin{Proposition}\label{P.family}
Let $(\uU,\sigma,D)$ define a wave satisfying Assumption~\ref{A.general-stable} with at least two characteristic points. Label the characteristic points as $\upsi_{1,\star}<\cdots <\upsi_{n,\star}$, with $n\in\NN$, $n\geq2$ and the elements of $D$ as $d_{1-\gamma_-}<\cdots<d_{n-1+\gamma_+}$ with $(\gamma_-,\gamma_+)\in\{0,1\}^2$, $\upsi_{1,\star}<d_1$, $\upsi_{n,\star}>d_{n-1}$.\\
There exist $\eps_0>0$ and $C>0$ such that for any $\Psi_\star=(\psi_{1,\star},\cdots,\psi_{n,\star})$ such that
\[
\Norm{(\upsi_{j,\star}-\upsi_{1,\star})_{2\leq j\leq n}
-(\psi_{j,\star}-\psi_{1,\star})_{2\leq j\leq n}}\leq \eps_0\,,
\]
there exist a unique $(\uU^{\Psi_\star},D^{\Psi_\star})$ such that 
\begin{enumerate}
\item $(\uU^{\Psi_\star},\sigma,D^{\Psi_\star})$ defines a wave satisfying Assumption~\ref{A.general-stable};
\item for any $1\leq j\leq n$, $\uU^{\Psi_\star}(\psi_{j,\star})=\uU(\upsi_{j,\star})$;
\item $D^{\Psi_\star}$ has the same cardinal as $D$ and, labeling its elements as $d^{\Psi_\star}_{1-\gamma_-}<\cdots<d^{\Psi_\star}_{n-1+\gamma_+}$, then $\psi_{1,\star}<d^{\Psi_\star}_1$, $\psi_{n,\star}>d^{\Psi_\star}_{n-1}$ and, for any $1-\gamma_-\leq k\leq n-1+\gamma_+$,
\[
\abs{d^{\Psi_\star}_k-(d_k+\psi_{1,\star}-\upsi_{1,\star})}\,\leq\,C\,
\Norm{(\upsi_{j,\star}-\upsi_{1,\star})_{2\leq j\leq n}
-(\psi_{j,\star}-\psi_{1,\star})_{2\leq j\leq n}}\,;
\]
\item there exists a $\cC^\infty$ maps $\psi^{\Psi_\star}:\RR\to\RR$ such that 
\begin{align*}
\Norm{\psi^{\Psi_\star}-(\,\cdot\,+\psi_{1,\star}-\upsi_{1,\star})}_{W^{1,\infty}(\RR)}
&+\Norm{\uU^{\Psi_\star}(\psi^{\Psi_\star}(\,\cdot\,))-\uU}_{W^{1,\infty}(\RR\setminus D)}\\
&\leq C\,\Norm{(\upsi_{j,\star}-\upsi_{1,\star})_{2\leq j\leq n}
-(\psi_{j,\star}-\psi_{1,\star})_{2\leq j\leq n}}\,.
\end{align*}
\end{enumerate}
Moreover, increasing $C$ if necessary, one may enforce that $\psi^{\Psi_\star}-(\cdot+\psi_{1,\star}-\upsi_{1,\star})$ is supported in any prescribed neighborhood of $D$.
\end{Proposition}

\begin{proof}
To begin with, let $(\uU_0,\cdots,\uU_n)$ denote extensions respectively of $\uU_{|(-\infty,d_1)}$ to a neighborhood of $(-\infty,d_1]$, of $\uU_{|(d_j,d_{j+1})}$ to a neighborhood of $[d_j,d_{j+1}]$ for $1\leq j\leq n-1$, and of $\uU_{|(d_n,+\infty)}$ to a neighborhood of $[d_n,+\infty)$,  obtained by solving the ODE associated with the profile equations. Note that $\uU_0$ (resp. $\uU_n$) contains a discontinuity if $\gamma_-=1$ (resp. $\gamma_+=1$).

For any $\Psi_\star$ satisfying the smallness condition of the proposition, we shall define $\uU^{\Psi_\star}$ as 
\[
\uU^{\Psi_\star}(x)=
\begin{cases} 
\uU_0(x+\upsi_{1,\star}-\psi_{1,\star}) & \text{ if } x<d^{\Psi_\star}_1\,,\\
\uU_k(x+\upsi_{k,\star}-\psi_{k,\star}) & \text{ if } d^{\Psi_\star}_k<x<d^{\Psi_\star}_{k+1}\,,\ 1\leq k\leq n-1\,,\\
\uU_n(x+\upsi_{n,\star}-\psi_{n,\star}) & \text{ if } x>d^{\Psi_\star}_{n-1}\,,\\
\end{cases}
\]
with $(d^{\Psi_\star}_k)_{1\leq k\leq n-1}$ determined by the Rankine-Hugoniot conditions
\[
[f(\uU^{\Psi_\star})-\sigma\,\uU^{\Psi_\star}]_{d^{\Psi_\star}_k}\,=\,0\,,\qquad 1\leq k\leq n-1\,,
\]
and, if $\gamma_-=1$ (resp. $\gamma_+=1$), $d^{\Psi_\star}_0=d_0+\psi_{1,\star}-\upsi_{1,\star}$ (resp. $d^{\Psi_\star}_n=d_n+\psi_{n,\star}-\upsi_{n,\star}$). The existence of $(d^{\Psi_\star}_k)_{1\leq k\leq n-1}$ follows from the Implicit Function Theorem applied for $1\leq k\leq n-1$, to
\[
RH_k\,:\quad(\delta,\eta,\eta')\mapsto 
f(\uU_{k+1}(d_k+\delta+\eta'))-\sigma\,\uU_{k+1}(d_k+\delta+\eta')
-\left(f(\uU_k(d_k+\delta+\eta))-\sigma\,\uU_k(d_k+\delta+\eta)\right)
\]
so as to determine $\delta_k=d^{\Psi_\star}_k-(d_k+\psi_{1,\star}-\upsi_{1,\star})$ as a function of 
\[
(\eta_k,\eta'_k)=((\upsi_{k,\star}-\upsi_{1,\star})
-(\psi_{k,\star}-\psi_{1,\star}),(\upsi_{k+1,\star}-\upsi_{1,\star})
-(\psi_{k+1,\star}-\psi_{1,\star}))\,,
\] 
since 
\begin{align*}
RH_k(0,0,0)&=
[f(\uU)-\sigma\,\uU]_{d_k}\,=\,0\,,\\
\d_\delta(RH_k)(0,0,0)&=
[(f(\uU)-\sigma\,\uU)']_{d_k}\,=\,
[g(\uU)]_{d_k}\,\neq\,0\,.
\end{align*}
The smallness condition implies that $(\uU^{\Psi_\star},\sigma,D^{\Psi_\star})$ does define a wave of~\eqref{eq-u} satisfying Assumption~\ref{A.general-stable}.

We conclude essentially as in the proof of Theorem~\ref{T.mixed}. Independently of $\Psi_\star$, we pick $\chi:\RR\to\RR$ smooth and compactly supported, constant equal to $1$ in a neighborhood of $D$, and constant equal to $0$ in a neighborhood of $\{\,\upsi_{k,\star}\,;\,1\leq k\leq n\,\}$. Then, for any $\Psi_\star$, we define $\psi^{\Psi_\star}\,:\ \RR\to\RR$ by
\begin{align*}
\psi^{\Psi_\star}(x)\,=\,
x&+\psi_{1,\star}-\upsi_{1,\star}\\
&+
\begin{cases} 
0& \text{ if } x\leq\psi_{1,\star}\,,\\
\left(d^{\Psi_\star}_k-(d_k+\psi_{1,\star}-\upsi_{1,\star})\right)\,\chi(x)& \text{ if } \psi_{k,\star}<x\leq\psi_{k+1,\star}\,,\ 1\leq k\leq n-1\,,\\
\left(d^{\Psi_\star}_n-(d_n+\psi_{1,\star}-\upsi_{1,\star})\right)\,\chi(x)& \text{ if } x>\psi_{n,\star}\text{ and }\gamma_+=1\,,\\
0& \text{ if } x>\psi_{n,\star}\text{ and }\gamma_+=0\,.
\end{cases}
\end{align*} 
It is then straightforward to check the claimed estimates as in the proof of Theorem~\ref{T.mixed}.
\end{proof}

\begin{Theorem}\label{T.general-stable}
Let $(\uU,\sigma,D)$ define a wave satisfying Assumption~\ref{A.general-stable} with at least two characteristic points and use notation from Proposition~\ref{P.family}. Let $\theta>0$ such that
\begin{align*}
\theta&\geq\min\left(\left\{-g'(\uu_{+\infty}),-g'(\uu_{-\infty})\right\}\cup\left\{g'(\uU(\upsi_{k,\star})),;\,1\leq k\leq n-1\,\right\}\right)\,,\\
\theta&>\min\left(\left\{-\frac{[g(\uU)]_d}{[\uU]_d};\,d\in D\,\right\}\right)\,.
\end{align*}
There exist $\eps>0$ and $C>0$ such that for any $(\psi_0,v_0)\in BUC^1(\RR)\times BUC^1(\RR\setminus D)$ satisfying
\[
\Norm{\d_x\psi_0-1}_{L^{\infty}(\RR)}+\Norm{v_0}_{W^{1,\infty}(\RR\setminus D)}\leq \eps,
\]
the initial datum $u\id{t=0}=(\uU+v_0)\circ \psi_0^{-1}$ generates a unique global entropic solution to~\eqref{eq-u} and there exist $\psi\in BUC^2(\RR_+\times\RR)$ and $\Psi_\infty=(\psi_{k,\star})_{1\leq k\leq n}\in\RR^n$ such that 
\[
\abs{\Psi_\infty-(\psi_0(\upsi_{k,\star}))_{1\leq k\leq n}}\leq 
C\,\left(\Norm{\d_x\psi_0-1}_{L^{\infty}(\RR)}+\Norm{v_0}_{L^{\infty}(\RR)}\right)
\]
and for any $t\geq 0$
\begin{align*}
\psi(t,\psi_{k,\star})&=\psi_{k,\star}+\sigma t\,,\ 1\leq k\leq n\,,
&u(t,\psi_{k,\star}+\sigma t)&=\uU(\upsi_{k,\star})\,,\ 1\leq k\leq n\,,
\end{align*}
$u(t,\psi(t,\cdot))\in BUC^1(\RR\setminus D^{\Psi_\infty})$ and
\begin{align*}
\Norm{\d_x\psi(t,\cdot)-1}_{L^{\infty}(\RR)}
+\Norm{u(t,\psi(t,\cdot))-\uU^{\Psi_\infty}}_{W^{1,\infty}(\RR\setminus D^{\Psi_\infty})}
&\leq C\,e^{-\theta t}
\,\left(\Norm{\d_x\psi_0-1}_{L^{\infty}(\RR)}+\Norm{v_0}_{W^{1,\infty}(\RR\setminus D)}\right)\,,\\
\Norm{\d_t\psi(t,\cdot)-\sigma}_{L^{\infty}(\RR)}+\Norm{\psi(t,\cdot)-(\cdot+\sigma t)}_{L^{\infty}(\RR)}
&\leq C\,e^{-\theta t}
\,\left(\Norm{\d_x\psi_0-1}_{L^{\infty}(\RR)}+\Norm{v_0}_{W^{1,\infty}(\RR\setminus D)}\right)\,.
\end{align*}
Moreover, increasing $C$ if necessary, one may enforce that, for any $t\geq0$, $\psi(t,\cdot)-(\cdot+\sigma t)$ is supported in any prescribed neighborhood of $(\psi_0(\upsi_{1,\star})-\upsi_{1,\star})+D$.
\end{Theorem}

Since the proof of Theorem~\ref{T.general-stable} is nearly identical to the one of Theorem~\ref{T.mixed}, we leave it to the reader. Again, we point out that Theorem~\ref{T.general-stable} possesses counterparts including perturbations with small shocks, higher-regularity descriptions or multidimensional perturbations supported away from characteristic points. 

\newcommand{\etalchar}[1]{$^{#1}$}
\newcommand{\SortNoop}[1]{}

\end{document}